\documentclass[a4paper,12pt,oneside,reqno]{amsart}

\usepackage{enumitem}
\usepackage{hyperref}
\usepackage[headinclude,DIV13]{typearea}
\areaset{15.1cm}{25.0cm}
\usepackage{amsfonts,amssymb,amsmath,amsthm,bbm}
\usepackage{xcolor}
\usepackage[latin1] {inputenc}
\usepackage{graphicx, psfrag}
\usepackage{color}
\usepackage{bold-extra}

\usepackage{subscript}

\makeatletter
\@namedef{subjclassname@2020}{\textup{2020} Mathematics Subject Classification}
\makeatother

\newtheorem{theorem}{Theorem}[section]
\newtheorem{lemma}[theorem]{Lemma}
\newtheorem{proposition}[theorem]{Proposition}
\newtheorem{conjecture}[theorem]{Conjecture}
\newtheorem{corollary}[theorem]{Corollary}

\theoremstyle{definition}
\newtheorem{definition}[theorem]{Definition}

\theoremstyle{remark}
\newtheorem*{remark}{Remark}

\def\paragraph#1{\noindent \textbf{#1}}

\numberwithin{equation}{section}

 \newcommand{\be}{\begin{equation}}
 \newcommand{\ee}{\end{equation}}

\newcommand{\bea}{\begin{eqnarray}}
 \newcommand{\eea}{\end{eqnarray}}

\def\TH(#1){\label{#1}}
\def\thv(#1){\ref{#1}}
\def\Eq(#1){\label{#1}}
\def\eqv(#1){(\ref{#1})}
\def\sfrac#1#2{{\textstyle{#1\over #2}}}

\def \1{\mathbbm{1}}

\def\wh{\widehat}

\def\sign{\mathop{\rm sign}\nolimits}

\def\a{\alpha}%
\def\b{\beta}%
\def\d{\delta}%

\newcommand\g{\gamma}
\def\l{\lambda}%

\def\s{\sigma}

\def\L{\Lambda}
\def\G{\Gamma}
\def\O{\Omega}
\def\S{\Sigma}

\def\R{{\Bbb R}}
\def\N{{\Bbb N}} %
\def\P{{\Bbb P}}  
\def\Z{{\Bbb Z}}

\def\E{{\Bbb E}}

\let\cal=\mathcal
\def\AA{{\cal A}}
\def\BB{{\cal B}}

\def\EE{{\cal E}}
\def\FF{{\cal F}}

\def\MM{{\cal M}}

\def\OO{{\cal O}}

\def\SS{{\cal S}}
\def\TT{{\cal T}}
\def\VV{{\cal V}}
\def\UU{{\cal U}}
\def\VV{{\cal V}}

\def\XX{{\cal X}}
\def\YY{{\cal Y}}
\def\ZZ{{\cal Z}}

\begin{document}

\title[Mixed memories in Hopfield networks]{Mixed memories in Hopfield networks}

\author[V. Gayrard]{V\'eronique Gayrard}
 \address{
V. Gayrard\\ Aix Marseille Univ, CNRS, I2M, Marseille, France
}
\email{veronique.gayrard@math.cnrs.fr}

\subjclass[2020]{60G50, 82D30, 68T07, 68T10} 
\keywords{Hopfield model, dense Hopfield model, modern Hopfield model, spurious patterns, spurious states, memory capacity, exact retrieval.}
\date{\today}

\begin{abstract} 
We consider the class of Hopfield models of associative memory with activation function $F$ and state space $\{-1,1\}^N$, where each vertex of the cube describes a configuration of $N$ binary neurons. $M$ randomly chosen configurations, called patterns, are stored using an energy function designed to make them local minima. If they are, which is known to depend on how $M$ scales with $N$, then they can  be retrieved using a dynamics that decreases the energy.  However, storing the patterns in the energy function also creates unintended local minima, and thus false memories. Although this has been known since the earliest work on the subject, it has only been supported by numerical simulations and non-rigorous calculations, except in elementary cases.

Our results are twofold. For a generic function $F$, we explicitly construct a set of configurations, called mixed memories,  whose properties are intended to characterise the local minima of the energy function.  For three prominent models, namely the classical, the dense and the modern Hopfield models, obtained for quadratic, polynomial and exponential functions $F$ respectively, we give conditions on the growth rate of $M$ which guarantee that, as $N$ diverges, mixed memories are fixed points of the retrieval dynamics and thus minima of the energy. We conjecture that in this regime, all local minima are mixed memories.
\end{abstract}

\thanks{V.~Gayrard would like to thank the Institute for Applied Mathematics at the University of Bonn for their kind hospitality during the writing of this paper. Funding for her stay was provided by the Deutsche Forschungsgemeinschaft (DFG, German Research Foundation) under Germany's Excellence Strategy -- EXC-2047/1 -- 390685813.
}

\maketitle


\section{Introduction}
    \TH(S1)

\subsection{Hopfield models} 
    \TH(S1.1)     
    
Written in the mathematical framework of statistical mechanics, Hopfield networks are a family of dynamical models of associative memory that find their origin in the Hebbian theory of learning. Pioneered over forty years ago \cite{Hop82}, \cite{Hop84} they have had a profound influence on several scientific disciplines, from physics to contemporary machine learning, as has recently been recognised \cite{NPP24}.

In their simplest form, such models consist of $N$ binary neurons taking values in $\{-1,1\}$. Possible  \emph{configurations} of the memory then are vertices $\s=(\s_i)_{1\leq i\leq N}\in\S_N$ of the $N$-dimensional discrete cube  $\S_N=\{-1,1\}^N$, and the objects to be memorised are $M$ specific configurations $\xi^1,\dots, \xi^M$ in $\S_N$, called \emph{patterns}.
Given a smooth function $F: \R\rightarrow \R$,  the patterns are stored through an \emph{energy function} $E_{N,M}$, defined on $\S_N$ by
\be
E_{N,M}(\s)=-\sum_{\mu=1}^M F\left(\frac{1}{N} \sum_{i=1}^N\xi^{\mu}_i\s_i\right), \quad \s\in\S_N.
\Eq(1.1.2)
\ee
Thus $E_{N,M}(\s)$ depends on the patterns only through their \emph{overlap} with $\s$, $N^{-1}\sum_{i}\xi^{\mu}_i\s_i$.
The guiding idea  is to choose the \emph{activation function} $F$ in such a way that each pattern lies in a deep (ideally global) minimum of an energy valley of $E_{N,M}$ and is surrounded by high energy barriers. An associative memory is then obtained by setting up a dynamics that decreases the energy so that, starting from a configuration that resembles a given pattern,
the dynamics converges to that pattern. One cannot hope that such properties will hold for every possible choice of patterns, but rather for a choice that is considered typical. This can be achieved by selecting the patterns at random. For this purpose,  we let $(\O,\FF,\P)$ be a probability space  on which is defined a doubly infinite sequence 
$
(\xi^{\mu}_i)_{i\in\N, \mu\in\N}
$ 
of  independent and identically distributed Bernoulli random variables, which satisfy
\be
\P\left(\xi^{\mu}_i=1\right)=1-\P\left(\xi^{\mu}_i=-1\right)=1/2.
\Eq(1.1.2')
\ee
The patterns are thus uncorrelated and unbiased.
Note that the energy function $E_{N,M}$ now is  a random function defined on $(\O,\FF,\P)$.

Several choices of the function $F$ have been considered in the literature. In the 1982 paper introducing the now \emph{classical} Hopfield model  \cite{Hop82}, $F(x)=\frac{1}{2}x^2$. This  choice was soon generalised to any polynomial of degree $p$ in analogy to the $p$-spin models of statistical mechanics \cite{PN86}, \cite{N88}. Three decades later, a series of papers \cite{KH18}, \cite{KH16} brought the case $F(x)=\frac{1}{p}x^p$ back into the limelight under the name of  \emph{dense} Hopfield model. Concomitantly, a \emph{modern}  Hopfield model was proposed in   \cite{DHLUV17}, where  $F(x)=e^{Nx}$.

Each of these models is effectively a whole class of models parametrised by $M$.  An issue of great practical and theoretical importance is that of \emph{memory (or storage) capacity}, \emph{i.e.}~the maximum number $M$ of patterns that a model can store and reliably retrieve.  Using computer simulations, it was found in \cite{Hop82}  that the memory capacity of the classical Hopfield model scales linearly with the number of neurons, $N$, and that memory is completely lost beyond  $\sim 0.15 N$.
Recent interest in dense and modern models stems from the fact that their memory capacity grows much faster, polynomially and exponentially in $N$ respectively. We return to this topic in Section \thv(S1.2.2).

This intriguing transition, together with the strong similarity between these models and mean-field models of disordered systems has sparked great interest among theoretical physicists and mathematicians working in statistical mechanics, \cite{MPV87}, \cite{Far23}, \cite{BP-book}, \cite{Ta1}, \cite{Ta2} (see also the references therein). In this context, the key object of interest is the Gibbs measure associated to $E_{N,M}$. This is the random probability measure defined on $\S_N$ by
\be
G_{\b,N}(\s)=\frac{1}{Z_{\b,N}}e^{\b NE_{N,M}(\s)}, \quad \s\in\S_N,
\nonumber
\ee
where  $\b>0$ is a parameter that physically represents the inverse of a temperature and 
$Z_{\b,n}=\sum_{\s\in\S_N}e^{\b NE_{N,M}(\s)}$. An important first step towards understanding the asymptotic properties of $G_{\b,N}$ for large $N$  is to compute the limiting free energy
\be
f_{\b}=\lim_{N\rightarrow\infty}\frac{1}{\b N}\log Z_{\b,N}.
\nonumber
\ee
In a seminal paper, Amit \emph{et al.}~\cite{AGS85b}  obtained a complete picture of the phase diagram of the classical Hopfield model at all temperatures using the non rigorous replica trick to compute the free energy. The loss of memory occurs at $M\sim 0.138 N$ and is linked to the appearance of a spin glass phase of the same nature as that found by Parisi  \cite{P79} in the paradigmatic Sherrington and Kirkpatrick (SK) model of a mean-field spin glass  \cite{SK}. Several aspects of their results have been understood with mathematical rigour (see \cite{BGP-PTRF94}, \cite{BG-PTRF97}, \cite{BG98}, \cite{TaHop98}, Chap.~4 in \cite{Ta1} and Chap.~10 in \cite{Ta2}  for the most advanced results, and references therein).  All these results concern  the static properties of the models and are governed, more or less explicitly, by the property that the global minima of the energy $E_{N,M}$ lie at or near the patterns. 

Clearly, the effective functioning of an associative memory depends not only on the properties of the deepest minima of its energy landscape, but more generally on the entire structure of its critical points, with local minima being fixed points of deterministic gradient descent dynamics or giving rise to metastable states in random dynamics \cite{BEGK-PTRF01}. In particular, it has been known since the first work on the subject  that the process of storing the patterns in the energy function also creates spurious, unintended memories that overlap with multiple patterns and can be retrieved by the dynamics just like the patterns themselves \cite{Hop83}. Several strategies have been developed to ``unlearn" \cite{Hop83}, \cite{Unlearn87}, \cite{BBF18} (see also the references therein) or to  mitigate the effects of these spurious memories \cite{KH16}, \cite{KH18}, mostly by modifying the energy function. However, except for the simplest cases \cite{KP89}, \cite{BR90}, the literature offers only an empirical understanding of these extraneous memories based on numerical simulations  and non-rigorous calculations \cite{AGS85a}.

In this paper, we explicitly construct a class of spurious memories of the  energy function \eqv(1.1.2) arising from mixtures of the initial patterns. These so-called \emph{mixed memories} are configurations of the form
\be
\xi_i(m)=\sign\left(\sum_{\mu=1}^{M}\xi^{\mu}_i F'(m_\mu)\1_{m_\mu\neq 0}\right), \quad 1\leq i\leq N,
\Eq(1.1.3)
\ee
where the \emph{mixture coefficients},  $m=(m_{\mu})_{1\leq \mu\leq M}\in [-1,1]^{M}$, form a deterministic vector with finitely many non-zero components,  $n\in\N$, and where $\1_{m_\mu\neq 0}=1$ if $m_\mu\neq 0$ and is zero else.
Furthermore, with a $\P$-probability that tends to one as $N$ diverges, for all $1\leq \mu\leq M$  such that 
$m_{\mu}\neq 0$, mixed memories have overlap $m_{\mu}$ with the pattern $\xi^{\mu}$,
\be
\lim_{N\rightarrow\infty} N^{-1}\sum_{i=1}^N\xi_i(m)\xi^{\mu}_i=m_{\mu}.
\Eq(1.1.4)
\ee
Our results are twofold: \hfill\break\null
\hskip.35truecm (i) First, we construct solutions to the system of equations defined through \eqv(1.1.3) and \eqv(1.1.4), whose unknowns are the $n$ non-zero components of $m$. More precisely, we construct a class of 
admissible mixture coefficients, $\MM_n^{all}$, which does not depend on $F$ and is thus common to  all models of Hopfield type. The solutions of  \eqv(1.1.4) for a given energy function \eqv(1.1.2) are then obtained as the subset $\MM_{n,F}\subset \MM_n^{all}$ of the mixture coefficients  that satisfy a particular system of inequalities, $\SS_{n,F}$, that depend on $F$.
\hfill\break\null
\hskip.35truecm (ii) We then give conditions on the growth rate of $M$ as a function of $N$ which guarantee that mixed memories
$(\xi_i(m))_{1\leq i\leq N}$ with $m\in\MM_{n,F}$ are local minima of the energy function \eqv(1.1.2) for the main models of interest, \emph{i.e.}~the classical, the dense and the modern Hopfield models. We ask two questions: whether this is true for each mixed memory or for all of them simultaneously, with $\P$-probability one as $N$ diverges.

The set $\MM_{n,F}$ contains all the local minima found numerically. This supports the conjecture that we have obtained the complete set of all local minima.

The existence of a common set $\MM_n^{all}$ of mixture coefficients from which mixed memories are constructed appears to confirm recent numerical findings by Hopfield \emph{et al.}~\cite{KH18}, \cite{KH16}, that spurious states can be ``transported'' within the class of dense models, from one model to another. 

These results provide a first insight into mixed memories in Hopfield models with analogue (or continuous) neurons. Indeed, it has been known since \cite{Hop84} that there is a correspondence between the sets of local minima of the continuous and binary neuron models, with the two sets coinciding in the so-called high-gain limit, an analogue of the zero-temperature limit in statistical mechanics.  This is of practical relevance  both in computer science, where Hopfield models are commonly integrated into deep learning architectures \cite{K23}, \cite{AllYouNeed}, and in statistical mechanics.

\subsection{Main results} 
    \TH(S1.2)   
    
In the classical Hopfield model, the existence of local minima corresponding to mixtures of a finite number of original patterns  was fully formalised by Amit \emph{et al.}~\cite{AGS85a}  through an extensive numerical study of the critical points of the free energy associated with the model at all temperatures. They discovered that these mixed memories are not random, but are given by well-defined, deterministic mixtures of patterns, and  classified them into three main groups: \emph{symmetric},  \emph{continuous asymmetric} and  \emph{discontinuous asymmetric} memories (referring to the way they emerge when the temperature is varied). The class of continuous asymmetric memories appears to be the largest. Together with the symmetric memories, it is expected to contain all local minima.  In contrast, the discontinuous asymmetric memories found were much rarer, with none being a local minimum.  Symmetric memories were mathematically understood a few years later \cite{KP89, BR90}.  In this paper, we construct a class of mixed memories  that encompasses  all examples of the symmetric and continuous asymmetric memories obtained in \cite{AGS85a}  at zero temperature.   We simply call this class \emph{mixed memories}.

\subsubsection{Mixed memories} 
   \TH(S1.2.1)      
   
Throughout the paper, $n\in\N$ is chosen to be independent of $N$ and $M$ is chosen to be a non-decreasing function of $N$. We begin with a formal definition of mixed memories for energy functions of the form  \eqv(1.1.2).

\begin{definition}[Mixed memories of type $F$]
    \TH(1.def1.2)
Let $F$ be a smooth function whose derivative satisfies $F'(x)>0$, for all $x>0$. Given $n\in\N$ independent of $N$, $n$-mixed memories of type $F$ are configurations in $\S_N$ denoted by $\xi^{(N)}(m)=\left(\xi_i(m)\right)_{1\leq i\leq N}$ and defined as
\be
\xi_i(m)=\sign\left(\sum_{\mu=1}^{M}\xi^{\mu}_i F'(m_\mu)\1_{m_\mu\neq 0}\right), \quad 1\leq i\leq N,
\Eq(1.2.1.11)
\ee
where $m=(m_{\mu})_{1\leq \mu\leq M}\in [-1,1]^{M}$  is a deterministic vector with the following properties:
\begin{itemize}
\item[(i)] $m$ has exactly $n$ non-zero components, \emph{i.e.}~there exists a subset $V\subset\{1,\dots,M\}$ of cardinality $|V|=n$ such that $m_{\mu}\neq 0$ if and only if $\mu\in V$.

\item[(ii)] 
Let $\{\mu_1,\dots,\mu_n\}$ be an enumeration of the elements of $V$ and,  for each $1\leq \nu\leq n$, set $m_{\mu_\nu}=\hat m_{\nu}$.
Then, for each $1\leq \nu\leq n$, the normalised overlap of $\xi^{(N)}(m)$  with the pattern $\xi^{\mu_{\nu}}$ converges to $\hat m_{\nu}$ as $N$ diverges,
\be
\P\left(\lim_{N\rightarrow\infty} N^{-1}\left(\xi^{(N)}(m),\xi^{\mu_{\nu}}\right)=\hat m_{\nu}\right)=1, \quad \forall1\leq \nu\leq n,
\Eq(1.2.1.2)
\ee
and it converges to zero else,
\be
\P\left(\lim_{N\rightarrow\infty} N^{-1}\left(\xi^{(N)}(m),\xi^{\mu}\right)=0\right)=1, \quad \forall\mu\in \{1,\dots,M\}\setminus V.
\Eq(1.2.1.3)
\ee
\end{itemize}
\end{definition}

\begin{remark}  Since $n$  is  independent of $N$, the set $V$ is countable and \eqv(1.2.1.2) implies
\be
\P\left(\bigcap_{1\leq \nu\leq n}\left\{\lim_{N\rightarrow\infty} N^{-1}\left(\xi^{(N)}(m),\xi^{\mu_{\nu}}\right)=\hat m_{\nu}\right\}\right)=1.
\Eq(1.theo1.mix8')
\ee
\end{remark}

\begin{remark}
The techniques developed in this paper \emph{a priori} allow us to treat cases where $n$ is a growing function of $N$,  albeit with the important limitation that this growth is logarithmic at most. See also the remark at the end of the proof of Proposition 3.1.
\end{remark}

Let us now describe the set that we will prove to be a set of $n$-mixed memories. Recall that a \emph{composition} of the integer $n$ into $\ell$ summands, or $\ell$-composition, is any solution $(n_1,\dots,n_{\ell})$ of $n=n_1+\dots+n_{\ell}$
with $n_k\geq 1$ for each $1\leq k\leq \ell$. Note that unlike integer partitions, the order of the summands counts. We call an $\ell$-composition $(n_1,\dots,n_{\ell})$ \emph{allowable} if $n_k\geq 2$ is even for all  $1\leq k\leq \ell-1$ and $n_{\ell}\geq 1$ is odd.  Thus, there is no allowable composition of an even integer $n$. Given an allowable $\ell$-composition set
\be
\a^{(n_{k})}=2^{-n_{k}+1}{{n_{k}-1}\choose{\lfloor(n_{k}-1)/2\rfloor}}
\Eq(1.2.1.4)
\ee
for each $1\leq k\leq \ell$, and let $\g_n=\left(\g^{(k)}\right)_{1\leq k\leq \ell}$ be the vector of components
\be
\g^{(k)}
\equiv \prod_{l=1}^{k}\a^{(n_{l})}.
\Eq(1.2.1.5)
\ee
Denote by $\G^{all}_n\subset [-1,1]^{n}$ the set of all such vectors, one for each allowable $\ell$-composition
\be
\G^{all}_n=\bigcup_{1\leq \ell\leq n}\left\{\g_n \mid 
(n_1,\dots,n_{\ell}) \text{ is an allowable $\ell$-composition of $n$}\right\},
\Eq(1.2.1.6)
\ee
and  define the subset 
\be
\G_{n,F}=\left\{\g_n\in\G^{all}_n : \g_n \text{ is a solution of }  \SS_{n,F} \right\},
\Eq(1.2.1.12)
\ee
where  $\SS_{n,F}$ is the system of $\ell -1$ inequalities
\be
2F'\left(\g^{(k)}\right)>n_{k+1}F'\left(\g^{(k+1)}\right)+\dots+n_{\ell}F'\left(\g^{(\ell)}\right) \,\text{ for all } 1\leq k\leq \ell-1. 
\Eq(1.2.1.12bis)
\ee
Given $\g_n\in\G_{n,F}$, let $m(\g_n)=(m_{\mu}(\g_n))_{1\leq \mu\leq M}$ be the vector whose components are constant and equal to $\g^{(k)}$ on consecutive blocks of length $n_k$,  $1\leq k\leq \ell$, and are $0$ beyond,
\be
m_{\mu}(\g_n) = 
\begin{cases}
\g^{(k)} & \text{if\, $1\leq \mu-(n_0+\dots+n_{k-1})\leq n_k$ \,for some $1\leq k\leq \ell$,} \\
0 &  \text{if\, $\mu>n$,}
\end{cases}
\Eq(1.2.1.8)
\ee
where by convention $n_0\equiv 0$. With this, define 
\be
\MM^{\circ}_{n,F}=\left\{m(\g_n) : \g_n\in\G_{n,F}\right\}.
\Eq(1.2.1.14)
\ee
Finally, the above set is extended to include all possible permutations and (depending on the parity of $F$) all possible signs of the coordinates of each of its elements. More precisely, setting $a=1$ if $F'$ is an odd function and $a=2$ otherwise,
\be
\MM_{n,F}=\bigcup_{\g_n\in\G_{n,F}}\bigcup_{\pi\in\Pi_M}\bigcup_{\varepsilon\in\{-1,1\}^M}
\left\{m' : m'_{\mu}= \left(\varepsilon_{\mu}\right)^a m_{\pi(\mu)}(\g_n), {1\leq \mu\leq M}\right\},
\Eq(1.2.1.15)
\ee
where $\Pi_M$ denotes the set of all permutations of $\{1,\dots,M\}$ and 
$\varepsilon=(\varepsilon_\mu)_{1\leq \mu\leq M}\in \{-1,1\}^M$ is a sequence of signs. 
Therefore,  if $F'$ is not an odd function, the coordinates of the vectors $m$ in $\MM_{n,F}$ are all non-negative.  
The statement made in point (i) of Section \thv(S1.1) can now be put into concrete form.
Define 
\be
\MM^{all}_{n}=\bigcup_{\g_n\in\G^{all}_n}\bigcup_{\pi\in\Pi_M}\bigcup_{\varepsilon\in\{-1,1\}^M}
\left\{m' : m'_{\mu}= \varepsilon_{\mu} m_{\pi(\mu)}(\g_n), {1\leq \mu\leq M}\right\}.
\ee
Then, for each $F$, $\MM_{n,F}$ is the subset of $\MM^{all}_{n}$  constructed from the sequences $\g_n\in\G_{n,F}\subseteq\G^{all}_n$ satisfying the system of inequalities $\SS_{n,F}$ and restricted to the positive orthant in $\R^n$ unless $F'$ is odd.

We are now ready to state our first theorem. 

\begin{theorem}[Mixed memories of type $F$]
    \TH(1.theo2.mix)
Let $F$ be a smooth function whose derivative $F'$ satisfies $F'(x)>0$, for all $x>0$. 
For any odd $n\in\N$, if $m\in\MM_{n,F}$ then $\xi^{(N)}(m)$ is an $n$-mixed memory of type $F$.
\end{theorem}

The set $\MM_{n,F}$ decomposes into the disjoint union of two subsets containing  the symmetric and asymmetric mixed memories, respectively.  We call symmetric the memories resulting from the trivial $\ell$-compositions given by $\ell=1$ and $n_{1}=n$. In this case, \eqv(1.2.1.8) becomes $m(\g_n)=(\g^{(1)},\dots,\g^{(1)}, 0,\dots, 0)$, where $\g^{(1)}$ is repeated $n$ times.  Note that this set contains the $M$ original  patterns themselves: these are obtained for $n=1$.
All other memories are called asymmetric.

It is clearly of interest to know how many of these mixed  memories are present in a given model.  The next proposition gives estimates of their growth rate in $M$.

\begin{proposition}[Bounds on the number of mixed memories]
    \TH(1.prop.mix)
Let $F$ be a smooth function whose derivative $F'$ satisfies $F'(x)>0$, for all $x>0$. 
For each $n\in\N$ odd
\be
A_{n,F}M^n\left(1-\frac{n}{M}\right)^{n-1}
\leq 
\left|\MM_{n,F}\right|
\leq 
A_{n,F}M^n,
\Eq(1.prop.mix.1)
\ee
where $A_{n,F}$ only depends on $n$ and $F$ and not on $N$.
\end{proposition}

We conclude this subsection by noting that the numerical work of Amit \emph{et al.}~\cite{AGS85a} for the classical Hopfield model strongly suggests that if $\xi^{(N)}(m)$ is an $n$-mixed memory of type $F(x)=\frac{1}{2}x^2$, then $m\in\MM_{n,F}$. Extrapolating to the case of general type $F$ memories suggests that the converse of Theorem \thv(1.theo2.mix) is true. This leads to the following conjecture.

\begin{conjecture}
    \TH(1.conj.2)
Given any $n\in\N$ odd, $\xi^{(N)}(m)$ is an $n$-mixed memory of type $F$ if and only if $m\in\MM_{n,F}$.
\end{conjecture}

Implications of this conjecture are discussed at the end of this section.

In Section \thv(S5), we solve the system of inequalities $\SS_{n,F}$ defined in \eqv(1.2.1.12bis), 
and explicitly enumerate the set of $n$-mixed memories of type $F$ for the classical and the dense Hopfield models. We then draw consequences for their energy functions \eqv(1.1.2). 
In particular, we prove (see Lemma \thv(5.lem2)) that, for $F(x)=\frac{1}{2}x^2$, if $M\ll N$, then, with $\P$-probability one,  the energy of the corresponding $n$-mixed memories becomes confined  to the strip
$(-\frac{1}{2}, -\frac{1}{\pi}]$ as $N\rightarrow\infty$ if $n>1$, while the energy of the original patterns converges to $-\frac{1}{2}$.

\subsubsection{Mixed memories are local minima} 
    \TH(S1.2.2) 
    
We are interested in finding the local minima of $E_{N,M}$. This can be achieved by setting up a deterministic dynamics $\s(t)$ on  $\S_N$ as follows: at time step $t+1$, an index $i$ is selected at random from $\{1,\dots,N\}$ (\emph{e.g.}~uniformly and independently from previous selections), and $\s(t)$ is updated to
\be
\s_j(t+1) =
\begin{cases}
T_i(\s(t))& \text{if $j=i$,} \\
\s_j(t) &  \text{if $j\neq i,$}
\end{cases}
\Eq(1.2.2)
\ee
where $T=(T_i)_{1\leq i\leq N}: \S_N\rightarrow\S_N$ is some map that decreases the energy.
Then $\s$ is a minimum of $E_{N,M}(\s)$ if and only if it is a fixed point of this dynamics, \emph{i.e.}~if and only if
\be
\s_i=T_i(\s)\quad\forall 1\leq i\leq N.
\Eq(1.2.3)
\ee
Two types of maps $T_i$ (or update rules) have been considered in the literature: the so-called gradient map,
\be
T^{\text{G}}_i(\s)=\sign\left\{\sum_{\mu=1}^M \xi^{\mu}_iF' \left(\frac{1}{N}\sum_{1\leq j\neq i\leq N}\xi^{\mu}_j\s_j\right)\right\},
\Eq(1.2.1)
\ee
and, more recently,  the Hopfield-Krotov map \cite{KH16}.
\be
T^{\text{HK}}_i(\s)=
\sign\left\{\sum_{\mu=1}^M\left[F\left(\frac{\xi^{\mu}_i}{N}+\frac{1}{N}\sum_{1\leq j\neq i\leq N}\xi^{\mu}_j\s_j\right)-F\left(-\frac{\xi^{\mu}_i}{N}+\frac{1}{N}\sum_{1\leq j\neq i\leq N}\xi^{\mu}_j\s_j\right)\right]\right\}.
\Eq(1.2.4)
\ee
It is clear that the dynamics $T^{\text{HK}}$ decreases the energy \eqv(1.1.2) and therefore its fixed points are local minima. 
$T^{\text{G}}$ was Hopfield's original choice \cite{Hop82}.  In statistical mechanics, a variant is considered in which the sum \eqv(1.2.1) is over all $1\leq j\leq N$. It is obtained by taking the zero temperature limit of a gradient dynamics minimising the free energy functional associated with the energy \eqv(1.1.2). Although there is little difference between this definition and $T^{\text{G}}$, we cannot claim that $T^{\text{G}}$ decreases the energy for all $F$. One of the motivations for studying both $T^{\text{HK}}$ and $T^{\text{G}}$ is to prove that for the main models of interest, \emph{i.e.}~the classical, the dense and the modern models, their fixed points coincide. As we will see, the question is only open for the dense model, since it turns out that $T^{\text{G}}=T^{\text{HK}}$ for the classical and modern Hopfield models.

The theorems below give sufficient conditions on the growth rate of $M\equiv M(N)$ for mixed memories to be  local minima of the energy function, asymptotically, for three models: the classical, the dense and the modern Hopfield models. 
Having stated these results, the obtained conditions are then compared to the known conditions for original patterns to be local minima.

Remember that $n$ is independent of $N$. Furthermore, a vector $m\in \MM_{n,F}$ is an $M$-dimensional vector with $n$ non-zero components (see \eqv(1.2.1.8) and \eqv(1.2.1.15)).

\begin{theorem}[Classical Hopfield network]
    \TH(1.theo1)
Take $F(x)=\frac{1}{2}x^2$ in  \eqv(1.1.2). For this model $T^{\text{HK}}=T^{\text{G}}\equiv T$.  
Given $n\in\N$ odd and independent of $N$, the following holds.
\item{(i)}
For every $m\in\MM_{n,F}$ there exists a constant $C(m)>0$ that depends only on the $n$ non-zero components of $m$ such that, if
\be
M(N)
\leq 
\frac{C^2(m) N}{2(2+\varepsilon)\ln N}
\Eq(1.theo1.1)
\ee
for arbitrary $\varepsilon> 0$,  then 
\be
\P\left[
\bigcup_{N_0}\bigcap_{N>N_0}\left\{\xi^{(N)}(m)=T(\xi^{(N)}(m))\right\}
\right]
=1.
\Eq(1.theo1.2)
\ee

\item{(ii)}
If
\be
M(N)
\leq 
\frac{\inf\{C^2(m) : m\in\MM_{n,F}\}N}{2(2+n+\varepsilon)\ln N}
\Eq(1.theo1.3)
\ee
for arbitrary $\varepsilon> 0$, where the infimum is strictly positive,  then 
\be
\P\left[
\bigcup_{N_0}\bigcap_{N>N_0}\left(\bigcap_{m\in\MM_{n,F}}\left\{\xi^{(N)}(m)=T(\xi^{(N)}(m))\right\}\right)
\right]
=1.
\Eq(1.theo1.4)
\ee
\end{theorem}

\begin{theorem}[Dense Hopfield network]
    \TH(1.theo2)
Given an integer $p\geq 3$, consider the model defined by \eqv(1.1.2) with $F(x)=\frac{1}{p}x^p$.
Let $T$ denote either $T^{\text{HK}}$ or $T^{\text{G}}$.  The following holds for all odd $n\in\N$ independent of $N$.
\item{(i)}
For each $m\in\MM_{n,F}$ there exists a constant $C_p(m)>0$  that depends only on the $n$ non-zero components of $m$ and  on $p$ such that, if
\be
M(N)
\leq 
\frac{C^2_p(m) N^{p-1}}{2(2+\varepsilon)[(2p-3)!!]\ln N}
\Eq(1.theo2.1)
\ee
for arbitrary $\varepsilon> 0$,  then 
\be
\P\left[
\bigcup_{N_0}\bigcap_{N>N_0}\left\{\xi^{(N)}(m)=T(\xi^{(N)}(m))\right\}
\right]
=1.
\Eq(1.theo2.2)
\ee
\item{(ii)}
If
\be
M(N)
\leq 
\frac{\inf\{C^2_p(m) : m\in\MM_{n,F}\}N^{p-1}}{2(2+n(p-1)+\varepsilon)[(2p-3)!!]\ln N}
\Eq(1.theo2.3)
\ee
for arbitrary $\varepsilon> 0$, where the infimum is strictly positive, then 
\be
\P\left[
\bigcup_{N_0}\bigcap_{N>N_0}\left(\bigcap_{m\in\MM_{n,F}}\left\{\xi^{(N)}(m)=T(\xi^{(N)}(m))\right\}\right)
\right]
=1.
\Eq(1.theo2.4)
\ee
\end{theorem}

\begin{remark}
Explicit bounds on $C(m)$ and $C_p(m)$ are given, respectively, in \eqv(2.prop1.iii4) and in \eqv(2.prop2.iii4) with $f(x)=\frac{1}{p}x^p$.\end{remark}

We now turn to the modern Hopfield model. For $|x|\leq 1$ set
\be
I(x)=\frac{1+x}{2}\ln(1+x)+\frac{1-x}{2}\ln(1-x).
\Eq(1.theo3.0)
\ee

\begin{theorem}[Modern Hopfield network]
    \TH(1.theo3)
    
Given $\b>0$, take $F(x)=\exp(N\b x)$ in \eqv(1.1.2).  For this model $T^{\text{HK}}=T^{\text{G}}\equiv T$.  
The following holds for all odd $n\in\N$ independent of $N$, all $\b>0$ and all $m\in\MM_{n,F}$.
Set
\be
\b_c(m)\equiv\inf\left\{m_{\mu}>0 : 1\leq \mu\leq M(N)\right\}.
\Eq(1.theo3'.1)
\ee
If for arbitrary $\varepsilon>0$
\be
M(N)
\leq e^{N\inf\left\{\b,\frac{1}{2}\right\}\left[I(\b_c(m))-\varepsilon\right]},
\Eq(1.theo3'.2bis)
\ee
then
\be
\P\left[
\bigcup_{N_0}\bigcap_{N>N_0}\left\{\xi^{(N)}(m)=T(\xi^{(N)}(m))\right\}
\right]
=1.
\Eq(1.theo3'.3)
\ee
\end{theorem}

\begin{remark}

We stress that in the modern model with $n>1$, we cannot prove that \eqv(1.theo3'.3) still holds if we replace the event that 
\emph{$\xi^{(N)}(m)$ is a fixed point for a given $m\in\MM_{n,F}$} with  the event that \emph{this is true for all $m\in\MM_{n,F}$} if $M(N)$ grows exponentially fast with $N$. The reason for this is not model-dependent, but stems from the fact that a central tool in our  strategy of proof (namely, Proposition 3.1) requires that  $N^{-1}\ln M(N)\rightarrow 0$  as $N\rightarrow\infty$. Thus, to deal with the second event, the best we can hope for is to allow the number of patterns to grow sub-exponentially with  $N$.
\end{remark}

Now, let us specialise  the conditions on the growth of $M(N)$ obtained in the three theorems above to the case $n=1$.
In this case, the set $\{\xi^{(N)}(m), m\in\MM_{1,F}\}$ is reduced to the set of the original  $M$ patterns, regardless of the function $F$, and the $m$-dependent constants appearing in the three theorems are easily determined (see the remark below \eqv(1.theo2.4)). Specifically, for all  $m\in\MM_{1,F}$,  $C(m)=1$ in Theorem \eqv(1.theo1),  $C_p(m)=1$ in Theorem \thv(1.theo2), and  $\b_c(m)=1$  in Theorem \thv(1.theo3). 
To compare the resulting bounds on  $M(N)$ with those in the literature, it is important to note that  most papers do not make $\P$-almost sure statements, such as  \eqv(1.theo1.2) or \eqv(1.theo1.4), but rather statements that are only valid in $\P$-probability\footnote{This means that the union of intersections (\emph{i.e.}, the 
$\liminf_{N_0}$) in the events in \eqv(1.theo1.2) and \eqv(1.theo1.4) are removed, and that the limit $N\rightarrow\infty$ of the resulting probability is taken. In that case, it is clear from the proofs that the bounds on $M$ of Theorems \thv(1.theo1) and \thv(1.theo2) must be multiplied by two, whereas the bound in \eqv(1.theo3'.2bis) of Theorems \thv(1.theo3)  is left unchanged.}. Taking this into account, Theorems  \thv(1.theo1),  \thv(1.theo2) and  \thv(1.theo3) precisely reproduce the classical conditions on $M(N)$, which can be found in Theorem 4.1 of \cite{DP96} and the early paper \cite{McEPRV87} for the classical Hopfield model, and in part $2$ of Theorem 2 in \cite{DHLUV17}  for the dense Hopfield model. The latter result was initially derived through non-rigorous analysis in  \cite{PN86}. A sharpness result was also obtained for the classical Hopfield model in  \cite{Bo99}.  Finally, Condition \eqv(1.theo3'.2bis) of Theorem \thv(1.theo3) reduces to that mentioned in the remark below Theorem 3 of \cite{DHLUV17}, where the modern Hopfield model was first introduced for the $\b=1$ case only. However, the next lemma shows that  Theorem \eqv(1.theo3) can be refined when $n = 1$.
\begin{lemma}
    \TH(1.theo3.n=1)
When $n=1$, Condition \eqv(1.theo3'.2bis) of Theorem \eqv(1.theo3) can be improved to 
\be
M(N)\leq 2e^{(N-1)\left[f(\b)-(1+\d)\frac{\ln N}{N-1}\right]},
\Eq(1.theo3'.2ter)
\ee
for all $\d>0$ and all $\b>0$, where $f(\b)=2\b-\ln\cosh(2\b)$ is a strictly increasing function satisfying  
$\lim_{\b\rightarrow\infty}f(\b)=\ln 2$ and $f(\b)>\inf\{\b,\frac{1}{2}\}I(1)$ for all $\b>0$.
\end{lemma}

Based on the above observations, we conclude that the conditions obtained on the growth of $M(N)$ in the general $n> 1$ case differ from those in the $n=1$ case by at most a multiplicative constant  in the classical and dense Hopfield models, and by the constant governing the exponential growth in the modern model.

Before closing this section, let us recall that the effective functioning of a memory based on the energy function $E_{N,M}$, as  defined in \eqv(1.1.2), depends critically on the set of local minima of $E_{N,M}$. These are either fixed points of deterministic gradient descent dynamics or metastable states in stochastic dynamics. As such, they are unintended, spurious memories that often exhibit the same convergence properties as the patterns themselves. Since there are far more spurious memories than original patterns, this has a significant impact on the memory's ability to retrieve the original patterns. Conjecture \thv(1.conj.2) has clear implications in this respect. If it holds true, then under a condition on $M$ that depends on the chosen function $F$ (\emph{e.g.}~ \eqv(1.theo1.3) for the classical Hopfield model), we will have found and explicitly constructed the complete set of local minima that have non-zero overlap with finitely many patterns. When $M$ is a constant, all local minima belong to this set and we obtain the complete set of local minima of $E_{N,M}$. This would pave the way for a mathematically rigorous analysis of strategies such as ``unlearning'' and ``dreaming'', which are devised to eliminate spurious memories and improve convergence but which, thus far, rely on theoretical and numerical methods (see \cite{Hop83, BBF18, Enzo22, Enzo26} and references therein).

The tools used to prove the results of Section \thv(S1.2)  are of two types, purely analytical on the one hand, and probabilistic on the other. The strategy behind the proof of Theorem \thv(1.theo2.mix) consists in reducing the system of random equations \eqv(1.1.3)-\eqv(1.1.4) to a deterministic system in which the patterns $\xi^{\mu}$ are replaced by configurations of the Rademacher system on the hypercube.   In Section 2, exploiting the remarkable properties of the Rademacher system, summarised in Section \thv(S2.1), this purely deterministic system is introduced. Its solutions are are presented in Section \thv(S2.2), specifically  in Proposition \thv(2.prop1) (for the case $F(x)=\frac{1}{2}x^2$) and Proposition \thv(2.prop2) (for the general case). Section \thv(S6) contains the proofs of the results of Section \thv(S2.2). Section \thv(S3) provides  the key probabilistic tool (Proposition \thv(3.prop1)) that allows us to reduce  \eqv(1.1.4) to deterministic equations in the Rademacher system. It also contains the proofs of the results of Section \thv(S1.1). Those of Section \thv(S1.2) are presented in Section  \thv(S4). They are based on the results of the previous two sections and on probabilistic techniques.  Finally, in Section \thv(S5), we solve the system $\SS_{n,F}$ (see \eqv(1.2.1.12bis)) for the classical and dense Hopfield models and briefly discuss the energy of their mixed memories.


\section{Preparatory tools}
    \TH(S2)

\subsection{The Rademacher system} 
    \TH(S2.1)     

Rademacher's orthonormal system is classically a system of functions defined on $\R$ \cite{Pal}, \cite{Rad}. Transposed to the discrete hypercube, which is the space of interest here, it becomes a system of orthogonal configurations. More precisely, given an integer $n$ and setting $d\equiv d(n) =2^n$, Rademacher's system on $\S_{d}=\{-1,1\}^d$ is a collection of  $n$ configurations $r^{(n),1}, r^{(n), 2}\dots, r^{(n), n}$ in $\S_{d}$  defined  as follows. Let $\rho:\R\mapsto\{-1,1\}$ be the function
\be
\rho(t) = 
\begin{cases}
1 & \text{if $t-\lfloor t\rfloor\in [0, \frac{1}{2})$}, \\
-1 &  \text{if $ t-\lfloor t\rfloor\in [\frac{1}{2},1)$},
\end{cases}
\Eq(2.1.1)
\ee
where $\lfloor t\rfloor=\max\{k\in\Z\mid k\leq t\}$.
Then, for each $1\leq \nu\leq n$,
$r^{(n),\nu}=\bigl(r^{(n),\nu}_j\bigr)_{1\leq j\leq d}$ where
\be
r^{(n),\nu}_{j}=
\rho\left(2^{\nu-(n+1)}(j-1)\right)\,, \quad 1\leq j\leq d.
\Eq(2.1.2)
\ee
We call also Rademacher matrix the $n\times d$ matrix $R^{(n)}=\left(R^{(n)}_{\nu,j}\right)_{1\leq \nu\leq n, 1\leq j\leq d}$ of entries 
\be
R^{(n)}_{\nu,j}=r^{(n),\nu}_j, \quad 1\leq \nu\leq n, 1\leq j\leq d.
\Eq(2.1.3)
\ee 
Therefore, the rows of this matrix are given by the $n$ Rademacher configurations
\be
r^{(n),\nu}=\bigl(r^{(n),\nu}_j\bigr)_{1\leq j\leq d}\in\S_d, \quad 1\leq \nu\leq n,
\Eq(2.1.4)
\ee
and its columns, denoted by
\be
r^{(n)}_j=\bigl(r^{(n),\nu}_j\bigr)_{1\leq \nu\leq n}\in\S_n, \quad 1\leq j\leq d,
\Eq(2.1.5)
\ee  
are $d$ configurations in $\S_n$. 

As follows from definitions \eqv(2.1.1)-\eqv(2.1.2), each configuration $r^{(n),\nu}$ is piecewise constant over intervals of length $2^{n-\nu}$, and there are $2^{\nu}$ such intervals, of alternating signs, with the leftmost interval consisting of $+1$'s. Below is an example of a Rademacher matrix with $n=5$. For simplicity, we write $+$ and $-$ instead of $+1$ and $-1$.

\bigskip
\centerline{{\textsc{Figure 1.}} $R^{(n)}$, $n=5$, $d=2^5$}
\vskip.5truecm
{\centering{
\begin{tabular}{c c} 
$r^{(5),1}$ & $+ + + + + + + + + + + + + + + + - - - - - - - - - - - - - - - - $
\\ 
$r^{(5),2}$ & $+ + + + + + + +  - - - - - - - - + + + + + + + + - - - - - - - - $
\\
\bf{$r^{(5),3}$} & \bf{$+ + + + - - - - + + + + - - - - + + + + - - - -  + + + + - - - - $}
\\
$r^{(5),4}$ & $+ +  - - + +  - - + + - - + + - - + +  - -  + + - -  + +  - - + + - - $
\\
$r^{(5),5}$ & $+  - +  - +  - +  - +  - +  - +  - + - +  - +  - +  - +  - +  - +  - +  - +  -  $
\\
 & $r^{(5)}_1,\dots\hskip4.25truecm \dots, r^{(5)}_{16},\dots\hskip4.55truecm \dots, r^{(5)}_{32}$ 
\\
\end{tabular}
}}
\bigskip
\smallskip

The Rademacher system and matrix have the following fundamental properties. Set
\be
r^{(n),0}=\bigl(r^{(n),0}_j\bigr)_{1\leq j\leq d},\quad r^{(n),0}_j=1 \quad\text{for all } 1\leq j\leq d.
\Eq(2.1.6)
\ee

\begin{lemma}
    \TH(2.lem1)
\item{(i) (Orthogonality) } The $n+1$ configurations $r^{(n),\nu}$, $0\leq \nu\leq d$, form an orthogonal system:  
\be
d^{-1}\left(r^{(n),\nu}, r^{(n),\nu'}\right)=\d_{\nu,\nu'}\quad \text{for all } 0\leq \nu, \nu'\leq d,
\Eq(1.lem1.1)
\ee
where $\d_{\nu,\nu'}$ is the Kronecker delta.
\item{(ii) (Axial symmetry)} For all $1\leq \nu\leq n, 1\leq j\leq d$, the matrix elements of $R^{(n)}$ obey 
\be
r^{(n),\nu}_j=-r^{(n),\nu}_{2^n-j+1}.
\Eq(1.lem1.2)
\ee
\item{(iii)} The collection $\bigl\{r^{(n)}_j\bigr\}_{1\leq j\leq d}$ forms a complete enumeration of the $d$ elements of $\S_n$.
\end{lemma}

As a consequence of Lemma \thv(2.lem1), (iii), we have the following two properties. If $x$ and $y$ are vectors in $\R^n$, 
we denote by $x\odot y=(x_\nu y_\nu)_{1\leq \nu\leq n}$, their Hadamard product.
\begin{corollary}[Permutations]
    \TH(1.cor1)
\item{(i)}  
Given $1\leq i\leq d$, there exists a unique permutation $\pi:\{1,\dots,d\}\mapsto\{1,\dots,d\}$ of the columns of $R^{(n)}$ such that for all $1\leq j\leq d$
\be
r^{(n)}_j\odot r^{(n)}_i=r^{(n)}_{\pi(j)}, 
\Eq(1.cor1.1)
\ee
and
\be 
\left(r^{(n)}_{\pi(j)}\right)_{1\leq j\leq d}
= \left(r^{(n)}_{j'}\right)_{1\leq j'\leq d}.
\Eq(1.cor1.2)
\ee

\item{(ii)} If $\pi_1:\{1,\dots,n\}\mapsto\{1,\dots,n\}$ is a permutation of the rows of $R^{(n)}$ then there exists a unique permutation 
$\pi_2:\{1,\dots,d\}\mapsto\{1,\dots,d\}$ of its columns such that for all $1\leq \nu\leq n$ and all $1\leq i\leq d$
\be
r^{(n),\pi_1(\nu)}_i = r^{(n),\nu}_{\pi_2(i)},
\Eq(1.cor1.3)
\ee
and
\be
\left(r^{(n)}_{\pi_2(i)}\right)_{1\leq i\leq d}
= \left(r^{(n)}_{j}\right)_{1\leq j\leq d}.
\Eq(1.cor1.4)
\ee
\end{corollary}
The elementary proofs of Lemma \thv(2.lem1) and  Corollary \thv(1.cor1) are omitted.

The next simple lemma is a direct consequence of the definition of Rademacher configurations, whose proof we omit. 
(Its claim can also be easily visualised using Figure 1.)  It provides insight into the more elaborate formulation introduced in the rest of this Section, which ultimately leads to Lemma \thv(2.lem2). 
\begin{lemma}
    \TH(2.lem1bis)
Given $n\geq 2$, let $(n_1,n_2)$ be  a $2$-composition of $n$.
For each $1\leq\nu\leq n_1$, $r^{(n),\nu}$ is piecewise constant over $d(n_1)$ consecutive intervals of length $d(n_2)$, 
namely, for $1\leq j\leq d(n_1)$
\be
r^{(n),\nu}_i=r^{(n_1),\nu}_j \,\,\text{ for all }\,\, (j-1)d(n_2)+1\leq i\leq j d(n_2).
\Eq(2.lem1bis.1)
\ee
\end{lemma}

Alternatively, the matrix $R^{(n)}$ can be constructed using rooted plane trees. Such trees are embedded in the plane, with one of their vertices marked as the root, and are equipped with a left-to-right order starting from the root, which determines a total ordering of the nodes. 

The simplest of these constructions uses a (strictly) binary tree. Call the root $\emptyset$ and label the levels of the tree $1,2,\dots,n$,  where $1$ is the level descending from $\emptyset$ and the leaves (or dangling nodes) are at level $n$. Each node branches into exactly two children nodes, the node descending from the left branch carrying a $+$ sign and the node descending from the right branch carrying a $-$ sign. The resulting tree completely determines $R^{(n)}$. To see this, consider level $1\leq \nu\leq n$ of the tree. It has exactly $2^{\nu}$ nodes. Then, $r^{(n),\nu}$ is the configuration consisting of $2^{\nu}$ intervals of length $2^{n-\nu}$, one associated with each node, which is constant on each interval and equal to $+1$ if the associated node
has a $+$  sign and $-1$ if it has a $-$ sign. This construction is illustrated in Figure 2. Compare with Figure 1.

\bigskip
\centerline{{\textsc{Figure 2.}} $R^{(n)}$, $n=5$, $d=2^5$}
\smallskip
\centerline{
\begin{minipage}[b]{17.3cm}
\includegraphics[width=17.3cm]
{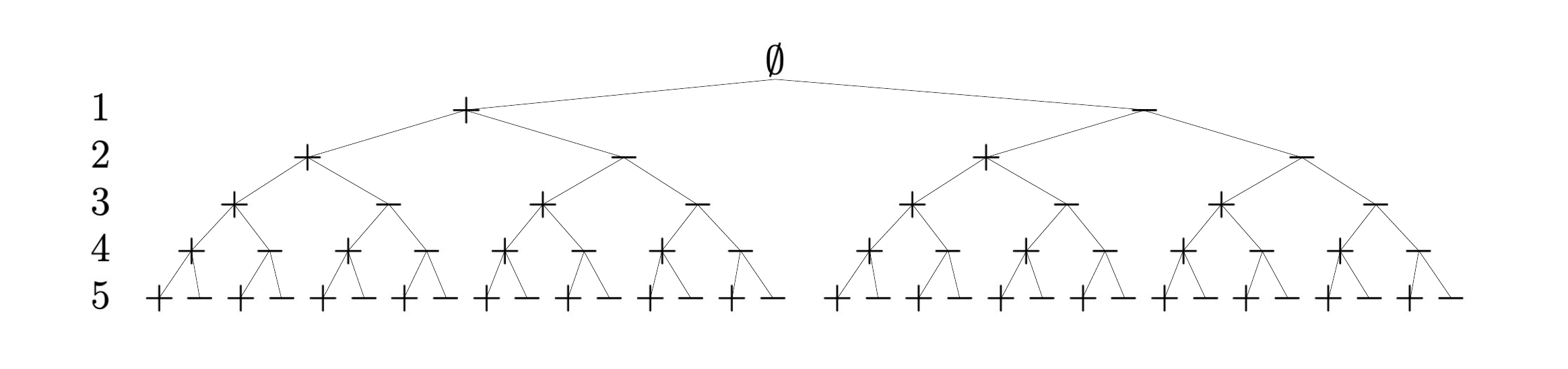}
\end{minipage}
}

To make this construction formal, we need some notation. 
Let $\s=(\s_1,\dots,\s_d)\in\S_{d}$ and $\s'=(\s'_1,\dots,\s'_{d'})\in\S_{d'}$ be two configurations, where as before
$d\equiv d(n)= 2^{n}$ and where we wrote $d'\equiv d(n')$ for simplicity. Given an integer $k\geq 1$, we denote by $k\otimes\s$ and call \emph{dilated configuration} the configuration 
\be
\begin{split}
k\otimes\s&=(
\underbrace{\s_1,\dots,\s_1},\dots\underbrace{\s_d,\dots,\s_d})\in\S_{kd},
\\
& \hskip1.5truecm k \hskip2.2truecm k
\end{split}
\Eq(2.1.7)
\ee
in which each coordinate $\s_i$ is duplicated in $k$ identical copies. We denote by $\s\oplus\s'$ and call \emph{concatenated configuration} the configuration
\be
\s\oplus\s'=(\s_1,\dots,\s_d,\s'_1,\dots,\s'_{d'})\in\S_{d+d'}.
\Eq(2.1.8)
\ee
When applied to matrices, the dilation and concatenation operators act on their rows. 
Specifically, the matrix  
$
k\otimes R^{(n)}
$
is the $n\times kd$ matrix whose rows are the dilated configurations (see \eqv(2.1.4))
\be
k\otimes r^{(n),\nu}\in\S_{kd}, \quad 1\leq \nu\leq n.
\Eq(2.1.9)
\ee
Thus, each column vector $r^{(n)}_j$ of $R^{(n)}$ (see \eqv(2.1.5)) is duplicated $k$ times.
Similarly, given $k$ copies of the Rademacher matrix $R^{(n)}$, the sum 
$
\oplus_{s=1}^{k}R^{(n)}
$ 
is the $n\times kd$ matrix whose rows are the concatenated configurations
\be
\begin{split}
&
\oplus_{s=1}^{k}r^{(n),\nu}=\underbrace{r^{(n),\nu}\oplus\dots\oplus r^{(n),\nu}}\in\S_{kd}, \quad 1\leq \nu\leq n.
\\
& \hskip4truecm k
\end{split}
\Eq(2.1.10)
\ee

Next, we introduce the specific trees we are interested in. Recall the definition of the composition of an integer $n$ into $\ell$ summands from the paragraph above \eqv(1.2.1.4).  (Also recall that each of these $\ell$ summands are strictly positive.)

\begin{definition}
    \TH(2.def.1)
Given an integer $\ell\leq n$ and an $\ell$-composition $(n_1,\dots,n_{\ell})$ of $n$, we denote by 
$\TT^{(n)}_{\ell}(n_1,n_2,\dots,n_\ell)$ the $\ell$-level rooted plane tree defined as follows. Let $0$ be the root level and number the subsequent levels $1,\dots,\ell$, where $\ell$ is the leaf level. Set $n_0=0$. Then, for each $0\leq k\leq \ell-1$, each node at level $k$ has  $2^{n_{k+1}}$ children labelled from left to right with the column configurations of $R^{(n_{k+1})}$, 
\be
r^{(n_{k+1})}_{j_{k+1}}=\bigl(r^{(n_{k+1}),\nu_{k+1}}_{j_{k+1}}\bigr)_{1\leq \nu_{k+1}\leq n_{k+1}}\in\S_{n_{k+1}}, \quad 1\leq j_{k+1}\leq d(n_{k+1}).
\Eq(2.1.11)
\ee
\end{definition}

As an example, the tree in Figure 2 is the $5$-level tree $\TT^{(5)}_{5}(1,1,\dots,1)$. Indeed $R^{(1)}$ is the $1\times 2$ matrix of column configurations $r^{(1)}_{1}=+1$ and $r^{(1)}_{2}=-1$. The next picture illustrates the tree $\TT^{(5)}_{2}(2,3)$ in two different ways: with the labels $r^{(n_k)}_i$ (Figure 3a) and replacing these labels by their actual column vectors (Figure 3b).

\bigskip
\centerline{{\textsc{Figure 3a.}}  $\TT^{(5)}_{2}(2,3)$}
\smallskip
\centerline{
\begin{minipage}[b]{12.5cm}
\includegraphics[width=12.5cm]
{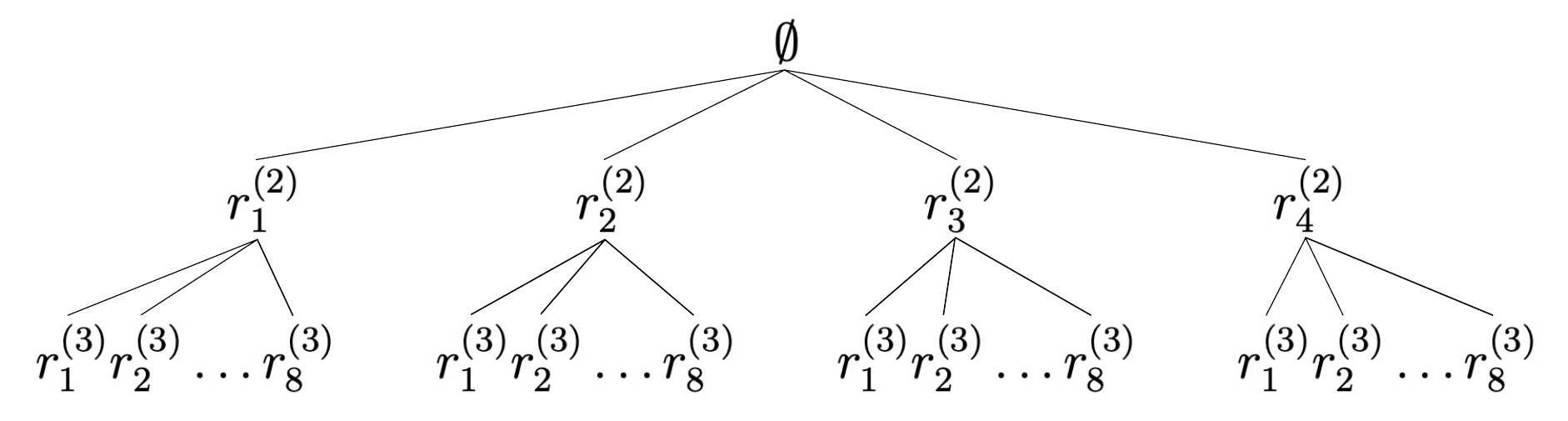}
\end{minipage}
}
\bigskip
\centerline{{\textsc{Figure 3b.}} $\TT^{(5)}_{2}(2,3)$}
\smallskip
\centerline{
\begin{minipage}[b]{15cm}
\includegraphics[width=15cm]
{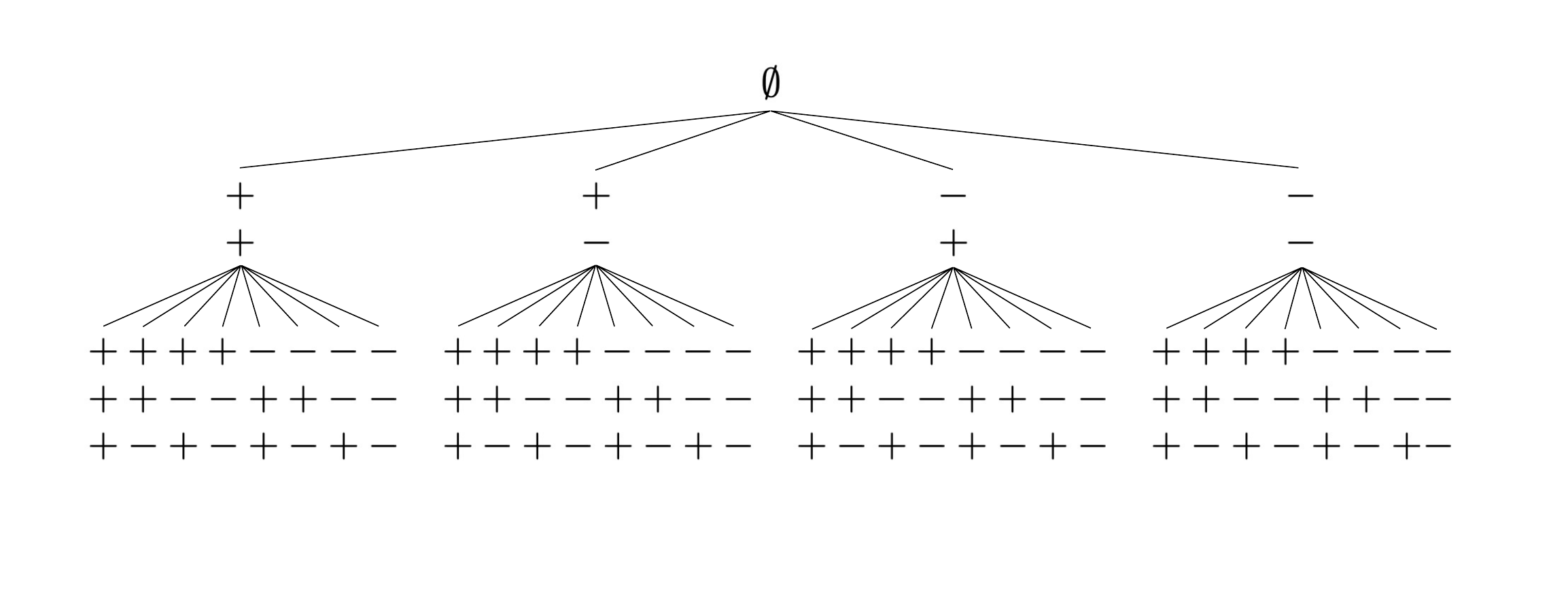}
\end{minipage}
}

As the next lemma shows, for any given $\ell$-composition $(n_1,\dots,n_{\ell})$ of $n$, $R^{(n)}$ is completely determined 
by the pair formed by the tree $\TT^{(n)}_{\ell}(n_1,n_2,\dots,n_\ell)$ and a specific sequence of dilation coefficients that we now define.

\begin{lemma}[Tree-based representation of $R^{(n)}$]
    \TH(2.lem2)
For any given tree $\TT^{(n)}_{\ell}(n_1,n_2,\dots,n_\ell)$, setting $n_0\equiv 0$, the $n$ Rademacher configurations $r^{(n),\nu}\in\S_d$, $1\leq \nu\leq n$, can be constructed as follows: for each $1\leq k\leq \ell$ and each 
$
n_0
+\dots+n_{k-1}+1
\leq 
\nu
\leq 
n_1+\dots+n_k
$, 
setting  $\nu_k= \nu-(n_0+\dots+n_{k-1})$,
\be
r^{(n),\nu}=2^{n-(n_1+\dots+n_k)}\otimes\left(r^{(n_k),\nu_k}\oplus\dots\oplus r^{(n_k),\nu_k}\right),
\Eq(2.1.12)
\ee
where the concatenation is over $2^{n_0+\dots+n_{k-1}}$ vectors $r^{(n_k),\nu_k}$.
Synthetically, the matrix $R^{(n)}$ can be written as
\be
R^{(n)}=
\begin{pmatrix} 
2^{n-n_1}\otimes R^{(n_1)}
\\
\vdots 
\\ 
2^{n-(n_1+\dots+n_k)}\otimes\left(R^{(n_k)}\oplus\dots\oplus R^{(n_k)}\right) 
\\ 
\vdots 
\\
2^{n-(n_1+\dots+n_\ell)}\otimes\left(R^{(n_\ell)}\oplus\dots\oplus R^{(n_\ell)}\right)
\end{pmatrix},
\Eq(2.1.14)
\ee
where for each $1\leq k\leq \ell$ the concatenation is over $2^{n_0+\dots+n_{k-1}}$ matrices $R^{(n_k)}$.
\end{lemma}

\noindent{\bf Examples:} 
Using the  $5$-level tree $\TT^{(5)}_{2}(2,3)$ of Figure 3a, we have
\be
R^{(5)}=
\begin{pmatrix} 
2^{3}\otimes R^{(2)}
\\
2^{0}\otimes\left(R^{(3)}\oplus\dots\oplus R^{(3)}\right)
\end{pmatrix}.
\Eq(2.ex.2)
\ee
Compare with  the representations of the  Rademacher matrix $R^{(5)}$ obtained using the $1$-level tree $\TT^{(5)}_{1}(5)$ of Figure 1 and the $5$-level tree $\TT^{(5)}_{5}(1,1,\dots,1)$ of Figure 2.

\begin{proof}[Proof of Lemma \thv(2.lem2)]  
First, taking $\ell=n$ in Definition \thv(2.def.1), it is readily verified using \eqv(2.1.1)-\eqv(2.1.2) that
the Rademacher system in $\S_d$  can be expressed using the pair 
\be
\{\TT^{(n)}_{n}(1,1,\dots,1), (2^{n-\nu})_{1\leq \nu\leq n}\}
\ee
consisting of the $n$-level tree $\TT^{(n)}_{n}(1,1,\dots,1)$ and the sequence $(2^{n-\nu})_{1\leq \nu\leq n}$ of dilatation coefficients. Indeed for $\ell=n$, $n_k=1$ and $\nu_k=1$  for all $1\leq k\leq n$,  the Rademacher configurations $r^{(n),\nu}\in\S_d$, $1\leq \nu\leq n$, can be written as
\be
r^{(n),\nu}=2^{n-\nu}\otimes\left(r^{(1),1}\oplus\dots\oplus r^{(1),1}\right),
\Eq(2.1.15)
\ee
where the concatenation is over $2^{n-(\nu-1)}$ copies of $r^{(1),1}$ and $r^{(1),1}=(r^{(1)}_{1},r^{(1)}_{2})=(1,-1)$. Since there is exactly one composition of $n$ into $n$  summands there is no ambiguity in writing  $\TT^{(n)}_{n}\equiv\TT^{(n)}_{n}(1,1,\dots,1)$ from now on. Note that $\TT^{(n)}_{n}$ is exactly  the  $n$-level  strictly binary tree described in the paragraph preceding Figure 2 (and illustrated in this figure for $n=5$).

Next, taking $\ell=1$ in Definition \eqv(2.def.1), it is straightforward that the Rademacher system in $\S_d$
can be identified with the $1$-level tree $\TT^{(n)}_{1}(n)$. In this case the dilatation sequence reduces to $1$.
From the previous two facts we deduce that the two pairs $\{\TT^{(n)}_{n}, (2^{n-\nu})_{1\leq \nu\leq n}\}$
and $\{\TT^{(n)}_{1}(n), 1\}$ completely determine one and the same  the Rademacher matrix $R^{(n)}$.

The conclusion of the lemma will now follow from the fact that  $\TT^{(n)}_{n}$ is a binary tree, and that such a tree can be constructed recursively from any sequence of $\ell\leq n$ binary trees $\TT^{(n_k)}_{n_k}$, $1\leq k\leq \ell$, where
$(n_1,\dots,n_{\ell})$ is an $\ell$-composition of $n$.  Simply replace each leaf of $\TT^{(n_1)}_{n_1}$ with a copy of  
$\TT^{(n_2)}_{n_2}$, then replace each leaf of each tree  $\TT^{(n_2)}_{n_2}$ with a copy of $\TT^{(n_3)}_{n_3}$, and so 
on  $\ell-1$ times. It remains to use the fact established above that, for each $1\leq k\leq \ell$, the pair 
$
\bigl\{\TT^{(n_k)}_{n_k}, (2^{n-(n_0+\dots + n_{k-1}+l)})_{1\leq l\leq n_k}\bigr\}
$,
where $1\leq l\leq n_k$ labels the levels of the tree $\TT^{(n_k)}_{n_k}$, can be replaced with the pair 
$
\bigl\{\TT^{(n_k)}_{1}(n_k), (2^{n-(n_1+\dots+n_k)})\bigr\}
$.
This gives the claim of the lemma.
\end{proof}

\subsection{Solutions of the mixed-memory equations in the Rademacher system}
    \TH(S2.2) 

Given an integer $n$ and a vector $m=(m_{\nu})_{1\leq\nu\leq n}\in\R^n$, consider the system of $n$ equations in $d(n)=2^n$ variables
\be
m_{\nu}
=
\frac{1}{d(n)}\sum_{1\leq i\leq d(n)}r^{(n),\nu}_i
\sign\left[
\left(r^{(n)}_i,m\right)
\right],
\quad 1\leq\nu\leq n,
\Eq(2.2.1)
\ee
where $(\cdot, \cdot)$ denotes the inner product in $\R^n$,  $r^{(n)}_j=(r^{(n),\nu}_j)_{1\leq \nu\leq n}\in\S_n$, $1\leq j\leq d(n)$, are the column vectors of the $n\times d(n)$ Rademacher matrix  $R^{(n)}$ (see \eqv(2.1.3)-(2.1.5)), and 
\be
\sign(t)=
\begin{cases}
+1 & \text{if $t>0$}, \\
-1 & \text{if $t<0$}, \\
0 & \text{if $t=0$}. \\
\end{cases}
\Eq(2.2.2)
\ee

In Section \thv(S3), we will prove that for each $1\leq\nu\leq n$, choosing $F(x)=\frac{1}{2}x^2$ in  \eqv(1.2.1.11), the event appearing in  \eqv(1.2.1.2) of Definition  \thv(1.def1.2) 
reduces to \eqv(2.2.1) with $\P$-probability $1$ (see in particular \eqv(1.theo1.mix3) in the proof of Theorem \thv(1.theo2.mix)). 
A similar result holds in the general case of mixed memories of type $F$, which will be introduced in \eqv(2.prop2.0).
Since the  r\^ole of the patterns in \eqv(1.2.1.2) is now played by the deterministic Rademacher row vectors, we will refer to the system of equations  \eqv(2.2.1) and  \eqv(2.prop2.0) as the \emph{mixed-memory equations} in the Rademacher system. The connection between the two sets of vectors is made explicit in \eqv(3.1.6'). For the sake of simplicity, we treat the case $F(x)=\frac{1}{2}x^2$ and the general case separately.

The aim of this section is to construct explicit solutions to these systems.  They will in turn be used in Section \thv(S3) to construct the mixed memories $\xi^{(N)}(m)$,  $m=(m_{\mu})_{1\leq \mu\leq M}\in \R^M$, of Definition \thv(1.def1.2). Note that  \emph{a priori} we are looking for all solutions, not just those that lead to local minima of the energy function of the Hopfield model under consideration.

As  mentioned in Section \thv(S1.2), several classes of solutions have been identified in \cite{AGS85a}, referred to as symmetric and asymmetric solutions, the latter  class being itself divided into continuous and discontinuous asymmetric solutions. 

Symmetric solutions were first discovered by Amit \emph{et al.}~\cite{AGS85a} and later rigorously and independently established in \cite{KP89} and \cite{BR90}. Given $1\leq s\leq n$, we call $s$-symmetric a solution 
$
m=(m_{\nu})_{1\leq \nu\leq n}
$  
of \eqv(2.2.1) with exactly $s$ non-zero components of equal absolute value. 
Throughout this section
\be
\boldsymbol{1}_{k}=(1,\dots,1)
\quad\text{and}\quad
\boldsymbol{0}_{k}=(0,\dots,0)
\ee
denote the vectors in $\R^k$ whose $k$ components are all equal to $1$ or $0$, respectively.

\begin{lemma}[$s$-symmetric solutions] 
    \TH(2.lem3)
Given $1\leq s\leq n$,  set
\be
\a^{(s)}=2^{-s+1}{{s-1}\choose{\lfloor(s-1)/2\rfloor}}.
\Eq(2.lem1.1)
\ee

\item{(i)} The vector 
$
m=\a^{(s)}\boldsymbol{1}_{s}\oplus \boldsymbol{0}_{n-s}
$
verifies \eqv(2.2.1).

\item{(ii)} In addition, if  $\pi$ is an arbitrary permutation of $\{1,\dots,n\}$ and  $(\varepsilon_{\nu})_{1\leq \nu\leq n}$,  
$\varepsilon_{\nu}\in\{-1,1\}$, is an arbitrary sequence of signs, then $(\varepsilon_{\nu}m_{\pi(\nu)})_{1\leq \nu\leq n}$  
verifies \eqv(2.2.1). 
\end{lemma}

The knowledge of symmetric solutions is the building block for constructing more complicated  \emph{asymmetric solutions}.
The asymmetric solutions we construct include all examples of continuous asymmetric solutions obtained numerically in \cite{AGS85a}.

\begin{proposition}[Asymmetric solutions]
    \TH(2.prop1)

Given $1\leq s\leq n$, let  $(s_1,\dots,s_{\ell})$ be an $\ell$-composition of $s$ such that $s_k\geq 2$ is even for all 
$1\leq k\leq \ell-1$ and $s_{\ell}\geq 1$ has arbitrary parity.  Set
\be
\g^{(k)}
\equiv \prod_{l=1}^{k}\a^{(s_{l})}, \quad 1\leq k\leq \ell,
\Eq(2.prop1.1)
\ee
where $\a^{(s_l)}$ is defined as in \eqv(2.lem1.1), and let $m=(m_{\nu})_{1\leq \nu\leq n}$ be the vector
\be
m=\g^{(1)}\boldsymbol{1}_{s_1}\oplus \g^{(2)}\boldsymbol{1}_{s_2}\oplus\dots
\oplus \g^{(\ell)}\boldsymbol{1}_{s_{\ell}}\oplus \boldsymbol{0}_{n-s}.
\Eq(2.prop1.2)
\ee
If the sequence $\left(\g^{(k)}\right)_{1\leq k\leq \ell}$ satisfies the conditions
\be
2\g^{(k)}>s_{k+1}\g^{(k+1)}+\dots+s_{\ell}\g^{(\ell)}, \quad 1\leq k\leq \ell-1,
\Eq(2.prop1.2')
\ee 
then the following holds.
\item{(i)} The vector  $m=(m_{\nu})_{1\leq \nu\leq n}$ verifies \eqv(2.2.1). 
\item{(ii)}  Given any permutation $\pi$  of $\{1,\dots,n\}$ and 
any  sequence $(\varepsilon_{\nu})_{1\leq \nu\leq n}$,  $\varepsilon_{\nu}\in\{-1,1\}$, 
the vector $(\varepsilon_{\nu}m_{\pi(\nu)})_{1\leq \nu\leq n}$  verifies  \eqv(2.2.1).
\item{(iii)} If $s_\ell$ is even then 
\be
\left(
\sign\left[
\left(r^{(n)}_i,m\right)
\right]
\right)_{\{1\leq i\leq d(n)\}}
\in\{-1,0,1\}^{d(n)},
\Eq(2.prop1.iii1)
\ee
and there exists $1\leq i\leq d(n)$ such that
$
\bigl(r^{(n)}_i,m\bigr)=0
$.
If $s_\ell$ is odd then 
\be
\left(
\sign\left[
\left(r^{(n)}_i,m\right)
\right]
\right)_{\{1\leq i\leq d(n)\}}
\in\{-1,1\}^{d(n)},
\Eq(2.prop1.iii2)
\ee
and
\be
\inf_{1\leq i\leq d(n)}
\left|\left(r^{(n)}_i,m\right)\right|\geq C(m)>0,
\Eq(2.prop1.iii3)
\ee
where 
\be
C(m)\geq\min\left\{ 
\min_{1\leq k\leq \ell-1}\left[2\g^{(k)}-\left(s_{k+1}\g^{(k+1)}+\dots+s_{\ell}\g^{(\ell)}\right)\right],  \g^{(\ell)}\right\}>0.
\Eq(2.prop1.iii4)
\ee
\end{proposition}

\begin{remark}
The order of the summands of the chosen $\ell$-composition $(s_1,\dots,s_{\ell})$ of $s$ is important: permuting distinct summands leads to a different sequence \eqv(2.prop1.1) and to a different set of conditions \eqv(2.prop1.2').
\end{remark}

The system of equations \eqv(2.2.1) will be used to construct the mixed memories of the classical Hopfield model. The next proposition will allow us to construct the general type-$F$ mixed memories of Theorem \thv(1.theo2.mix). Given a function 
$f:\R\rightarrow \R$ satisfying $f(x)>0$ for all $x>0$ and $f(0)=0$, let $\boldsymbol{f}:\R^n\rightarrow\R^n$ be the function that assigns to $m$ the vector $\boldsymbol{f}(m)=(f(m_1),\dots,f(m_n))$. The fixed point equation \eqv(2.2.1) is now replaced by
\be
m_{\nu}
=
\frac{1}{d(n)}\sum_{1\leq i\leq d(n)}r^{(n),\nu}_i
\sign\left[
\left(r^{(n)}_i, \boldsymbol{f}(m)\right)
\right],
\quad 1\leq\nu\leq n.
\Eq(2.prop2.0)
\ee

\begin{proposition}[Asymmetric solutions of type $f$] 
    \TH(2.prop2)

Given $1\leq s\leq n$, let  $(s_1,\dots,s_{\ell})$ be an $\ell$-composition of $s$ such that $s_k\geq 2$ is even for all 
$1\leq k\leq \ell-1$ and $s_{\ell}\geq 1$ has arbitrary parity.  Let $m=(m_{\nu})_{1\leq \nu\leq n}$ be the vector
\be
m=\g^{(1)}\boldsymbol{1}_{s_1}\oplus \g^{(2)}\boldsymbol{1}_{s_2}\oplus\dots
\oplus \g^{(\ell)}\boldsymbol{1}_{s_{\ell}}\oplus \boldsymbol{0}_{n-s},
\Eq(2.prop2.2)
\ee
where the sequence $\left(\g^{(k)}\right)_{1\leq k\leq \ell}$ is defined as in \eqv(2.prop1.1).
If the sequence $\left(\g^{(k)}\right)_{1\leq k\leq \ell}$ satisfies the conditions
\be
2f\left(\g^{(k)}\right)>s_{k+1}f\left(\g^{(k+1)}\right)+\dots+s_{\ell}f\left(\g^{(\ell)}\right), \quad 1\leq k\leq \ell-1,
\Eq(2.prop2.3)
\ee 
then the following holds.
\item{(i)}  
The vector $m=(m_{\nu})_{1\leq \nu\leq n}$ verifies \eqv(2.prop2.0). 
\item{(ii)} Given any permutation $\pi$  of $\{1,\dots,n\}$, the vector $(m_{\pi(\nu)})_{1\leq \nu\leq n}$  verifies \eqv(2.prop2.0).
If in addition the function $f$ is odd then, given any permutation $\pi$  of $\{1,\dots,n\}$ and 
any  sequence $(\varepsilon_{\nu})_{1\leq \nu\leq n}$,  $\varepsilon_{\nu}\in\{-1,1\}$, 
the vector $(\varepsilon_{\nu}m_{\pi(\nu)})_{1\leq \nu\leq n}$  verifies \eqv(2.prop2.0).
\item{(iii)} If $s_\ell$ is even then 
\be
\left(
\sign\left[
\left(r^{(n)}_i, \boldsymbol{f}(m)\right)
\right]
\right)_{\{1\leq i\leq d(n)\}}
\in\{-1,0,1\}^{d(n)},
\Eq(2.prop2.iii1)
\ee
and there exists $1\leq i\leq d(n)$ such that
$
\bigl(r^{(n)}_i,\boldsymbol{f}(m)\bigr)=0
$.
If $s_\ell$ is odd then 
\be
\left(
\sign\left[
\left(r^{(n)}_i, \boldsymbol{f}(m)\right)
\right]
\right)_{\{1\leq i\leq d(n)\}}
\in\{-1,1\}^{d(n)},
\Eq(2.prop2.iii2)
\ee
and
\be
\inf_{1\leq i\leq d(n)}
\left|\left(r^{(n)}_i,\boldsymbol{f}(m)\right)\right|\geq C_f(m)>0,
\Eq(2.prop2.iii3)
\ee
where 
\be
\begin{split}
&
C_f(m)
\\
&\geq\min\Bigl\{
\min_{1\leq k\leq \ell-1}\left\{2f\left(\g^{(k)}\right)-\left[s_{k+1}f\left(\g^{(k+1)}\right)+\dots+s_{\ell}f\left(\g^{(\ell)}\right)\right]\right\}, f\left( \g^{(\ell)}\right)\Bigr\}
\\
&>0.
\end{split}
\Eq(2.prop2.iii4)
\ee
\end{proposition}

The proof of assertions (ii) of Lemma \thv(2.lem3) and Propositions \thv(2.prop1) and \thv(2.prop2) follow from the same arguments. In order to avoid repetition of the proofs, we  gather these statements in the lemma below. 

\begin{lemma}
    \TH(2.lem4)
Let $\pi$ be is an arbitrary permutation of $\{1,\dots,n\}$ and let $\varepsilon=(\varepsilon_{\nu})_{1\leq \nu\leq n}\in\{-1,1\}^n$ be an arbitrary sequence of signs. 
\item{(i)} If $m=(m_{\nu})_{1\leq \nu\leq n}$  verifies  \eqv(2.2.1), then $(\varepsilon_{\nu}m_{\pi(\nu)})_{1\leq \nu\leq n}$  verifies \eqv(2.2.1). 

\item{(ii)} If $m=(m_{\nu})_{1\leq \nu\leq n}$  verifies  \eqv(2.prop2.0) for an arbitrary function $f$ satisfying $f(x)>0$ for all $x>0$ and $f(0)=0$, then
$(m_{\pi(\nu)})_{1\leq \nu\leq n}$ verifies \eqv(2.prop2.0). In addition, if $f$ is an odd function, then $(\varepsilon_{\nu}m_{\pi(\nu)})_{1\leq \nu\leq n}$ verifies \eqv(2.prop2.0).
\end{lemma}

Finally, we come to the discontinuous asymmetric solutions.  We can think of these as all asymmetric solutions that are not of the type defined in Proposition \thv(2.prop1). We have little to say about them, since we could not find a discontinuous solution other than the single example proposed  in \cite{AGS85a}, namely,  $s=5$, $\ell=3$, $s_1=s_2=2$,  $s_3=1$ and
\be
m=\left(
\sfrac{3}{8}, \sfrac{3}{8}, \sfrac{1}{4}, \sfrac{1}{4}, \sfrac{1}{2}, 0\dots,0
\right).
\Eq(2.discont.1)
\ee
One can indeed check that \eqv(2.discont.1) verifies  \eqv(2.2.1), but one also sees that  for this solution
\be
\left(
\sign\left[
\left(r^{(5)}_i,m\right)
\right]
\right)_{\{1\leq i\leq d(5)\}}
\in\{-1,0,1\}^{d(5)},
\Eq(2.discont.2)
\ee
where zero is achieved, \emph{e.g.},~for $r^{(5)}_i=(+1,-1,+1,+1,-1)$.  Since $s_{\ell}$ is odd, this contradicts \eqv(2.prop1.iii2) of Proposition \thv(2.prop1). This type of solution thus seems to obey a completely different logic to that which governs the emergence of the asymmetric solutions of Proposition \thv(2.prop1).


\section{Proofs of the results of Section \thv(S2.2)}
    \TH(S6) 

In this section,  we successively prove Lemma \thv(2.lem3), Proposition \thv(2.prop1), Proposition \thv(2.prop2) and Lemma \thv(2.lem4). The claim of Lemma \thv(2.lem3) is contained in \cite{KP89} (see Theorem 1.3 and Proposition 3.4 (b)) and in a less direct way in \cite{BR90} (see Theorem 1 and its proof). Nevertheless, we present a simple proof  within the technical framework of Section \thv(S2.1), both for the sake of completeness and to provide a first illustration of our approach. Recall the notation  $d(n)=2^n$.

\subsection{Proof of Lemma \thv(2.lem3)} 
    \TH(S6.1)    
We start by proving assertion (i), namely, we  look for solutions $m$  of \eqv(2.2.1) of the form 
$
m=a\boldsymbol{1}_{s}\oplus \boldsymbol{0}_{n-s}
$
where $a>0$ is to be determined.
The proof hinges on the key fact that, by  Lemma \thv(2.lem2) with $n=n_1+n_2$, $n_1=s$ and $n_2=n-s$,
\be
R^{(n)}=
\begin{pmatrix} 
2^{n-s}\otimes R^{(s)}
\\
2^{0}\otimes\left(R^{(n-s)}\oplus\dots\oplus R^{(n-s)}\right)
\end{pmatrix},
\Eq(2.lem1.8)
\ee
where the concatenation is over $2^{s}$ matrices $R^{(n-s)}$. 

First assume that  $1\leq\nu\leq s$. By \eqv(2.lem1.8) and the above choice of $m$,  \eqv(2.2.1) becomes
\bea
m_{\nu}
=
a
&=&
\frac{1}{d(n)}\sum_{1\leq i\leq d(n)}r^{(n),\nu}_i
\sign\left[
\left(r^{(n)}_i,\boldsymbol{1}_{s}\oplus \boldsymbol{0}_{n-s}\right)
\right]
\Eq(2.lem1.2ter)
\\
&=&
\frac{1}{d(s)}\sum_{1\leq i\leq d(s)}r^{(s),\nu}_i
\sign\left[
\left(r^{(s)}_i,\boldsymbol{1}_{s}\right)
\right].
\Eq(2.lem1.2bis)
\eea
Note that since $a$ is positive, we have removed it from the right-hand side of \eqv(2.lem1.2ter). 
Note also that \eqv(2.lem1.2bis) now  only depends on the matrix $R^{(s)}$. For simplicity, fix $\nu=1$ in \eqv(2.lem1.2bis). By Lemma \thv(2.lem2) with $s=s_1+s_2$, $s_1=1$ and $s_2=s-1$, we have  $r^{(s),1}=2^{s-1}\otimes R^{(1)}$ and
\be
R^{(s)}=
\begin{pmatrix} 
2^{s-1}\otimes R^{(1)}
\\
2^{0}\otimes\left(R^{(s-1)}\oplus R^{(s-1)}\right)
\end{pmatrix}.
\Eq(2.lem1.2)
\ee
Using \eqv(2.lem1.2) and the fact that, by \eqv(1.lem1.2) of Lemma \thv(2.lem1), multiplying the column vectors of $R(s-1)$ by $-1$ only induces a permutation of the columns of $R(s-1)$,  \eqv(2.lem1.2bis)  reduces to
\be
a
=
\frac{2}{d(s)}\sum_{1\leq i\leq d(s)/2}
\sign\left[
\left(1+\bigl(r^{(s-1)}_i,\boldsymbol{1}_{s-1}\bigr)\right)
\right].
\Eq(2.lem1.3)
\ee
Now we distinguish two cases.  First, if $s-1$ is even, then
\be
\bigl(r^{(s-1)}_i,\boldsymbol{1}_{s-1}\bigr) \in\left\{-(s-1), \dots,-2, 0, 2,\dots, (s-1)\right\}.
\Eq(2.lem1.4)
\ee
Again by \eqv(1.lem1.2) of Lemma \thv(2.lem1), 
\be
\left|\left\{1\leq i\leq \frac{d}{2} : \bigl(r^{(s-1)}_i,\boldsymbol{1}_{s-1}\bigr)\geq 2\right\}\right|
=
\left|\left\{1\leq i\leq \frac{d}{2} : \bigl(r^{(s-1)}_i,\boldsymbol{1}_{s-1}\bigr)\leq -2\right\}\right|.
\Eq(2.lem1.5)
\ee
Since the sign function in \eqv(2.lem1.3) has opposite values on these two sets, \eqv(2.lem1.3) becomes
\be
a
=
\frac{2}{d(s)}\sum_{1\leq i\leq d(s)/2}\1_{\left\{\left(r^{(s-1)}_i,\boldsymbol{1}_{s-1}\right)=0\right\}}
=
2^{-(s-1)}{{s-1}\choose{(s-1)/2}}
=
\a^{(s)},
\Eq(2.lem1.6)
\ee
where$\1_{A}$ denotes the indicator function of the set $A$ and $\a^{(s)}$ in defined in \eqv(2.lem1.1).
If on the other hand $s-1$ is odd, \eqv(2.lem1.4) is replaced by
\be
\bigl(r^{(s-1)}_i,\boldsymbol{1}_{s-1}\bigr) \in\left\{-(s-1), \dots, -3,-1, 1, 3,\dots, (s-1)\right\},
\Eq(2.lem1.7)
\ee
and the indicator function appearing in the first equality in \eqv(2.lem1.6) must be replaced by the indicator function of
$\bigl\{\bigl(r^{(s-1)}_i,\boldsymbol{1}_{s-1}\bigr)=1\bigr\}$. This again gives $a=\a^{(s)}$. 
Thus, $m_1=\a^{(s)}$ and the proof in the case $\nu=1$ is complete.

The case $2\leq \nu\leq s$ can be reduced to the case $\nu=1$ by permutation of the rows and columns of $R^{(s)}$, using Corollary \thv(1.cor1), (ii). 

In the case $s+1\leq \nu\leq n$, we deduce from \eqv(2.lem1.8)  that the quantities
$
\sign\bigl[
\bigl(r^{(n)}_i,m\bigr)
\bigr]
$ 
in \eqv(2.2.1) are  constant on intervals of length $2^{n-s}$, whereas for each $s+1\leq \nu\leq n$, the restriction of $r^{(n),\nu}$ to any such constancy interval is the Rademacher vector $r^{(n-s),\nu-s}$. By \eqv(2.1.6) and \eqv(1.lem1.1) of Lemma \thv(2.lem1) (i), this implies that the right-hand side of \eqv(2.2.1) is zero. 
(This last argument will be used many times in the proof of Proposition \thv(2.prop1). Given here without detail, it is presented in extenso in the proof of \eqv(2.prop1.12).) The proof of assertion (i) of Lemma \thv(2.lem3) is now complete. 
Assertion (ii) is a special case of Lemma \thv(2.lem4), (i).

We now turn to the construction of asymmetric solutions.

\subsection{Proof of Proposition \thv(2.prop1)} 
    \TH(S6.2)

The proof of assertion (i) is divided into three main steps.

\vspace{.3truecm}
\paragraph{First step.} 
Take $s=n$, $\ell=2$ and let  $(n_1,n_2)$ be a $2$-composition of $n$ such that $n_1$ is even and $n_2$  has arbitrary parity.  
We look for solutions of the system \eqv(2.2.1) of the form 
\be
m=a_1\boldsymbol{1}_{n_1}\oplus a_2\boldsymbol{1}_{n_2},
\Eq(2.prop1.3)
\ee
for some strictly positive numbers $a_1, a_2>0$. By Lemma \thv(2.lem2),
\be
R^{(n)}
=
\begin{pmatrix} 
2^{n_2}\otimes R^{(n_1)}
\\ 
R^{(n_2)}\oplus\dots\oplus R^{(n_2)}
\end{pmatrix},
\Eq(2.prop1.4)
\ee
where the concatenation is over $2^{n_1}$ matrices $R^{(n_2)}$. The Rademacher configurations (or rows) of $R^{(n)}$ 
thus fall into two groups: the first $n_1$, $\{r^{(n),\nu}, 1\leq \nu\leq n_1\}$, are piecewise constant over intervals of length $2^{n_2}$, while the remaining $n_2$, $\{r^{(n),\nu}, n_1+1\leq \nu\leq n_1+n_2\}$, have the property that, restricted to any such constancy interval, they reduce to the $n_2$ Rademacher configurations $\{r^{(n_2),\nu-n_1}, n_1+1\leq \nu\leq n_1+n_2\}$ of $R^{(n_2)}$.

Recall that $d(n)=2^n$. Let $Z^{n}_{n_1}$ be the set defined, using the first $n_1$ configurations of $R^{(n)}$, through
\be
Z^{n}_{n_1}
=
\left\{
1\leq i\leq d(n) : \sum_{1\leq \nu\leq n_1}r^{(n),\nu}_i=0,
\right\},
\Eq(2.prop1.5)
\ee
and let $(Z^{n}_{n_1})^c$ denote its complement,
\be
\{1,\dots,d(n)\}=(Z^{n}_{n_1})^c\cup Z^{n}_{n_1}.
\Eq(2.prop1.6)
\ee
Since $n_1$ is even, the set $Z^{n}_{n_1}$ is non-empty. In view of \eqv(2.prop1.4), it decomposes into ${{n_1}\choose{n_1/2}}$ intervals of constancy of length $2^{n_2}$ and has total lenght
\be
\left|Z^{n}_{n_1}\right|
=2^{n_2}{{n_1}\choose{n_1/2}}.
\Eq(2.prop1.7)
\ee
For each $1\leq\nu\leq n$, we use $Z^{n}_{n_1}$ to split the sum in the right-hand side of \eqv(2.2.1) into two terms,
\be
\begin{split}
\SS^{(n),\nu}_{n_1}&\equiv\frac{1}{d(n)}\sum_{1\leq i\leq d(n)}\1_{\left\{i\in (Z^{n}_{n_1})^c\right\}} r^{(n),\nu}_i
\sign\left[
\left(r^{(n)}_i,m\right)
\right],
\\
\overline\SS^{(n),\nu}_{n_1}
&\equiv\frac{1}{d(n)}\sum_{1\leq i\leq d(n)}\1_{\left\{i\in Z^{n}_{n_1}\right\}} r^{(n),\nu}_i
\sign\left[
\left(r^{(n)}_i,m\right)
\right],
\end{split}
\Eq(2.prop1.7bis)
\ee
and treat two these terms separately. 

Consider first $\SS^{(n),\nu}_{n_1}$. Let $m$ is given by \eqv(2.prop1.3) and assume that 
\be
2a_1>n_2a_2.
\Eq(2.prop1.8)
\ee
This implies that for all $i\in (Z^{n}_{n_1})^c$
\be
\textstyle
a_1\left|\sum_{1\leq \nu\leq n_1}r^{(n),\nu}_i\right|\geq 2a_1>n_2a_2
\geq a_2\left|\sum_{n_1+1\leq \nu\leq n_1+n_2}r^{(n),\nu}_i\right|.
\Eq(2.prop1.8')
\ee
Note that the first inequality of \eqv(2.prop1.8') relies on the fact that  $n_1$ is even.
Thus, the sign of $\bigl(r^{(n)}_i,m\bigr)$ is determined by the first $n_1$ Rademacher configurations, namely, for all 
$i\in (Z^{n}_{n_1})^c$
\be
\sign\left[
\left(r^{(n)}_i,m\right)
\right]
=
\sign\left[
\left(r^{(n),\nu}_i, a_1\boldsymbol{1}_{n_1}\oplus \boldsymbol{0}_{n_2}\right)
\right]
=
\sign\left(
a_1 {
\sum_{1\leq\nu\leq n_1}}r^{(n),\nu}_i
\right).
\Eq(2.prop1.9)
\ee
We now distinguish two cases, $1\leq \nu\leq n_1$ and $n_1+1\leq \nu\leq n_1+n_2$. 
If $1\leq \nu\leq n_1$, then we have
\be
\begin{split}
\SS^{(n),\nu}_{n_1}
=&
\frac{1}{d(n)}\sum_{1\leq i\leq d(n)}r^{(n),\nu}_i\1_{\left\{i\in (Z^{n}_{n_1})^c\right\}}
\sign\left(
a_1 { 
\sum_{1\leq\nu\leq n_1}}r^{(n),\nu}_i
\right)
\\
=&
\frac{1}{d(n)}\sum_{1\leq i\leq d(n)}r^{(n),\nu}_i
\sign\left(
a_1 {
\sum_{1\leq\nu\leq n_1}}r^{(n),\nu}_i
\right)
\\
=&\frac{1}{d(n_1)}\sum_{1\leq i\leq d(n_1)}r^{(n_1),\nu}_i
\sign\left[a_1
\left(r^{(n_1)}_i,\boldsymbol{1}_{n_1}\right)
\right].
\end{split}
\Eq(2.prop1.10)
\ee
The first equality in \eqv(2.prop1.10) follows from the definition of $\SS^{(n),\nu}_{n_1}$ and \eqv(2.prop1.9).  To go from the first to the second, we used  that by \eqv(2.prop1.5) and \eqv(2.prop1.6), the sum of the terms satisfying  the condition 
$\left\{i\in Z^{n}_{n_1}\right\}$ vanishes. The last equality then follows from \eqv(2.prop1.4). Note that since $a_1>0$, it can be removed. Doing so, the last line of \eqv(2.prop1.10) reduces to the right-hand side of the mixed-memory equation \eqv(2.2.1) for the $n_1$-symmetric solution (see the more explicit formula \eqv(2.lem1.2bis) in the proof of Lemma \thv(2.lem3)).
Thus, by Lemma \thv(2.lem3), (i), for all  even integer $n_1\geq 2$, all $a_1>0$  satisfying \eqv(2.prop1.8) and all 
$1\leq \nu\leq n_1$, we have
\be
\SS^{(n),\nu}_{n_1}
=
\a^{(n_1)}.
\Eq(2.prop1.11bis)
\ee

Now, let us check that if $n_1+1\leq \nu\leq n_1+n_2$, then
\be
\SS^{(n),\nu}_{n_1}
=
0.
\Eq(2.prop1.12)
\ee
This uses the following two facts. Firstly, as we just saw,  the sequence of signs \eqv(2.prop1.9) is piecewise constant over intervals of length $d(n_2)$,
\be
\begin{split}
&
\left(
\sign\left[
\left(r^{(n)}_i,m\right)
\right]
\right)_{i\in\left(Z^{n}_{n_1}\right)^c}
=d(n_2)\otimes \left(
\sign\left[
\left(r^{(n_1)}_j,a_1\boldsymbol{1}_{n_1}\right)
\right]
\right)_{\left\{j\in \left(Z^{n_1}_{n_1}\right)^c\right\}},
\end{split}
\Eq(2.prop1.12bis)
\ee
where  in the right-hand side, $\left(Z^{n_1}_{n_1}\right)^c=\left\{1,\dots,d(n_1)\right\}\setminus Z^{n_1}_{n_1}$ and
\be
Z^{n_1}_{n_1}
=
\left\{
1\leq i\leq d(n_1) : \sum_{1\leq \nu\leq n_1}r^{(n_1),\nu}_i=0
\right\}.
\Eq(2.prop1.14)
\ee
Note that \eqv(2.prop1.14) is obtained by taking $n=n_1$ in \eqv(2.prop1.5). Secondly, by \eqv(2.prop1.4), the configurations 
$\{r^{(n),\nu}, n_1+1\leq \nu\leq n_1+n_2\}$, restricted to any such constancy interval, reduce to the $n_2$ Rademacher configurations $\{r^{(n_2),\nu-n_1}, n_1+1\leq \nu\leq n_1+n_2\}$ of $R^{(n_2)}$.
By the first of these two facts, 
\be
\begin{split}
&
\SS^{(n),\nu}_{n_1}
\\
=
&
\frac{1}{d(n)}\sum_{1\leq j\leq d(n_1)}\,\sum_{(j-1)d(n_2)+1\leq i\leq j d(n_2)}r^{(n),\nu}_i\1_{\left\{i\in (Z^{n}_{n_1})^c\right\}}
\sign\left[
\left(r^{(n)}_i,m\right)
\right]
\\
=
&
\frac{1}{d(n)}\sum_{1\leq j\leq d(n_1)}\1_{\left\{j\in \left(Z^{n_1}_{n_1}\right)^c\right\}}
\sign\left[
\left(r^{(n_1)}_j,a_1
\boldsymbol{1}_{n_1}\right)
\right]
\sum_{(j-1)d(n_2)+1\leq i\leq j d(n_2)}r^{(n),\nu}_i,
\end{split}
\Eq(2.prop1.13)
\ee
and by the second,
\be
\sum_{(j-1)d(n_2)+1\leq i\leq j d(n_2)}r^{(n),\nu}_i
=
\sum_{1\leq i\leq d(n_2)}r^{(n_2),\nu-n_1}_i
=\left(r^{(n_2),\nu-n_1}, r^{(n_2),0}\right)
=0,
\Eq(2.prop1.15bis)
\ee
where the last equality is \eqv(1.lem1.1) of Lemma \thv(2.lem1), (i). Inserting \eqv(2.prop1.15bis) into \eqv(2.prop1.13) then yields the claim of  \eqv(2.prop1.12).

We now turn to the term $\overline\SS^{(n),\nu}_{n_1}$ in \eqv(2.prop1.7bis). For all $1\leq\nu\leq n$, by definition of 
$Z^{n}_{n_1}$
\be
\begin{split}
&\overline\SS^{(n),\nu}_{n_1}
\\
=&\frac{1}{d(n)}\sum_{1\leq i\leq d(n)}\1_{\left\{i\in Z^{n}_{n_1}\right\}} r^{(n),\nu}_i
\sign\left[
\left(r^{(n)}_i,\boldsymbol{0}_{n_1}\oplus a_2\boldsymbol{1}_{n_2}\right)
\right]
\\
=&\frac{1}{d(n)}\sum_{1\leq j\leq d(n_1)}\,\sum_{(j-1)d(n_2)+1\leq i\leq j d(n_2)}r^{(n),\nu}_i\1_{\left\{i\in Z^{n}_{n_1}\right\}}
\sign\left[
\left(r^{(n)}_i,\boldsymbol{0}_{n_1}\oplus a_2\boldsymbol{1}_{n_2}\right)
\right]
\\
=&\frac{1}{d(n_1)}\sum_{1\leq j\leq d(n_1)}\1_{\left\{j\in Z^{n_1}_{n_1}\right\}}\UU^{(n),\nu}_{n_2, j},
\end{split}
\Eq(2.prop1.15)
\ee
where $Z^{n_1}_{n_1}$ is defined in \eqv(2.prop1.14) and where, for each $j\in Z^{n_1}_{n_1}$ and $1\leq\nu\leq n$,
\be
\UU^{(n),\nu}_{n_2, j}
\equiv
\frac{1}{d(n_2)}\sum_{(j-1)d(n_2)+1\leq i\leq j d(n_2)}
r^{(n),\nu}_i
\sign\left[
\left(r^{(n)}_i,\boldsymbol{0}_{n_1}\oplus a_2\boldsymbol{1}_{n_2}\right)
\right].
\Eq(2.prop1.16)
\ee
As before, we distinguish two cases. First, if $n_1+1\leq\nu\leq n_1+n_2$ then, by \eqv(2.prop1.4),
$r^{(n),\nu}$ is the concatenation of $d(n_1)$ identical copies of $r^{(n_2),\nu}$, and so,
\be
\begin{split}
\UU^{(n),\nu}_{n_2, j}
&
=
\frac{1}{d(n_2)}\sum_{1\leq i\leq d(n_2)}
r^{(n_2),\nu}_i
\sign\left[
a_2\left(r^{(n_2)}_i, \boldsymbol{1}_{n_2}\right)
\right].
\end{split}
\Eq(2.prop1.17)
\ee
After removing the quantity $a_2$, which is possible since it is strictly positive, we again recognise in \eqv(2.prop1.17) the right-hand side of the mixed-memory equation of the  $n_2$-asymmetric solution (see \eqv(2.lem1.2bis) in the proof of Lemma \thv(2.lem3)). Thus, by Lemma \thv(2.lem3), (i), we have for all $a_2>0$,  all $j\in Z^{n_1}_{n_1}$ and all  
$n_1+1\leq\nu\leq n_1+n_2$
\be
\UU^{(n),\nu}_{n_2, j}
=
\a^{(n_2)}.
\Eq(2.prop1.18bis)
\ee
Inserting \eqv(2.prop1.18bis) into \eqv(2.prop1.15) and using the last equality of \eqv(2.prop1.7) to evaluate $|Z^{n_1}_{n_1}|$, we get
\be
\begin{split}
\overline\SS^{(n),\nu}_{n_1}
&
=\frac{\a^{(n_2)}}{d(n_1)}\sum_{1\leq j\leq d(n_1)}\1_{\left\{j\in Z^{n_1}_{n_1}\right\}}
=\frac{\a^{(n_2)}}{d(n_1)}|Z^{n_1}_{n_1}|
=\frac{\a^{(n_2)}}{d(n_1)}{{n_1}\choose{n_1/2}}
\\
&
=\a^{(n_1+1)}\a^{(n_2)}.
\end{split}
\Eq(2.prop1.19)
\ee
Second, if $1\leq\nu\leq n_1$, then by \eqv(2.prop1.4) $r^{(n),\nu}$ is piecewise constant over intervals of length $d(n_2)$, \emph{i.e.}~$r^{(n),\nu}_i=r^{(n_1),\nu}_j$ for all $i$ in the interval $(j-1)d(n_2)+1\leq i\leq j d(n_2)$. Thus,
\be
\begin{split}
\UU^{(n),\nu}_{n_2, j}
&
=
r^{(n_1),\nu}_j
\frac{1}{d(n_2)}\sum_{(j-1)d(n_2)+1\leq i\leq j d(n_2)}
\sign\left[
\left(r^{(n)}_i,\boldsymbol{0}_{n_1}\oplus a_2\boldsymbol{1}_{n_2}\right)
\right]
\\
&
=
r^{(n_1),\nu}_j
\frac{1}{d(n_2)}\sum_{1\leq i\leq d(n_2)}
\sign\left[
a_2
\left(r^{(n_2)}_i,\boldsymbol{1}_{n_2}\right)
\right]
\\
&
=0,
\end{split}
\Eq(2.prop1.20)
\ee
where we used that due to the axial symmetry \eqv(1.lem1.2) of Lemma \thv(2.lem1), (ii), the sum in the second line of \eqv(2.prop1.20) is zero. Plugging \eqv(2.prop1.20) into \eqv(2.prop1.15), we get that for all $1\leq\nu\leq n_1$ 
\be
\overline\SS^{(n),\nu}_{n_1}=0.
\Eq(2.prop1.21)
\ee

We can now collect our estimates.
Combining \eqv(2.prop1.11bis), \eqv(2.prop1.12), \eqv(2.prop1.19) and \eqv(2.prop1.21), we obtain
\be
\SS^{(n),\nu}_{n_1}
=
\begin{cases}
\a^{(n_1)} & \text{if $1\leq \nu\leq n_1$}, \\
0&  \text{if $n_1+1\leq\nu\leq n_1+n_2$},
\end{cases}
\Eq(2.prop1.21')
\ee
and
\be
\overline\SS^{(n),\nu}_{n_1}
=
\begin{cases}
0 & \text{if $1\leq \nu\leq n_1$}, \\
\a^{(n_1+1)} \a^{(n_2)}&  \text{if $n_1+1\leq\nu\leq n_1+n_2$}.
\end{cases}
\Eq(2.prop1.21'')
\ee
Recall that we seek solutions to the system \eqv(2.2.1) of the form \eqv(2.prop1.3).
By \eqv(2.prop1.7bis), the right-hand side of \eqv(2.2.1) is equal to the sum of the above two terms,
\be
m_{\nu}
=
\SS^{(n),\nu}_{n_1}+\overline\SS^{(n),\nu}_{n_1},
\Eq(2.prop1.22')
\ee
and so, inserting \eqv(2.prop1.21') and \eqv(2.prop1.21'') in \eqv(2.prop1.22'), we obtain that for all $n_1\geq 2$ even, $n_2\geq 1$ of arbitrary parity and all $a_1, a_2>0$ that satisfy  \eqv(2.prop1.8), 
\be
m_{\nu}
=
\begin{cases}
\a^{(n_1)} & \text{if $1\leq \nu\leq n_1$}, \\
\a^{(n_1+1)} \a^{(n_2)}&  \text{if $n_1+1\leq\nu\leq n_1+n_2$}.
\end{cases}
\Eq(2.prop1.22)
\ee
Thus, observing that for $n_1$ even, $\a^{(n_1+1)}=\a^{(n_1)}$ (see also \eqv(2.rem2.1)-\eqv(2.rem2.2)),
and choosing
\be
\begin{split}
a_1
&
=\g^{(1)}\equiv \a^{(n_1)},
\\
a_2
&
=\g^{(2)}\equiv \a^{(n_1)} \a^{(n_2)},
\end{split}
\Eq(2.prop1.23)
\ee
the vector $m=\g^{(1)}\boldsymbol{1}_{n_1}\oplus \g^{(2)}\boldsymbol{1}_{n_2}$ verifies the mixed-memory equation \eqv(2.2.1), provided that $n_1\geq 2$ is even, $n_2\geq 1$ has any parity and 
\be
2\g^{(1)}>n_2\g^{(2)}.
\Eq(2.prop1.25)
\ee
This proves Proposition \thv(2.prop1), (i),  in the case of an asymmetric solution constructed from two symmetric solutions.

\smallskip
\paragraph{General construction step.} We now take  $s=n$, $\ell>2$ and assume that $(n_1,\dots, n_{\ell})$ is  an $\ell$-composition   of $n$  such that $n_1,\dots, n_{\ell-1}\geq 2$ are even and $n_{\ell}\geq 1$ has any parity. 
We seek solutions of the system \eqv(2.2.1) of the form 
\be
m=a_1\boldsymbol{1}_{n_1}\oplus \dots\oplus a_{\ell}\boldsymbol{1}_{n_{\ell}},
\Eq(2.prop1.55)
\ee
for some numbers $a_1,\dots, a_{\ell}>0$ satisfying the system of conditions
\be
2a_k>n_{k+1}a_{k+1}+\dots+n_{\ell}a_{\ell},\,\,\, \text{for all}\,\,\, 1\leq k\leq \ell-1.
\Eq(2.prop1.56)
\ee

To formulate the general construction step, the following definitions are needed. The first one extends the definition 
\eqv(2.prop1.5) to each of the $1\leq k\leq \ell-1$ groups of configurations of $R^{(n)}$. Using the $k$-th group, namely, the configurations $r^{(n),\nu}$ with  $n_0+\dots+n_{k-1}+1\leq \nu\leq n_1+\dots+n_{k}$, we let  
$Z^{n}_{n_k}$ be the set 
\be
Z^{n}_{n_k}
=
\left\{
1\leq i\leq d(n) : 
\sum_{1\leq \nu-(n_0+\dots+n_{k-1})\leq n_{k}}
r^{(n),\nu}_i=0
\right\},
\quad 1\leq k\leq \ell-1,
\Eq(2.prop1.29)
\ee
with the convention that $n_0=0$.
Since $n_k$ is even for all $1\leq k\leq \ell-1$, these sets are non-empty. We denote by $\left(Z^{n}_{n_k}\right)^c$ their complements,
\be
\left(Z^{n}_{n_k}\right)^c
=
\left\{1, \dots, d(n)\right\}\setminus Z^{n}_{n_k}, \quad 1\leq k\leq \ell-1.
\Eq(2.prop1.29bis)
\ee
These sets are then used to decompose the sum in the right-hand side of \eqv(2.2.1),  for each $1\leq\nu\leq n$,  into $\ell$ terms:
\be
\SS^{(n),\nu}_{n_1}\equiv\frac{1}{d(n)}\sum_{1\leq i\leq d(n)}\1_{\left\{i\in \left(Z^{n}_{n_1}\right)^c\right\}} r^{(n),\nu}_i
\sign\left[
\left(r^{(n)}_i,m\right)
\right],
\Eq(2.prop1.57)
\ee
for all $2\leq k\leq \ell-1$,
\be
\SS^{(n),\nu}_{n_1,\dots,n_k}\equiv\frac{1}{d(n)}\sum_{1\leq i\leq d(n)}
\1_{\left\{i\in Z^{n}_{n_1}\cap\dots \cap Z^{n}_{n_{k-1}} \cap{\left(Z^{n}_{n_k}\right)^c}\right\}} r^{(n),\nu}_i
\sign\left[
\left(r^{(n)}_i,m\right)
\right],
\Eq(2.prop1.58)
\ee
and
\be
\overline\SS^{(n),\nu}_{n_1,\dots,n_{\ell-1}}\equiv\frac{1}{d(n)}\sum_{1\leq i\leq d(n)}
\1_{\left\{i\in Z^{n}_{n_1}\cap \dots\cap Z^{n}_{n_{\ell-1}}\right\}} r^{(n),\nu}_i
\sign\left[
\left(r^{(n)}_i,m\right)
\right].
\Eq(2.prop1.59)
\ee
With these definitions, the system of equations \eqv(2.2.1) can be rewritten as
\be
m_{\nu}
=
\sum_{k=1}^{\ell-1} \SS^{(n),\nu}_{n_1,\dots,n_k} + \overline\SS^{(n),\nu}_{n_1,\dots,n_{\ell-1}}
\quad 1\leq\nu\leq n.
\Eq(2.2.1bis)
\ee

To construct solutions to this system, we evaluate each of the sums separately, except for the last two 
(\emph{i.e.},~\eqv(2.prop1.58) whith $k=\ell-1$ and \eqv(2.prop1.59)), which we treat simultaneously. The gist of the proof is that for each $k$,  the problem of evaluating the sum  $\SS^{(n),\nu}_{n_1,\dots,n_k}$  can be reduced to a situation analogous to the one we encountered in the first step. As will become clear later, there is little difference in the treatment of  $\SS^{(n),\nu}_{n_1,\dots,n_k}$ for 
$1\leq k\leq \ell-2$ and $k=\ell-1$. We therefore start by considering the case  $k=\ell -1$ and evaluate the last two sums,  
$\SS^{(n),\nu}_{n_1,\dots,n_{\ell-1}}$ and $\overline\SS^{(n),\nu}_{n_1,\dots,n_{\ell-1}}$.

\smallskip
\emph{\textbf{The case $k=\ell -1$.}}
 More precisely, let us establish that under the condition that
\be
2 a_{\ell-1}>n_{\ell}a_{\ell},
\Eq(2.prop1.82)
\ee
for all $n_{\ell-1}\geq 2$ even and $n_{\ell}\geq 1$ of arbitrary parity, we have
\be
\begin{split}
\SS^{(n),\nu}_{n_1,\dots,n_{\ell-1}}
=&
\begin{cases}
0& \text{if $1\leq \nu\leq n_1+\dots +n_{\ell-2}$}, \\
\a^{(n_{\ell-1})}\prod_{l=1}^{\ell-2}\a^{(n_{l}+1)}& \text{if $1\leq\nu-(n_1+\dots +n_{\ell-2})\leq n_{\ell-1}$}, \\
0&  \text{if $1\leq\nu-(n_1+\dots +n_{\ell-1})\leq n_{\ell}$},
\end{cases}
\end{split}
\Eq(2.prop1.63')
\ee
and
\be
\begin{split}
\overline\SS^{(n),\nu}_{n_1,\dots,n_{\ell-1}}
=&
\begin{cases}
0& \text{if $1\leq \nu\leq n_1+\dots +n_{\ell-1}$}, \\
\a^{(n_{\ell})}\prod_{l=1}^{\ell-1}\a^{(n_{l}+1)}&  \text{if $1\leq\nu-(n_1+\dots +n_{\ell-1})\leq n_{\ell}$.}
\end{cases}
\end{split}
\Eq(2.prop1.63'')
\ee
For this, note that $\SS^{(n),\nu}_{n_1,\dots,n_{\ell-1}}$ and $\overline\SS^{(n),\nu}_{n_1,\dots,n_{\ell-1}}$ are functions of a sequence of signs,
\be
\begin{split}
&
\left(
\sign\left[
\left(r^{(n)}_i,m\right)
\right]
\right)_{\left\{i\in Z^{n}_{n_1}\cap \dots\cap Z^{n}_{n_{\ell-2}}\right\}}
\\
=&
\left(
\sign\left[
\left(r^{(n)}_i,\boldsymbol{0}_{n_1+\dots+n_{\ell-2}}
\oplus a_{\ell-1}\boldsymbol{1}_{n_{\ell-1}}\oplus a_{\ell}\boldsymbol{1}_{n_{\ell}}\right)
\right]
\right)_{\left\{i\in Z^{n}_{n_1}\cap \dots\cap Z^{n}_{n_{\ell-2}}\right\}},
\end{split}
\Eq(2.prop1.64)
\ee
which, by definition of the sets $Z^{n}_{n_k}$, no longer depends on the first $n_1+\dots+n_{\ell-2}$ configurations of $R^{(n)}$.
By  Lemma \thv(2.lem2), we can write  $R^{(n)}$ using a 2-level tree as
\be
R^{(n)}
=
\begin{pmatrix} 
2^{(n_{\ell-1}+n_{\ell})}\otimes R^{(n_1+\dots+n_{\ell-2})}
\\ 
R^{(n_{\ell-1}+n_{\ell})}\oplus\dots\oplus R^{(n_{\ell-1}+n_{\ell})} 
\end{pmatrix},
\Eq(2.prop1.65)
\ee
where the concatenation is over $2^{(n_1+\dots+n_{\ell-2})}$ matrices $R^{(n_{\ell-1}+n_{\ell})}$. Thus, the matrix obtained from $R^{(n)}$ by removing the first $n_1+\dots+n_{\ell-2}$ configurations  is a concatenation of $2^{(n_1+\dots+n_{\ell-2})}$ identical matrices $R^{(n_{\ell-1}+n_{\ell})}$. Reasoning as  in the first step  (see in particular \eqv(2.prop1.15)-\eqv(2.prop1.16)),  the terms 
$\SS^{(n),\nu}_{n_1,\dots,n_{\ell-1}}$ and $\overline\SS^{(n),\nu}_{n_1,\dots,n_{\ell-1}}$ can be rewritten as follows. 
The role of the set $Z^{n_1}_{n_1}$  is now played by the set
\be
\begin{split}
Z^{n_1+\dots+n_{\ell-2}}_{n_1,\dots,n_{\ell-2}}
\equiv &
\bigcap_{1\leq k\leq \ell-2}Z^{n_1+\dots+n_{\ell-2}}_{n_k}
\\
\equiv &
\Bigg\{
1\leq i\leq  d(n_1+\dots+n_{\ell-2}) : \textrm{for all}\,\,  1\leq k\leq \ell-2
\\
&\quad\quad\quad\quad\quad\quad\quad\quad\quad\,\,\, \sum_{1\leq \nu-(n_0+\dots+n_{k-1})\leq n_{k}}
r^{(n_1+\dots+n_{\ell-2}),\nu}_i=0
\Bigg\},
\end{split}
\Eq(2.prop1.67)
\ee
with the convention that $n_0=0$. Given 
$
j\in Z^{n_1+\dots+n_{\ell-2}}_{n_1,\dots,n_{\ell-2}}
$ 
and a subset 
$
\ZZ\subseteq\left\{1,\dots,d(n_{\ell-1}+n_{\ell})\right\}
$, 
define, for all $1\leq\nu\leq n$
\be
\begin{split}
&\UU^{(n),\nu}_{n_{\ell-1}+n_{\ell},j}(\ZZ)
\\
\equiv &\,\,
\frac{1}{d(n_{\ell-1}+n_{\ell})}\sum_{(j-1)d(n_{\ell-1}+n_{\ell})+1\leq i\leq j d(n_{\ell-1}+n_{\ell})}
\1_{\left\{i-(j-1)d(n_{\ell-1}+n_{\ell})\in \ZZ\right\}}
\\
&\,\,
r^{(n),\nu}_i
\sign\left[
\left(r^{(n)}_i,
\boldsymbol{0}_{n_1+\dots+n_{\ell-2}}\oplus a_{\ell-1}\boldsymbol{1}_{n_{\ell-1}}\oplus a_{\ell}\boldsymbol{1}_{n_{\ell}}\right)
\right].
\end{split}
\Eq(2.prop1.67bis)
\ee
and
\be
\begin{split}
&\VV^{(n_{\ell-1}+n_{\ell}),\nu}(\ZZ)
\\
\equiv
&\,\, \frac{1}{d(n_1+\dots+n_{\ell-2})}\sum_{1\leq j\leq d(n_1+\dots+n_{\ell-2})}
\1_{\left\{j\in  Z^{n_1+\dots+n_{\ell-2}}_{n_1,\dots,n_{\ell-2}}\right\}}
\UU^{(n),\nu}_{n_{\ell-1}+n_{\ell},j}\left(\ZZ\right).
\end{split}
\Eq(2.prop1.68)
\ee
Then, for $1\leq\nu\leq n$,
\be
\begin{split}
\SS^{(n),\nu}_{n_1,\dots,n_{\ell-1}}&=\VV^{(n_{\ell-1}+n_{\ell}),\nu}\left(\bigl(Z^{n_{\ell-1}+n_{\ell}}_{n_{\ell-1}}\bigr)^c\right),
\\
\overline\SS^{(n),\nu}_{n_1,\dots,n_{\ell-1}}&=\VV^{(n_{\ell-1}+n_{\ell}),\nu}\left(Z^{n_{\ell-1}+n_{\ell}}_{n_{\ell-1}}\right),
\end{split}
\Eq(2.prop1.69)
\ee
where $Z^{n_{\ell-1}+n_{\ell}}_{n_{\ell-1}}$  is the set
\be
Z^{n_{\ell-1}+n_{\ell}}_{n_{\ell-1}}
=
\left\{
1\leq i'\leq d(n_{\ell-1}+n_{\ell}) : 
\sum_{1\leq \nu'\leq n_{\ell-1}}
r^{(n_{\ell-1}+n_{\ell}),\nu'}_{i'}=0
\right\}
\Eq(2.prop1.70)
\ee
obtained by taking $n=n_{\ell-1}+n_{\ell}$ and $k=\ell-1$ in  \eqv(2.prop1.29).

As can be seen from the right-hand sides of \eqv(2.prop1.63') and \eqv(2.prop1.63''), we now must distinguish between several cases. Let us start with the case $1\leq \nu\leq n_1+\dots +n_{\ell-2}$. For such $\nu$, it follows from \eqv(2.prop1.65) that for each 
$
1\leq j\leq d(n_1+\dots+n_{\ell-2})
$,
$
r^{(n),\nu}_i=r^{(n_1+\dots+n_{\ell-2}),\nu}_j
$
for all $i$ in the constancy interval
$
(j-1)d(n_{\ell-1}+n_{\ell})+1\leq i\leq j d(n_{\ell-1}+n_{\ell})
$.
Thus, $r^{(n),\nu}_i$ can be taken out from the sum over $i$ in \eqv(2.prop1.67bis).
Consequently, for any subset 
$\ZZ\subseteq\left\{1,\dots,d(n_{\ell-1}+n_{\ell})\right\}$, $\VV^{(n_{\ell-1}+n_{\ell}),\nu}\left(\ZZ\right)$ factorises into
\be
\VV^{(n_{\ell-1}+n_{\ell}),\nu}\left(\ZZ\right)=\XX^{(n_1+\dots+n_{\ell-2}),\nu}\,\YY^{(n_{\ell-1}+n_{\ell}),\nu}\left(\ZZ\right),
\Eq(2.prop1.71)
\ee
where 
\be
\begin{split}
&
\YY^{(n_{\ell-1}+n_{\ell}),\nu}\left(\ZZ\right) 
\\
\equiv
&
\frac{1}{d(n_{\ell-1}+n_{\ell})}\sum_{1\leq i\leq d(n_{\ell-1}+n_{\ell})}\1_{\left\{i\in \ZZ\right\}}
\sign\left[
\left(r^{(n_{\ell-1}+n_{\ell})}_i,
\oplus a_{\ell-1}\boldsymbol{1}_{n_{\ell-1}}\oplus a_{\ell}\boldsymbol{1}_{n_{\ell}}\right)
\right],
\end{split}
\Eq(2.prop1.72)
\ee
and
\be
\begin{split}
&
\XX^{(n_1+\dots+n_{\ell-2}),\nu}
\\
\equiv
&
\frac{1}{d(n_1+\dots+n_{\ell-2})}\sum_{1\leq j\leq d(n_1+\dots+n_{\ell-2})}
\1_{\left\{j\in  Z^{n_1+\dots+n_{\ell-2}}_{n_1,\dots,n_{\ell-2}}\right\}}
r^{(n_1+\dots+n_{\ell-2}),\nu}_j.
\end{split}
\Eq(2.prop1.73)
\ee
To evaluate $\XX^{(n_1+\dots+n_{\ell-2}),\nu}$,  we use that, in the light of  the definition \eqv(2.prop1.67), 
the axial symmetry of Lemma \thv(2.lem1), (ii), is preserved on  $Z^{n_1+\dots+n_{\ell-2}}_{n_1,\dots,n_{\ell-2}}$, 
namely, if  $j\in Z^{n_1+\dots+n_{\ell-2}}_{n_1,\dots,n_{\ell-2}}$ then  
$2^{n_1+\dots+n_{\ell-2}}-j+1\in Z^{n_1+\dots+n_{\ell-2}}_{n_1,\dots,n_{\ell-2}}$ and
\be
r^{(n_1+\dots+n_{\ell-2}),\nu}_j=-r^{(n_1+\dots+n_{\ell-2}),\nu}_{2^{n_1+\dots+n_{\ell-2}}-j+1}.
\Eq(2.prop1.74)
\ee
Therefore, $\XX^{(n_1+\dots+n_{\ell-2}),\nu}=0$. By \eqv(2.prop1.71), this implies that
$\VV^{(n_{\ell-1}+n_{\ell}),\nu}(\ZZ)=0$ for all $\ZZ$ and all  $1\leq \nu\leq n_1+\dots+n_{\ell-2}$,
and so, by \eqv(2.prop1.69), 
 for all $1\leq \nu\leq n_1+\dots+n_{\ell-2}$
\be
\begin{split}
\SS^{(n),\nu}_{n_1,\dots,n_{\ell-1}}&=0,\\
\overline\SS^{(n),\nu}_{n_1,\dots,n_{\ell-1}}&=0.
\end{split}
\Eq(2.prop1.75)
\ee

We now focus on the case $n_1+\dots +n_{\ell-2}+1\leq\nu\leq n_1+\dots +n_{\ell}$.
Again, in the light of \eqv(2.prop1.65), we have for such $\nu$   that
\be
\UU^{(n),\nu}_{n_{\ell-1}+n_{\ell},j}(\ZZ)
=
\UU^{(n),\nu}_{n_{\ell-1}+n_{\ell},1}(\ZZ)
\quad \text{for all}~ j\in Z^{n_1+\dots+n_{\ell-2}}_{n_1,\dots,n_{\ell-2}},
\Eq(2.prop1.76)
\ee
and
\be
\begin{split}
\UU^{(n),\nu}_{n_{\ell-1}+n_{\ell},1}(\ZZ)
=&
\frac{1}{d(n_{\ell-1}+n_{\ell})}\sum_{1\leq i\leq d(n_{\ell-1}+n_{\ell})}
\1_{\left\{i\in \ZZ\right\}}
\\
&
 r^{(n_{\ell-1}+n_{\ell}),\nu}_i
\sign\left[
\left(r^{(n_{\ell-1}+n_{\ell})}_i, 
a_{\ell-1}\boldsymbol{1}_{n_{\ell-1}}\oplus a_{\ell}\boldsymbol{1}_{n_{\ell}}
\right)
\right].
\end{split}
\Eq(2.prop1.77)
\ee
Therefore, \eqv(2.prop1.68) becomes
\be
\VV^{(n_{\ell-1}+n_{\ell}),\nu}(\ZZ)
=
\frac{\bigl |Z^{n_1+\dots+n_{\ell-2}}_{n_1,\dots,n_{\ell-2}}\bigr|}{d(n_1+\dots+n_{\ell-2})}
\,\UU^{(n),\nu}_{n_{\ell-1}+n_{\ell},1}(\ZZ),
\Eq(2.prop1.78)
\ee
where, by the definition \eqv(2.prop1.67),
\be
\frac{\bigl |Z^{n_1+\dots+n_{\ell-2}}_{n_1,\dots,n_{\ell-2}}\bigr|}{d(n_1+\dots+n_{\ell-2})}
=
\prod_{k=1}^{\ell-2}\frac{1}{d(n_k)}{{n_k}\choose{n_k/2}}
=
\prod_{k=1}^{\ell-2}\a^{(n_k+1)}.
\Eq(2.prop1.81)
\ee
Next, remembering \eqv(2.prop1.70) and comparing  \eqv(2.prop1.77) with \eqv(2.prop1.7bis), we have
\be
\begin{split}
\UU^{(n),\nu}_{n_{\ell-1}+n_{\ell},1}\left(\bigl(Z^{n_{\ell-1}+n_{\ell}}_{n_{\ell-1}}\bigr)^c\right)
&=
\SS^{(n_{\ell-1}+n_{\ell}),\nu}_{n_{\ell-1}},
\\
\UU^{(n),\nu}_{n_{\ell-1}+n_{\ell},1}\left(Z^{n_{\ell-1}+n_{\ell}}_{n_{\ell-1}}\right)
&=
\overline\SS^{(n_{\ell-1}+n_{\ell}),\nu}_{n_{\ell-1}}.
\end{split}
\Eq(2.prop1.79)
\ee
Thus, inserting \eqv(2.prop1.81) and \eqv(2.prop1.79) into \eqv(2.prop1.78), it follows from \eqv(2.prop1.69) that
\be
\begin{split}
\SS^{(n),\nu}_{n_1,\dots,n_{\ell-1}}&=
\SS^{(n_{\ell-1}+n_{\ell}),\nu}_{n_{\ell-1}}
\prod_{k=1}^{\ell-2}\a^{(n_k+1)},
\\
\overline\SS^{(n),\nu}_{n_1,\dots,n_{\ell-1}}&=
\overline\SS^{(n_{\ell-1}+n_{\ell}),\nu}_{n_{\ell-1}}
\prod_{k=1}^{\ell-2}\a^{(n_k+1)}.
\end{split}
\Eq(2.prop1.80)
\ee
It remains to note that the sums appearing on the right-hand sides of \eqv(2.prop1.80) have been evaluated in the first step of the proof. Transposed to the present case (\emph{i.e.},~replacing the pair $n, n_1$ of Step 1 by the pair $n_{\ell-1}+n_{\ell}, n_{\ell-1}$), it follows from \eqv(2.prop1.21') and \eqv(2.prop1.21'') that under the condition
\be
2 a_{\ell-1}>n_{\ell}a_{\ell},
\Eq(2.prop1.82')
\ee
for all $n_{\ell-1}\geq 2$ even and $n_{\ell}\geq 1$ of arbitrary parity
\be
\begin{split}
\SS^{(n_{\ell-1}+n_{\ell}),\nu}_{n_{\ell-1}}
=\begin{cases}
\a^{(n_{\ell-1})} & \text{if $1\leq\nu-(n_1+\dots +n_{\ell-2})\leq n_{\ell-1}$}, \\
0&  \text{if $1\leq\nu-(n_1+\dots +n_{\ell-1})\leq n_{\ell}$},
\end{cases}
\end{split}
\Eq(2.prop1.83')
\ee
and
\be
\begin{split}
\overline\SS^{(n_{\ell-1}+n_{\ell}),\nu}_{n_{\ell-1}}
=\begin{cases}
0 & \text{if $1\leq\nu-(n_1+\dots +n_{\ell-2})\leq n_{\ell-1}$}, \\
\a^{(n_{\ell-1}+1)} \a^{(n_{\ell})}&   \text{if $1\leq\nu-(n_1+\dots +n_{\ell-1})\leq n_{\ell}$}.
\end{cases}
\end{split}
\Eq(2.prop1.83'')
\ee
Collecting \eqv(2.prop1.75), \eqv(2.prop1.80),  \eqv(2.prop1.83') and  \eqv(2.prop1.83''), we finally obtain \eqv(2.prop1.63') and \eqv(2.prop1.63'') under the assumption \eqv(2.prop1.82) and for all $n_{\ell-1}\geq 2$ even and $n_{\ell}\geq 1$ of arbitrary parity.
The case $k=\ell-1$ is complete.

\smallskip
\emph{\textbf{The case $1\leq k\leq \ell-2$.}} Let us now establish that for each $1\leq k\leq \ell-2$, under the condition that
\be
2a_k>n_{k+1}a_{k+1}+\dots+n_{\ell}a_{\ell},
\Eq(2.prop1.85')
\ee
where $n_{k}\geq 2$ is even for all $1\leq k\leq \ell-1$ and $n_{\ell}\geq 1$ is of arbitrary parity,  we have
\be
\begin{split}
\SS^{(n),\nu}_{n_1,\dots,n_{k}}
=&
\begin{cases}
0& \text{if $1\leq\nu\leq n_0+\dots +n_{k-1}$}, \\
\a^{(n_k)}\prod_{l=1}^{k-1}\a^{(n_{l}+1)} & \text{if $1\leq\nu-(n_0+\dots +n_{k-1})\leq n_{k}$}, \\
0&  \text{if $n_1+\dots +n_{k}+1\leq\nu\leq n_1+\dots +n_{\ell}$,}
\end{cases}
\end{split}
\Eq(2.prop1.85)
\ee
where $n_0=0$.

Let us first assume that $2\leq k\leq \ell-2$. The proof of \eqv(2.prop1.85) in this case closely follows the proof of \eqv(2.prop1.63') for $k= \ell-1$. For the sake of clarity, we give the explicit definitions of the quantities needed to derive the analogue 
for $\SS^{(n),\nu}_{n_1,\dots,n_{k}}$ of the expression for $\SS^{(n),\nu}_{n_1,\dots,n_{\ell-1}}$
given in  \eqv(2.prop1.69). We first observe that, as in \eqv(2.prop1.64), $\SS^{(n),\nu}_{n_1,\dots,n_{k}}$ is a function of a sequence of signs,
\be
\begin{split}
&
\left(
\sign\left[
\left(r^{(n)}_i,m\right)
\right]
\right)_{\left\{i\in Z^{n}_{n_1}\cap \dots\cap Z^{n}_{n_{k-1}}\right\}}
\\
=&
\left(
\sign\left[
\left(r^{(n)}_i,\boldsymbol{0}_{n_1+\dots+n_{k-1}}
\oplus a_{k}\boldsymbol{1}_{n_{k}}\oplus\dots\oplus a_{\ell}\boldsymbol{1}_{n_{\ell}}\right)
\right]
\right)_{\left\{i\in Z^{n}_{n_1}\cap \dots\cap Z^{n}_{n_{k-1}}\right\}},
\end{split}
\Eq(2.prop1.86)
\ee
which, by definition of the sets $Z^{n}_{n_k}$, does not depend on the first $n_1+\dots+n_{k-1}$ configurations of $R^{(n)}$ any more. This prompts us to write $R^{(n)}$, using  Lemma \thv(2.lem2) with a 2-level tree, as
\be
R^{(n)}
=
\begin{pmatrix} 
2^{n_{k}+\dots+n_{\ell}}\otimes R^{(n_1+\dots+n_{k-1})}
\\ 
R^{(n_{k}+\dots+n_{\ell})}\oplus\dots\oplus R^{(n_{k}+\dots+n_{\ell})} 
\end{pmatrix},
\Eq(2.prop1.87)
\ee
where the concatenation is over $2^{(n_1+\dots+n_{k-1})}$ matrices $R^{(n_{k}+\dots+n_{\ell})}$. Thus, the matrix obtained from $R^{(n)}$ by removing the first $n_1+\dots+n_{k-1}$ configurations  is the concatenation of $2^{(n_1+\dots+n_{k-1})}$ identical matrices $R^{(n_{k}+\dots+n_{\ell})}$.  Reasoning as in \eqv(2.prop1.67)-\eqv(2.prop1.70), we define
\be
\begin{split}
Z^{n_1+\dots+n_{k-1}}_{n_1,\dots,n_{k-1}}
\equiv &
\Bigg\{
1\leq i\leq  d(n_1+\dots+n_{k-1}) : \textrm{for all}\,\,  1\leq l\leq k-1
\\
&\quad\quad\quad\quad\quad\quad\quad\quad\,\,\, \sum_{1\leq \nu-(n_0+\dots+n_{l-1})\leq n_{l}}
r^{(n_1+\dots+n_{k-1}),\nu}_i=0
\Bigg\},
\end{split}
\Eq(2.prop1.89)
\ee
with the convention that $n_0=0$. For all $1\leq\nu\leq n$ we then set, given 
$
j\in Z^{n_1+\dots+n_{k-1}}_{n_1,\dots,n_{k-1}}
$ 
and a subset 
$
\ZZ\subseteq\left\{1,\dots,d(n_{k}+\dots+n_{\ell})\right\}
$,
\be
\begin{split}
&\UU^{(n),\nu}_{n_{k}+\dots+n_{\ell},j}(\ZZ)
\\
\equiv &\,\,
\frac{1}{d(n_{k}+\dots+n_{\ell})}\sum_{(j-1)d(n_{k}+\dots+n_{\ell})+1\leq i\leq j d(n_{k}+\dots+n_{\ell})}
\1_{\left\{i-(j-1)d(n_{k}+\dots+n_{\ell})\in \ZZ\right\}}
\\
&\,\,
r^{(n),\nu}_i
\sign\left[
\left(r^{(n)}_i,\boldsymbol{0}_{n_1+\dots+n_{k-1}}
\oplus a_{k}\boldsymbol{1}_{n_{k}}\oplus\dots\oplus a_{\ell}\boldsymbol{1}_{n_{\ell}}\right)\right],
\end{split}
\Eq(2.prop1.90)
\ee
and
\be
\begin{split}
&\VV^{(n_{k}+\dots+n_{\ell}),\nu}(\ZZ)
\\
\equiv
&\,\, \frac{1}{d(n_1+\dots+n_{k-1})}\sum_{1\leq j\leq d(n_1+\dots+n_{k-1})}
\1_{\left\{j\in  Z^{n_1+\dots+n_{k-1}}_{n_1,\dots,n_{k-1}}\right\}}
\UU^{(n),\nu}_{n_{k}+\dots+n_{\ell},j}\left(\ZZ\right).
\end{split}
\Eq(2.prop1.100)
\ee
Equipped with these definitions, we arrive at the expression, valid for all $1\leq\nu\leq n$
\be
\begin{split}
\SS^{(n),\nu}_{n_1,\dots,n_{k}}&=\VV^{(n_{k}+\dots+n_{\ell}),\nu}\left(\left(Z^{n_{k}+\dots+n_{\ell}}_{n_{k}}\right)^c\right),
\end{split}
\Eq(2.prop1.101)
\ee
where $Z^{n_{k}+\dots+n_{\ell}}_{n_{k}}$  is the set
\be
Z^{n_{k}+\dots+n_{\ell}}_{n_{k}}
=
\left\{
1\leq i'\leq d(n_{k}+\dots+n_{\ell}) : 
\sum_{1\leq \nu'\leq n_{k}}
r^{(n_{k}+\dots+n_{\ell}),\nu'}_{i'}=0
\right\}
\Eq(2.prop1.102)
\ee
obtained by replacing $n$ by $n_{k}+\dots+n_{\ell}$  in \eqv(2.prop1.29).
From this point on, the proof closely follows the case $k=\ell-1$.
If $1\leq\nu\leq n_1+\dots +n_{k-1}$, reasoning as in \eqv(2.prop1.71)-\eqv(2.prop1.75), we have
\be
\SS^{(n),\nu}_{n_1,\dots,n_{k}} = 0.
\Eq(2.prop1.103)
\ee
If $n_1+\dots +n_{k-1}+1\leq \nu\leq n_1+\dots +n_{\ell}$ then, proceeding exactly as in \eqv(2.prop1.76)-\eqv(2.prop1.80),
we have, on the one hand, that for all $j\in  Z^{n_1+\dots+n_{k-1}}_{n_1,\dots,n_{k-1}}$,
\be
\UU^{(n),\nu}_{n_{k}+\dots+n_{\ell},j}\left(\bigl(Z^{n_{k}+\dots+n_{\ell}}_{n_{k}}\bigr)^c\right)
=
\SS^{(n_{k}+\dots+n_{\ell}),\nu}_{n_{k}},
\Eq(2.prop1.104)
\ee
where
\be
\begin{split}
\SS^{(n_{k}+\dots+n_{\ell}),\nu}_{n_{k}}
&\equiv
\frac{1}{d(n_{k}+\dots+n_{\ell})}\sum_{1\leq i\leq d(n_{k}+\dots+n_{\ell})}\1_{\left\{i\in \bigl(Z^{n_{k}+\dots+n_{\ell}}_{n_{k}}\bigr)^c\right\}} 
\\
&
r^{(n_{k}+\dots+n_{\ell}),\nu}_{i}
\sign\left[
\left(r^{(n_{k}+\dots+n_{\ell})}_{i},
a_{k}\boldsymbol{1}_{n_{k}}\oplus\dots\oplus a_{\ell}\boldsymbol{1}_{n_{\ell}}\right)
\right].
\end{split}
\Eq(2.prop1.104bis)
\ee
On the other hand, for all $j\in  Z^{n_1+\dots+n_{k-1}}_{n_1,\dots,n_{k-1}}$,
\be
\VV^{(n_{k}+\dots+n_{\ell}),\nu}\left(\left(Z^{n_{k}+\dots+n_{\ell}}_{n_{k}}\right)^c\right)
=
\frac{\bigl |Z^{n_1+\dots+n_{k-1}}_{n_1,\dots,n_{k-1}}\bigr|}{d(n_1+\dots+n_{k-1})}
\,\UU^{(n),\nu}_{n_{k}+\dots+n_{\ell},j}\left(\bigl(Z^{n_{k}+\dots+n_{\ell}}_{n_{k}}\bigr)^c\right),
\ee
where
\be
\frac{\bigl |Z^{n_1+\dots+n_{k-1}}_{n_1,\dots,n_{k-1}}\bigr|}{d(n_1+\dots+n_{k-1})}
=
\prod_{l=1}^{k-1}\frac{1}{d(n_l)}{{n_l}\choose{n_l/2}}
=
\prod_{l=1}^{k-1}\a^{(n_l+1)}.
\Eq(2.prop1.105)
\ee
Thus,
\be
\SS^{(n),\nu}_{n_1,\dots,n_{k}} =
\SS^{(n_{k}+\dots+n_{\ell}),\nu}_{n_{k}}
\prod_{l=1}^{k-1}\a^{(n_l+1)}.
\Eq(2.prop1.106)
\ee
Note that the sum \eqv(2.prop1.104bis) is nothing other than $\SS^{(n),\nu}_{n_1}$ in  \eqv(2.prop1.7bis), where the pair  $n$, $n_1$ is replaced by the pair $n_{k}+\dots+n_{\ell}$,  $n_k$  and where
$
m= a_{k}\boldsymbol{1}_{n_{k}}\oplus\dots\oplus a_{\ell}\boldsymbol{1}_{n_{\ell}}
$.
It is therefore evaluated in the same way as $\SS^{(n),\nu}_{n_1}$ (see the first step of the proof).
In parallel with \eqv(2.prop1.8)-\eqv(2.prop1.9), note that under the assumption 
\be
2a_k>n_{k+1}a_{k+1}+\dots+n_{\ell}a_{\ell},
\Eq(2.prop1.106ter)
\ee
we have, for all $i\in\bigl(Z^{n_{k}+\dots+n_{\ell}}_{n_{k}}\bigr)^c$
\be
\begin{split}
\left|a_k\sum_{1\leq\nu\leq n_{k}}r^{(n_{k}+\dots+n_{\ell}),\nu}_{i}\right|
\geq 
&
2a_k
> 
n_{k+1}a_{k+1}+\dots+n_{\ell}a_{\ell}
\\
\geq 
&
\left|
\sum_{l=k+1}^{\ell}a_{l}\sum_{1\leq\nu-(n_k+\dots+n_{l-1})\leq n_{l}}r^{(n_{k}+\dots+n_{\ell}),\nu}_{i}
\right|
\end{split}
\Eq(2.prop1.106quat)
\ee
(where for $l=k+1$, the sum in the last line is over $1\leq\nu\leq n_k$). Note also that the second inequality is strict. Thus, the value of the sign function in \eqv(2.prop1.104bis) is determined by the Rademacher configurations $r^{(n_{k}+\dots+n_{\ell}),\nu}$ with $1\leq\nu-(n_1+\dots +n_{k-1})\leq n_{k}$, namely,
\be
\begin{split}
&
\sign\left[
\left(r^{(n_{k}+\dots+n_{\ell}),\nu}_i,
 a_{k}\boldsymbol{1}_{n_{k}}\oplus\dots\oplus a_{\ell}\boldsymbol{1}_{n_{\ell}}\right)
\right]
\\
=\,&
\sign\left[
\left(r^{(n_{k}+\dots+n_{\ell}),\nu}_i,
a_{k}\boldsymbol{1}_{n_{k}}\oplus\boldsymbol{0}_{n_{k+1}+\dots+n_{\ell}}\right)
\right].
\end{split}
\Eq(2.prop1.106bis)
\ee
As we are now accustomed to, there are two cases.
If $1\leq\nu-(n_1+\dots +n_{k-1})\leq n_{k}$, then repeating the strategy of \eqv(2.prop1.10) step by step, using the identity
\be
R^{(n_{k}+\dots+n_{\ell})}
=
\begin{pmatrix} 
2^{n_{k-1}+\dots+n_{\ell}}\otimes R^{(n_k)}
\\ 
R^{(n_{k-1}+\dots+n_{\ell})}\oplus\dots\oplus R^{(n_{k-1}+\dots+n_{\ell})}
\end{pmatrix}
\ee
instead of \eqv(2.prop1.4), we obtain
\be
\begin{split}
\SS^{(n_{k}+\dots+n_{\ell}),\nu}_{n_{k}}
=&\frac{1}{d(n_k)}\sum_{1\leq i\leq d(n_k)}r^{(n_k),\nu}_i
\sign\left[a_k
\left(r^{(n_k)}_i,\boldsymbol{1}_{n_k}\right)
\right]
\\
=&\a^{(n_k)}.
\end{split}
\Eq(2.prop1.104ter)
\ee
In the last equality we used that for $a_k>0$ the expression in the previous line is equal to the right-hand side of the mixed-memory equation \eqv(2.2.1) for the $n_k$-asymmetric solution. If, however,  $n_1+\dots +n_{k}+1\leq\nu\leq n_1+\dots +n_{\ell}$, then a repetition of the proof of \eqv(2.prop1.12) readily yields $\SS^{(n_{k}+\dots+n_{\ell}),\nu}_{n_{k}}=0$. Note that these results are valid for all $1\leq k\leq \ell-1$. Collecting them, we conclude that for all $1\leq k\leq \ell-1$ 
\be
\SS^{(n_{k}+\dots+n_{\ell}),\nu}_{n_{k}}
=
\begin{cases}
\a^{(n_k)} & \text{if $1\leq\nu-(n_0+\dots +n_{k-1})\leq n_{k}$}, \\
0&  \text{if  $n_1+\dots +n_{k}+1\leq\nu\leq n_1+\dots +n_{\ell}$},
\end{cases}
\Eq(2.prop1.104qua)
\ee
where   $n_0=0$.

Inserting \eqv(2.prop1.104qua) in \eqv(2.prop1.106), we get that if
\be
2a_k>n_{k+1}a_{k+1}+\dots+n_{\ell}a_{\ell}
\Eq(2.prop1.107pre)
\ee
where $n_{k}\geq 2$ is even for all $2\leq k\leq \ell-1$ and $n_{\ell}\geq 1$ is of arbitrary parity,  then
\be
\begin{split}
\SS^{(n),\nu}_{n_1,\dots,n_{k}}
=&
\begin{cases}
\a^{(n_k)}\prod_{l=1}^{k-1}\a^{(n_{l}+1)} & \text{if $1\leq\nu-(n_1+\dots +n_{k-1})\leq n_{k}$}, \\
0&  \text{if $n_1+\dots +n_{k}+1\leq\nu\leq n_1+\dots +n_{\ell}$.}
\end{cases}
\end{split}
\Eq(2.prop1.107)
\ee
Note that when $k=1$,  by \eqv(2.prop1.104qua),  $\SS^{(n_{k}+\dots+n_{\ell}),\nu}_{n_{k}}=\SS^{(n),\nu}_{n_1}$
where $\SS^{(n),\nu}_{n_1}$ is defined in \eqv(2.prop1.57).
Putting this observation together with \eqv(2.prop1.103) and \eqv(2.prop1.107) finally yields the claim of  \eqv(2.prop1.85')-\eqv(2.prop1.85).

\smallskip
\emph{\textbf{Conclusion of the general construction step.}}
We can now return  to the problem of constructing solutions to the system of equations \eqv(2.2.1) of the form \eqv(2.prop1.55) under the assumptions \eqv(2.prop1.56). Inserting \eqv(2.prop1.63'), \eqv(2.prop1.63'') and \eqv(2.prop1.85) in \eqv(2.2.1bis), we get that if $n_1,\dots, n_{\ell-1}\geq 2$ are even and $n_{\ell}\geq 1$ has any parity, for any numbers  $a_1,\dots, a_{\ell}>0$ satisfying the conditions
\be
2a_k>n_{k+1}a_{k+1}+\dots+n_{\ell}a_{\ell}\quad\text{for all}\quad 1\leq k\leq \ell-1,
\Eq(2.prop1.108)
\ee
we have, for all $1\leq k\leq \ell$ and all $n_0+\dots +n_{k-1}+1\leq\nu\leq n_1+\dots +n_{k}$
\be
m_{\nu}
=\sum_{1\leq k\leq \ell} \SS^{(n),\nu}_{n_1,\dots,n_{k}}
=\a^{(n_k)}\prod_{l=1}^{k-1}\a^{(n_{l}+1)},
\Eq(2.prop1.109)
\ee
where $n_0=0$ and with the convention that the product in \eqv(2.prop1.109) equals $1$ if $k=1$.
Thus, choosing
\be
a_k=\g^{(k)}\equiv \a^{(n_k)}\prod_{l=1}^{k-1}\a^{(n_{l}+1)}, \quad 1\leq k\leq \ell,
\Eq(2.prop1.119')
\ee
the vector 
\be
m=\g^{(1)}\boldsymbol{1}_{n_1}\oplus \dots\oplus \g^{(\ell)}\boldsymbol{1}_{n_{\ell}}
\Eq(2.prop1.110)
\ee
verifies the mixed-memory equation \eqv(2.2.1), under the conditions that 
$n_1,\dots, n_{\ell-1}\geq 2$ are even, $n_{\ell}\geq 1$ has any parity and
\be
2\g^{(k)}>n_{k+1}\g^{(k+1)}+\dots+n_{\ell}\g^{(\ell)}\quad \text{for all}\quad 1\leq k\leq \ell-1.
\Eq(2.prop1.111)
\ee

The expression \eqv(2.prop1.119') of $\g^{(k)}$ can be slightly simplified. From \eqv(2.lem1.1), the definition 
$
\lfloor x\rfloor =\max\{k\in \Z \mid k\leq x\}
$
and the well-known identity ${{u}\choose{v}}=\frac{u}{v}{{u-1}\choose{v-1}}$, it follows that
\be
\a^{(2k+1)}=2^{-2k}{{2k}\choose{k}}=2^{-2k+1}{{2k-1}\choose{k-1}}=\a^{(2k)}.
\Eq(2.rem2.1)
\ee
Since, by assumption, $n_k$ is  even for all $1\leq k\leq \ell-1$, \eqv(2.prop1.119') becomes 
\be
\g^{(k)}
= \prod_{l=1}^{k}\a^{(n_{l})}, \quad 1\leq k\leq \ell.
\Eq(2.rem2.2)
\ee
This concludes the proof of Proposition \thv(2.prop1), (i), in the case  where $s=n$, $\ell>2$ and $(n_1,\dots, n_{\ell})$ is an $\ell$-composition of $n$ such that $n_1,\dots, n_{\ell-1}\geq 2$ are even and $n_{\ell}\geq 1$ has any parity.

\smallskip
\paragraph{Last step.} To complete the proof of assertion (i), it remains to see that the solutions $m\in\R^n$ we have just constructed can be embedded in $\R^{n+n_0}$ for any $n_0>0$, which means that if $m$ is as defined in \eqv(2.prop1.110) and satisfies \eqv(2.prop1.111), then for any $n_0>0$, $m\oplus\boldsymbol{0}_{n_0}$ verifies the mixed-memory equation \eqv(2.2.1). Clearly, it suffices to check that for all $n+1\leq \nu\leq n+n_0$,
\be
\sum_{1\leq i\leq d(n+n_0)}r^{(n+n_0),\nu}_i
\sign\left[
\left(r^{(n+n_0)}_i,m\oplus\boldsymbol{0}_{n_0}\right)
\right]
=0.
\Eq(2.prop1.112)
\ee
The proof of this statement is a repetition of the proof of \eqv(2.prop1.12). Using Lemma \thv(2.lem2) to write
\be
R^{(n+n_0)}
=
\begin{pmatrix} 
2^{n_0}\otimes R^{(n)}
\\ 
R^{(n_0)}\oplus\dots\oplus R^{(n_0)}
\end{pmatrix},
\Eq(2.prop1.113)
\ee
and proceeding as in \eqv(2.prop1.12bis)-\eqv(2.prop1.15bis), replacing the matrix \eqv(2.prop1.4) by \eqv(2.prop1.113), we have that for all $n+1\leq \nu\leq n+n_0$, the left-hand side of \eqv(2.prop1.112) is equal to
\be
\begin{split}
&
\sum_{1\leq j\leq d(n)}\sum_{(j-1)d(n_0)+1\leq i\leq jd(n_0)}r^{(n+n_0),\nu}_i
\sign\left[
\left(\sum_{1\leq \nu\leq n} r^{(n)}_j m_{\nu}\right)
\right]
\\
=&
\sum_{1\leq j\leq d(n)}
\sign\left[
\left(r^{(n)}_j,m\right)
\right]
\left(r^{(n_0),\nu-n}, r^{(n_0),0}\right)
\\
=&
0,
\end{split}
\Eq(2.prop1.114)
\ee
where the last equality follows from \eqv(2.1.6) and \eqv(1.lem1.1) of Lemma \thv(2.lem1), (i).
The proof of assertion (i) of Proposition \thv(2.prop1) is now complete.

Assertion (ii) of the proposition follows from Lemma \thv(2.lem4), (i).
Finally, the proof of assertion (iii) is a by-product of the proof of (i). 
Without loss of generality we can take  $s=n$ and $\ell\geq 1$. Given an $\ell$-composition $(n_1,\dots, n_{\ell})$ of $n$  such that $n_1,\dots, n_{\ell-1}\geq 2$ are even and $n_{\ell}\geq 1$ has arbitrary parity, consider solutions to \eqv(2.2.1) of the form 
\be
m=\g^{(1)}\boldsymbol{1}_{n_1}\oplus\dots\oplus \g^{(\ell)}\boldsymbol{1}_{n_{\ell}}
\Eq(2.prop1.120)
\ee
where the sequence $\left(\g^{(k)}\right)_{1\leq k\leq \ell}$ satisfies the conditions \eqv(2.prop1.2'). When $\ell=1$, $m$ reduces to an $n_1$-symmetric solution (see Lemma \thv(2.lem3)). Eq.~\eqv(2.prop1.iii1) and \eqv(2.prop1.iii2) are trivially verified. Moreover, if $n_1$ is odd
\be
\inf_{1\leq i\leq d(n)}
\left|\left(r^{(n)}_i,m\right)\right|= \g^{(1)}>0.
\Eq(2.prop1.121)
\ee
Assume now that $\ell\geq 2$. Using \eqv(2.prop1.29) and \eqv(2.prop1.29bis) we write
\be
\begin{split}
&\{1,\dots,d(n)\}
\\
=
&\left(Z^{n}_{n_1}\right)^c\bigcup_{k=2}^{\ell-1}\left(Z^{n}_{n_1}\cap \dots\cap Z^{n}_{n_{k-1}}\cap\left(Z^{n}_{n_k}\right)^c\right)
\bigcup\left(Z^{n}_{n_1}\cap \dots\cap Z^{n}_{n_{\ell-1}}\right),
\end{split}
\Eq(2.prop1.115)
\ee
where by convention the union over $2\leq k\leq \ell-1$  is the empty set if $\ell=2$. Now note that if $\ell>2$  and $2\leq k\leq\ell-1$ then for all $i\in\bigl(Z^{n}_{n_1}\cap \dots\cap Z^{n}_{n_{k-1}}\bigr)$ 
\be
\begin{split}
\bigl(r^{(n)}_i,m\bigr)
&=
\left(r^{(n)}_i,\boldsymbol{0}_{n_1+\dots+n_{k-1}}
\oplus \g^{(k)}\boldsymbol{1}_{n_{k}}\oplus\dots\oplus \g^{(\ell)}\boldsymbol{1}_{n_{\ell}}\right)
\\
&=
\left(r^{(n_{k}+\dots+n_{\ell}),\nu}_{i'},
\g^{(k)}\boldsymbol{1}_{n_{k}}\oplus\dots\oplus \g^{(\ell)}\boldsymbol{1}_{n_{\ell}}\right),
\end{split}
\Eq(2.prop1.116)
\ee
for some $1\leq i'\leq  d(n_{k}+\dots+n_{\ell})$. This last identity follows from the fact that, according to \eqv(2.prop1.87), the matrix obtained from $R^{(n)}$ by removing the first $n_1+\dots+n_{k-1}$ configurations  is the concatenation of $2^{(n_1+\dots+n_{k-1})}$ matrices $R^{(n_{k}+\dots+n_{\ell})}$. So, if $\left(\g^{(k)}\right)_{1\leq k\leq \ell}$ satisfies \eqv(2.prop1.2'), 
reasoning as in \eqv(2.prop1.106ter)-\eqv(2.prop1.106quat), we obtain that if $\ell>2$ and $2\leq k\leq\ell-1$ then for all
$i\in\bigl(Z^{n}_{n_1}\cap \dots\cap Z^{n}_{n_{k-1}}\cap\left(Z^{n}_{n_k}\right)^c\bigr)$
\be
\left|\bigl(r^{(n)}_i,m\bigr)\right|
>2\g^{(k)}-\left[n_{k+1}\g^{(k+1)}+\dots+n_{\ell}\g^{(\ell)}\right]
>0.
\Eq(2.prop1.118)
\ee
Similarly, we check that if $\ell\geq 2$ then \eqv(2.prop1.118) holds with $k=1$ for all  $i\in\left(Z^{n}_{n_1}\right)^c$.
The case where $\ell\geq 2$ and $i\in\bigl(Z^{n}_{n_1}\cap \dots\cap Z^{n}_{n_{\ell-1}}\bigr)$ is different. Here
\be
\begin{split}
\bigl(r^{(n)}_i,m\bigr)
=\,
&
\left(r^{(n)}_i,\boldsymbol{0}_{n_1+\dots+n_{\ell-1}}
\oplus \g^{(\ell)}\boldsymbol{1}_{n_{\ell}}\right)
\\
=\,
&
\g^{(\ell)}\sum_{n_1+\dots +n_{\ell-1}+1\leq\nu\leq n_1+\dots +n_{\ell}}r^{(n),\nu}_i.
\end{split}
\Eq(2.prop1.118bis)
\ee
Thus, if $n_{\ell}$ is even, then for all $i\in\bigl(Z^{n}_{n_1}\cap \dots\cap Z^{n}_{n_{\ell-1}}\bigr)\cap Z^{n}_{n_{\ell}}$
\be
\bigl(r^{(n)}_i,m\bigr)=0,
\ee
and if  $n_{\ell}$ is odd,  then  for all $i\in\bigl(Z^{n}_{n_1}\cap \dots\cap Z^{n}_{n_{\ell-1}}\bigr)$
\be
\left|\bigl(r^{(n)}_i,m\bigr)\right|\geq \g^{(\ell)}.
\ee
Assertion (iii) of Proposition \thv(2.prop1)  now easily follows.
The proof of Proposition \thv(2.prop1) is done.

\subsection{Proof of Proposition \thv(2.prop2)} 
    \TH(S6.3)

The proof of this result is identical, mutatis mutandis, to that of Proposition \thv(2.prop1). 
We only indicate how to modify the first step of the proof of assertion (i)  of Proposition \thv(2.prop1), namely, the case  $s=n$ and $\ell=2$. As in the latter, we seek solutions of the system \eqv(2.prop2.0) of the form \eqv(2.prop1.3) for some strictly positive numbers $a_1, a_2>0$ that now satisfy 
\be
2f(a_1)>n_{2}f(a_{2})
\Eq(2.prop2.4)
\ee
instead of \eqv(2.prop1.8). The definitions and property \eqv(2.prop1.5), \eqv(2.prop1.6) and \eqv(2.prop1.7) are unchanged whereas the inner product $\bigl(r^{(n)}_i,m\bigr)$ in the definitions \eqv(2.prop1.7bis)  is replaced by
$\bigl(r^{(n)}_i,\boldsymbol{f}(m)\bigr)$. Thus, using \eqv(2.prop2.4), \eqv(2.prop1.9) is replaced by
\be
\sign\left[
\left(r^{(n)}_i,\boldsymbol{f}(m)\right)
\right]
=
\sign\left(
f(a_1) {
\sum_{1\leq\nu\leq n_1}}r^{(n),\nu}_i
\right).
\Eq(2.prop2.5)
\ee
Since, by assumption,  $a_1>0$ and $f(x)>0$ for all $x>0$, 
\be
\sign\left(
f(a_1) {
\sum_{1\leq\nu\leq n_1}}r^{(n),\nu}_i
\right)
=
\sign\left(
 {
\sum_{1\leq\nu\leq n_1}}r^{(n),\nu}_i
\right).
\Eq(2.prop2.6)
\ee
From this point on, the proof is completed in exactly the same way as the first step of the proof of assertion (i) of Proposition \thv(2.prop1). The proof of the general step is modified in the same way. This readily leads to the claim of assertion (i). Assertion (ii) of the proposition follows from Lemma \thv(2.lem4), (ii). Finally, the proof of assertion (iii) is an elementary adaptation  of that of  Proposition \thv(2.prop1), (iii).

We close this section with the proof of Lemma \thv(2.lem4).

\subsection{Proof of Lemma \thv(2.lem4)} 
    \TH(S6.4) 

We begin with assertion (i). We first claim that given any permutation  $\pi_1:\{1,\dots,n\}\mapsto\{1,\dots,n\}$ of the rows of $R^{(n)}$, 
$m=(m_{\nu})_{1\leq \nu\leq n}$  verifies  \eqv(2.2.1) if and only if  $(m_{\pi_1(\nu)})_{1\leq \nu\leq n}$ verifies  \eqv(2.2.1) .
Indeed, assume that $\bar m\equiv(m_{\pi_1(\nu)})_{1\leq \nu\leq n}$  verifies  \eqv(2.2.1), \emph{i.e.} 
\be
\begin{split}
m_{\pi_1(\nu)}
&=
\frac{1}{d(n)}\sum_{1\leq i\leq d(n)}r^{(n),\nu}_i
\sign\left[
\left(r^{(n)}_i, \bar m\right)
\right],\quad 1\leq\nu\leq n.
\end{split}
\Eq(2.lem4.8)
\ee
Let  $\pi_1^{-1}$ denote the inverse of $\pi_1$ (\emph{i.e.},~$\pi_1^{-1}\pi_1=\pi_1\pi_1^{-1}$ is the identity).
Then, $\bigl(r^{(n)}_i, \bar m\bigr)= \bigl(\bar r^{(n)}_i, m\bigr)$ where $\bar r^{(n), \nu}_i=r^{(n), \pi_1^{-1}(\nu)}_i$ for all $1\leq \nu\leq n$, and so, applying $\pi_1^{-1}$ to the set of  indices $\nu$, \eqv(2.lem4.8) is equivalent to
\be
m_{\nu}
=
\frac{1}{d(n)}\sum_{1\leq i\leq d(n)}r^{(n), \pi_1^{-1}(\nu)}_i
\sign\left[
\left(\sum_{\nu=1}^nr^{(n), \pi_1^{-1}(\nu)}_im_{\nu}\right)
\right],\quad 1\leq\nu\leq n.
\Eq(2.lem4.8bis)
\ee
By Corollary \thv(1.cor1), (ii), there exists a unique permutation  $\pi_2:\{1,\dots,d\}\mapsto\{1,\dots,d\}$ of the columns of $R^{(n)}$ such that, using \eqv(1.cor1.3) and \eqv(1.cor1.4) in turn,  \eqv(2.lem4.8bis)  can be expressed as
\be
\begin{split}
m_{\nu}
&=
\frac{1}{d(n)}\sum_{1\leq i\leq d(n)}r^{(n),\nu}_{\pi_2(i)}
\sign\left[
\left(r^{(n)}_{\pi_2(i)}, m\right)
\right],\quad 1\leq\nu\leq n,
\\
&=
\frac{1}{d(n)}\sum_{1\leq i\leq d(n)}r^{(n),\nu}_{i}
\sign\left[
\left(r^{(n)}_i,m\right)
\right],\quad 1\leq\nu\leq n,
\end{split}
\Eq(2.lem4.9)
\ee
which is  \eqv(2.2.1) and proves our  claim.
Next, we claim that for any vector  $\varepsilon=(\varepsilon_{\nu})_{1\leq \nu\leq n}\in\{-1,1\}^n$, 
$m=(m_{\nu})_{1\leq \nu\leq n}$  verifies \eqv(2.2.1) if and only if the Hadamard product
$
\varepsilon\odot m
$
verifies \eqv(2.2.1).
By \eqv(2.2.1) evaluated at $\varepsilon\odot m$,
\be
\varepsilon_{\nu}m_{\nu}
=
\frac{1}{d(n)}\sum_{1\leq i\leq d(n)}r^{(n),\nu}_i
\sign\left[
\left(r^{(n)}_i,\varepsilon\odot m\right)
\right]
,\quad 1\leq\nu\leq n.
\Eq(2.lem4.10)
\ee
Since $\bigl(r^{(n)}_i,\varepsilon\odot m\bigr)=\bigl(\varepsilon\odot r^{(n)}_i, m\bigr)$  and since $\varepsilon=r^{(n)}_{j}$ for some $1\leq j\leq d(n)$, it follows from Corollary \thv(1.cor1), (i), that there exists a unique permutation $\pi:\{1,\dots,d\}\mapsto\{1,\dots,d\}$ of the columns of $R^{(n)}$ such that, using \eqv(1.cor1.1) and \eqv(1.cor1.2) in turn, \eqv(2.lem4.10) is
equivalent to
\be
\begin{split}
\varepsilon_{\nu}m_{\nu}
=&\frac{1}{d(n)}\sum_{1\leq i\leq d(n)}\varepsilon_{\nu}r^{(n),\nu}_{\pi(i)}
\sign\left[
\left(r^{(n)}_{\pi(i)}, m\right)
\right],\quad 1\leq\nu\leq n,
\\
=&
\frac{1}{d(n)}\sum_{1\leq i\leq d(n)}\varepsilon_{\nu}r^{(n),\nu}_{i}
\sign\left[
\left(r^{(n)}_{i}, m\right)
\right],\quad 1\leq\nu\leq n,
\end{split}
\Eq(2.lem4.11)
\ee
where the final equality is equivalent to \eqv(2.2.1). This proves our second claim.
Combining the above two properties implies the claim of  assertion (i) of the lemma.

To prove assertion (ii),  note that the only difference with assertion (i) is in the treatment of the inner product in \eqv(2.lem4.10), which now becomes
\be
\frac{1}{d(n)}\sum_{1\leq i\leq d(n)}r^{(n),\nu}_i
\sign\left[
\left(r^{(n)}_i, \boldsymbol{f}(\varepsilon\odot m)\right)
\right].
\Eq(2.lem4.12)
\ee
Clearly, when $f$ is odd $\bigl(r^{(n)}_i,\boldsymbol{f}(\varepsilon\odot m)\bigr)=\bigl(\varepsilon\odot r^{(n)}_i, \boldsymbol{f}(m)\bigr)$
and the proof is completed as in assertion (i). If however $f$ is not odd, this property fails. Note that since the solutions constructed in Proposition \thv(2.prop2), (i) are such that $m_{\nu}\geq 0$ for all $1\leq \nu\leq n$, all of our solutions in this case have this same property. The proof of the lemma is done.


\section{Reduction to the Rademacher system}
    \TH(S3)
    
This section is devoted to the proof of the results of Section \thv(S1.2.1), \emph{i.e.}~Theorem \thv(1.theo2.mix) and Proposition \thv(1.prop.mix). It relies on a simple but key tool, Proposition \thv(3.prop1), which allows to link the $n\times N$ random matrix formed by any $n$ patterns to the deterministic Rademacher system $R^{(n)}$. This strategy was first used in the setting of Hopfield-type models in \cite{G92}, \cite{KP89}  to study their thermodynamic properties when the number of patterns $M\equiv M(N)$ grows with $N$ no faster than $cst\ln N$. In what follows, we draw heavily on the results of \cite{G92}.

Given an integer $n\leq M$ independent of $N$, let $\{\mu_1,\dots, \mu_n\}\subset \{1,\dots,M\}$ be an arbitrary subset of $n$ patterns, $\left(\xi^{\mu_\nu}\right)_{1\leq \nu\leq n}$, chosen from the family of $M$ patterns that are used to define a given model. Form the $n\times N$ matrix $\xi^{(n)}\equiv \left(\xi^{\mu_\nu}_i\right)_{1\leq i\leq N, 1\leq \nu\leq n}$ whose rows are given by the $n$ patterns
\be
\xi^{\mu_\nu}=\left(\xi^{\mu_\nu}_i\right)_{1\leq i\leq N}\in\S_N, \quad 1\leq \nu\leq n,
\Eq(3.1.1)
\ee
and whose columns, denoted by 
\be
\xi_i=\left(\xi^{\mu_\nu}_i\right)_{1\leq \nu\leq n}\in\S_n, \quad 1\leq i\leq N, 
\Eq(3.1.2)
\ee
are $N$ configurations in $\S_n$. Recall  from Lemma \thv(2.lem1), (iii) that the collection $\bigl\{r^{(n)}_j\bigr\}_{1\leq j\leq d}$ of column vectors \eqv(2.1.5) of the Rademacher matrix $R^{(n)}$ forms a complete enumeration of the $d(n)\equiv 2^n$ elements of $\S_n$. As a consequence, this collection induces a partition of the set  $\L\equiv \{1,\dots,N\}$ into $d(n)$ disjoint (possibly empty) subsets $\L=\cup_{1\leq j\leq d(n)}\L_j(\xi^{(n)})$ with the  property that $i\in\L_j(\xi^{(n)})$ if and only if $\xi_i=r^{(n)}_j$,
\be
\L_j(\xi^{(n)})\equiv\left\{i\in\L : \xi_i=r^{(n)}_j\right\}, \quad 1\leq j\leq d(n).
\Eq(3.1.3)
\ee

Recalling that the patterns
$
\left(\xi^{\mu}_i\right)_{1\leq i\leq N, 1\leq \mu\leq M}
$,
for all $N\geq 1$, are assumed to be i.i.d. Bernoulli random variables satisfying   \eqv(1.1.2'), \eqv(3.1.3) forms a random partition of $\L$.
We want to determine the typical size
$
L_j(\xi^{(n)})=\left|\L_j(\xi^{(n)})\right|
$
of the sets \eqv(3.1.3). To this end, define  the sequence of subsets
\be
\EE^{\mu_1,\dots, \mu_n}_{N}=\left\{
 L_j\left(\xi^{(n)}\right)=\frac{N}{d(n)}(1+\l_j),\, |\l_j|\leq \d_N,\, 1\leq j\leq d(n)
\right\}
\subseteq \O,
\Eq(3.1.4)
\ee
where $\d_N$ is some strictly decreasing function of $N$ that converges to zero.  Thus,  on $\EE^{\mu_1,\dots ,\mu_n}_{N}$, for large enough $N$,  each $L_j(\xi^{(n)})$ is constrained 
to be close to its expected value, ${N}/{d(n)}$. 
How close is determined by the function  $\d_N$, which we will choose later depending on the application.
The next proposition establishes that almost all realisations of the random variables $(L_j(\xi^{(n)}))_{1\leq j\leq d(n)}$ will eventually,  for $N$ sufficiently large, be contained in the subset $\EE^{\mu_1,\dots ,\mu_n}_{N}$, either for a given subset  $\{\mu_1,\dots ,\mu_n\}$ of patterns or for all possible such subsets. More precisely, set 
\be
\O^{\mu_1,\dots, \mu_n}=
\bigcup_{N_0} \bigcap_{N>N_0}\EE^{\mu_1,\dots, \mu_n}_{N},
\Eq(3.1.5)
\ee 
for any $\{\mu_1,\dots, \mu_n\}\subset \{1,\dots,M\}$ and
\be
\O^{(n)}=
\bigcap_{\{\mu_1,\dots, \mu_n\}\subset\{1,\dots, M\}}\O^{\mu_1,\dots, \mu_n}.
\Eq(3.1.6)
\ee 
Since the sets $\EE^{\mu_1,\dots, \mu_n}_{N}$ depend on the choice of the function $\d_N$, so do the sets $\O^{\mu_1,\dots, \mu_n}$ and $\O^{(n)}$. The next proposition provides sufficient conditions on $\d_N$ for the latter sets to have full measure.

\begin{proposition} 
    \TH(3.prop1)
\item{(i)} 
If $\d^2_N>4\frac{d(n)}{N}\ln N$ then 
\be
\P\left(\O^{\mu_1,\dots, \mu_n}\right)=1.
\Eq(3.prop1.1)
\ee
\item{(ii)} 
If $\d^2_N>2\frac{d(n)}{N}(2\ln N+n\ln M)$ then 
\be
\P\left(\O^{(n)}\right)=1.
\Eq(3.prop1.2)
\ee
\end{proposition}

We can now make the connection between the matrix $\xi^{(n)}$ and the Rademacher matrix $R^{(n)}$ more concrete. By definition of the sets \eqv(3.1.3), using the dilation and concatenation notations  \eqv(2.1.7)-\eqv(2.1.8), $\xi^{(n)}$ can be written as a permutation of the columns of the matrix
\be
\left(L_1\left(\xi^{(n)}\right)\otimes r^{(n)}_1\right)\oplus\dots \oplus \left(L_j\left(\xi^{(n)}\right)\otimes r^{(n)}_j\right)\oplus\dots \oplus \left(L_d\left(\xi^{(n)}\right)\otimes r^{(n)}_d\right),
\Eq(3.1.6')
\ee
where by Proposition \thv(3.prop1), on either of the sets $\O^{\mu_1,\dots, \mu_n}$ or $\O^{(n)}$, all $L_j$'s are asymptotically equal to ${N}/{d(n)}$. In this sense,  $\xi^{(n)}$ is close to a permutation of the columns of the dilated Rademacher matrix $({N}/{d(n)})\otimes R^{(n)}$. 

\begin{proof}[Proof of Proposition \thv(3.prop1)]  The first assertion of the proposition is proved as Proposition 4.1 in \cite{G92}, with $q=2$. 
(In contrast to the original choice made in \cite{G92}, our choice of $\d_N$  is optimal up to a multiplicative constant.
\footnote{Note that the power of $3/2$ appearing in the first formula on p.~989 of \cite{G92} should be $2$. The same typo appears in the previous formula.})
Assertion (ii) is proved in the same way, using in addition that by the union bound
\be
\P\left(\bigcup_{\{\mu_1,\dots, \mu_n\}\subset\{1,\dots, M\}}\left(\EE^{\mu_1,\dots, \mu_n}_{N}\right)^c\right)
\leq 
\sum_{\{\mu_1,\dots, \mu_n\}\subset\{1,\dots, M\}}
\P\Bigl(\left(\EE^{\mu_1,\dots, \mu_n}_{N}\right)^c\Bigr),
\Eq(3.prop1.3)
\ee 
where the number of terms in the last sum, ${M\choose n}$, is bounded above by $M^n$.
\begin{remark}
In  Proposition 4.1 of  \cite{G92}, $n$ is allowed to grow with $N$, provided that $n\leq \a\frac{\log N}{\log2}$, where $0\leq \a<1$ is  an arbitrary constant. This extension carries over unchanged to Proposition \thv(3.prop1). Therefore, the results of the paper can  be extended \emph{a priori} to any such $n$. 
\end{remark}
\end{proof}

We now turn to the proof of the results of Section \thv(S1.2.1).

\begin{proof}[Proof of Theorem \thv(1.theo2.mix)] We first prove the theorem for the classical Hopfield model, \emph{i.e.}~we take $F(x)=\frac{1}{2}x^2$. Given $n\in\N$ odd, take any $m\in\MM_{n,F}$ and denote by $\{\mu_1,\dots, \mu_n\}\subset \{1,\dots,M\}$ the $n$ coordinates of $m$ that are non-zero. With this $m$, form the configuration $\xi^{(N)}(m)=\left(\xi_i(m)\right)_{1\leq i\leq N}$ in $\S_N$ defined by
\be
\xi_i(m)=\sign\left(\sum_{\nu=1}^{n}\xi^{\mu_\nu}_i m_{\mu_\nu}\right), \quad 1\leq i\leq N.
\Eq(1.theo1.mix1)
\ee
To prove that $\xi^{(N)}(m)$ is an $n$-mixed memory, it suffices to check that $\xi^{(N)}(m)$ verifies condition (ii) of Definition \thv(1.def1.2). In view of Proposition \thv(3.prop1), (i), choosing \emph{e.g.}~$\d^2_N=8\frac{d(n)}{N}\ln N$, we have,
for all realisations of $\xi^{(n)}\equiv \left(\xi^{\mu_\nu}_i\right)_{1\leq i\leq N, 1\leq \nu\leq n}$ that belong to 
$\O^{\mu_1,\dots, \mu_n}$,
\be
\begin{split}
N^{-1}\left(\xi^{(N)}(m),\xi^{\mu_\nu}\right)
&=
N^{-1}\sum_{i=1}^N\xi^{\mu_\nu}_i\sign\left(\sum_{\nu=1}^{n}\xi^{\mu_\nu}_i m_{\mu_\nu}\right)
\\
&=
N^{-1}\sum_{j=1}^{d(n)}\sum_{i\in\L_j(\xi^{(n)})}r^{(n), \mu_\nu}_j\sign\left(\sum_{\nu=1}^{n}r^{(n), \mu_\nu}_j m_{\mu_\nu}\right)
\\
&=
\frac{1}{d(n)}\sum_{j=1}^{d(n)}(1+\l_j)r^{(n), \nu}_j\sign\left(\sum_{\nu=1}^{n}r^{(n), \nu}_j m_{\mu_\nu}\right),
\end{split}
\Eq(1.theo1.mix2)
\ee
where we used \eqv(3.1.5), \eqv(3.1.4) and \eqv(3.1.3). Since $\sup_{1\leq j\leq d(n)} |\l_j|\leq \d_N\downarrow 0$ 
as $N\uparrow\infty$
\be
\lim_{N\rightarrow\infty} N^{-1}\left(\xi^{(N)}(m),\xi^{\mu_\nu}\right)
=
\frac{1}{d(n)}\sum_{j=1}^{d(n)}r^{(n), \nu}_j\sign\left(\sum_{\nu=1}^{n}r^{(n), \nu}_j m_{\mu_\nu}\right).
\Eq(1.theo1.mix3)
\ee
Furthermore, since $m\in\MM_{n,F}$, there exists a permutation $\pi$ of $\{1,\dots,n\}$ and a sequence of signs $(\varepsilon_{\nu})_{1\leq \nu\leq n}\in\{-1,1\}^n$ such that 
\be
\left(\varepsilon_{\nu}m_{\pi(\mu_\nu)}\right)_{1\leq\nu\leq n}=\g^{(1)}\boldsymbol{1}_{n_1}\oplus \g^{(2)}\boldsymbol{1}_{n_2}\oplus\dots
\oplus \g^{(\ell)}\boldsymbol{1}_{n_{\ell}},
\Eq(1.theo1.mix4)
\ee
for some allowable $\ell$-composition and some $\g_n=\left(\g^{(k)}\right)_{1\leq k\leq \ell}\in\G_{n,F}$. 
Therefore, by Proposition \thv(2.prop1), \eqv(1.theo1.mix3) gives, for each $\nu\in\{1,\dots,n\}$
\be
\lim_{N\rightarrow\infty} N^{-1}\left(\xi^{(N)}(m),\xi^{\mu_\nu}\right)
=
m_{\nu}.
\Eq(1.theo1.mix5)
\ee
It follows from this and \eqv(3.prop1.1) of Proposition \thv(3.prop1), (i),  that \eqv(1.2.1.2) is verified.

To prove that  \eqv(1.2.1.3) is also verified, we take any $\mu_{n+1}\in\{1,\dots,M\}\setminus\{\mu_1,\dots, \mu_n\}$ 
and repeat the above construction step by step, replacing the set  $\{\mu_1,\dots, \mu_n\}$  by the set of $n+1$ 
coordinates  $\{\mu_1,\dots, \mu_{n+1}\}$. In this way, the representation \eqv(1.theo1.mix4) of $m$ becomes
\be
\left(\varepsilon_{\nu}m_{\pi(\mu_\nu)}\right)_{1\leq\nu\leq n}=\g^{(1)}\boldsymbol{1}_{n_1}\oplus \g^{(2)}\boldsymbol{1}_{n_2}\oplus\dots
\oplus \g^{(\ell)}\boldsymbol{1}_{n_{\ell}}\oplus \boldsymbol{0}_{1},
\Eq(1.theo1.mix6)
\ee
for some permutation $\pi$ of $\{\mu_1,\dots, \mu_{n+1}\}$. Thus, on $\O^{\mu_1,\dots, \mu_{n+1}}$, by Proposition \thv(2.prop1)
\be
\lim_{N\rightarrow\infty} N^{-1}\left(\xi^{(N)}(m),\xi^{\mu_{n+1}}\right)
=
0.
\Eq(1.theo1.mix7)
\ee
It then follows fom Proposition \thv(3.prop1), (i), that \eqv(1.2.1.3) is verified. Since both \eqv(1.2.1.2) and \eqv(1.2.1.3) are verified,  condition (ii) of Definition \thv(1.def1.2) is verified. The proof of Theorem \thv(1.theo2.mix) in the case $F(x)=\frac{1}{2}x^2$ is done.

The proof of Theorem \thv(1.theo2.mix) in the case of general $F$  is a repetition of its proof in the case $F(x)=\frac{1}{2}x^2$  that uses Proposition \thv(2.prop2) with $f(x)=F'(x)\1_{x\neq 0}$ instead of Proposition \thv(2.prop1).  We omit the details.
\end{proof}

\begin{remark}
Since $n$ is independent of $N$ we have in fact proved a slightly stronger result than \eqv(1.2.1.2), namely that for any 
$V=\{\mu_1,\dots,\mu_n\}\subset\{1,\dots,N\}$, using the notations of part (ii) of Definition \thv(1.def1.2)
\be
\P\left(\bigcap_{1\leq \nu\leq n}\left\{\lim_{N\rightarrow\infty} N^{-1}\left(\xi^{(N)}(m),\xi^{\mu_{\nu}}\right)=\hat m_{\nu}\right\}\right)=1.
\Eq(1.theo1.mix8)
\ee
Note that if $M$ is such that $\frac{d(n)}{N}(2\ln N+n\ln M)\rightarrow 0$ as $N\rightarrow\infty$, then using assertion (ii) of Proposition \thv(3.prop1) instead of (i) gives the much stronger statement
\be
\P\left(\bigcap_{\{\mu_1,\dots,\mu_n\}\subset\{1,\dots,N\}}\bigcap_{1\leq \nu\leq n}\left\{\lim_{N\rightarrow\infty} N^{-1}\left(\xi^{(N)}(m),\xi^{\mu_{\nu}}\right)=\hat m_{\nu}\right\}\right)=1.
\Eq(1.theo1.mix9)
\ee
\end{remark}

\begin{proof}[Proof of Proposition \thv(1.prop.mix)] Recall that each $m(\g_n)\in \MM^{\circ}_{n,F}$ takes the form \eqv(1.2.1.8), \emph{i.e.}~is piecewise constant on blocks of length $n_1,\dots,n_\ell$ for some $1\leq \ell\leq n$ and some allowable $\ell$-composition $(n_1,\dots,n_{\ell})$  while all components of coordinate larger than or equal to $n+1$ is zero. The number of distinct permutations of the elements of such an $m(\g_n)$ is
\be
q(\g_n)\equiv {M\choose n_1}{M-n_1\choose n_2}\dots{M-(n_1+\dots+n_{\ell-1})\choose n_\ell}.
\ee
This obeys the bounds 
\be
\frac{M^n}{n_1!\dots n_\ell!}\left(1-\frac{n}{M}\right)^{\ell-1}
\leq
q(\g_n)
\leq
\frac{M^n}{n_1!\dots n_\ell!}\left(1-\frac{1}{M}\right)^{\ell-1}.
\ee
Thus,
\be
\left|\MM_{n,F}\right|=(2^{n})^a\sum_{m(\g_n)\in \MM^{\circ}_{n,F}}q(\g_n),
\ee
where $a=1$ if $F'$ is an odd function and $a=2$ otherwise.
Since $\MM^{\circ}_{n,F}$ only depends on $n$ and $F$ and not on $M$, taking
\be
 A_{n,F} =(2^{n})^a\sum_{m(\g_n)\in \MM^{\circ}_{n,F}}\frac{1}{n_1!\dots n_\ell!},
\ee
gives
\be
A_{n,F}M^n\left(1-\frac{n}{M}\right)^{n-1}
\leq 
\left|\MM_{n,F}\right|
\leq 
A_{n,F}M^n,
\ee
which is \eqv(1.prop.mix.1). The proof is done.
\end{proof}


\section{Conditions for exact retrieval of mixed solutions}
    \TH(S4)
    
In this section, we prove  Theorems \thv(1.theo1), \thv(1.theo2) and  \thv(1.theo3). 
The three follow the same general structure, which will be described in detail in the proof of Theorem \thv(1.theo1).
To simplify the (already cumbersome) notations, we write $\xi(m)\equiv \xi^{(N)}(m)$. We also use the classical notation
$[n]\equiv\{1,\dots, n\}$.

\subsection{Proof of Theorem \thv(1.theo1)} 
    \TH(S4.1)     
Taking  $F(x)=\frac{1}{2}x^2$ in the definitions \eqv(1.2.1) and \eqv(1.2.4), it is easy to check that $T^{\text{HK}}=T^{\text{G}}$. Let us call this map $T$. We start by proving  assertion (i). Without loss of generality we can assume that $m\in\MM^{\circ}_{n,F}$ (see \eqv(1.2.1.14)), so that $\xi(m)$ is a mixture of the first $n$ patterns $(\xi^{\nu})_{1\leq \nu\leq n}$,
\be
\xi_i(m)=\sign\left(\sum_{\nu=1}^n\xi^{\nu}_i m_\nu\right),
\Eq(4.1.1)
\ee
where $m_{\nu}>0$ for all $1\leq \nu\leq n$. The  first step of  the proof of \eqv(1.theo1.2) consists in writing
\be
\begin{split}
1-\P\left[\xi(m)=T(\xi(m))\right]
&=
1-\P\left[
\forall_{1\leq i\leq N}\xi_i(m)=T_i(\xi(m))
\right]
\\
&=
\P\left[
\exists_{1\leq i\leq N}\xi_i(m)\neq T_i(\xi(m))
\right]
\\
&
\leq 
\sum_{i=1}^N\P\left[\xi_i(m)T_i(\xi(m))\neq 1\right],
\end{split}
\Eq(4.1.3)
\ee
were we used that $\xi_i(m)\neq T_i(\xi(m))$ if and only if $\xi_i(m)T_i(\xi(m))\neq1$. Our task is thus reduced to finding a suitable upper bound on $\P\left[\xi_i(m)T_i(\xi(m))\neq 1\right]$, \emph{i.e.}~one that decays fast enough with $N$.  To this end, we write, 
using \eqv(1.2.1) and the fact that  $F'(x)=x$,
\be
\xi_i(m)T_i(\xi(m))=\sign\left\{\xi_i(m)I_{N,i}^{(1)}(m)
+ \xi_i(m)I_{N,i}^{(2)}(m)\right\},
\Eq(4.1.5)
\ee
where
\be
\begin{split}
I_{N,i}^{(1)}(m) 
&\equiv
\frac{1}{N}\sum_{\nu=1}^{n}\sum_{1\leq j\neq i \leq N} \xi^{\nu}_i\xi^{\nu}_j\xi_j(m),
\\
I_{N,i}^{(2)}(m)
&\equiv
\frac{1}{N}\sum_{\mu=n+1}^{M}\sum_{1\leq j\neq i \leq N} \xi^{\mu}_i\xi^{\mu}_j\xi_j(m).
\\
\end{split}
\Eq(4.1.6)
\ee
We will view the first term of the sum on the right-hand side of \eqv(4.1.5) as a leading contribution and the second as a fluctuation. 
The proof of the theorem then hinges on the next two lemmata.

\begin{lemma}[Leading term]
    \TH(4.lem1)
If $n=1$, then $\xi_i(m)I_{N,i}^{(1)}(m)=1-1/N$.
If $n> 1$, then under the assumptions and notations of Proposition \thv(3.prop1), (i) with 
$
\{\mu_1,\dots, \mu_n\}
=
[n]
$, 
the following holds on $\O^{[n]}$ for all $m\in\MM^{\circ}_{n,F}$. 
For all $1\leq i\leq N$ there exists $1\leq {i'}\leq d(n)$ such that 
\be
\xi_i(m)I_{N,i}^{(1)}(m)= \left|\left(r^{(n)}_{i'},m\right)\right| +\xi_i(m)\left(-\frac{n}{N}+ \rho_{N}(m,n)\right),
\Eq(4.lem1.1)
\ee
where $\bigl|\bigl(r^{(n)}_{i'},m\bigr)\bigr|$ obeys the bound \eqv(2.prop1.iii3)-\eqv(2.prop1.iii4) of Proposition \thv(2.prop1) and
$\left|\rho_{N}(m,n)\right|\leq \d_N n$.
\end{lemma}

Given  $\l>0$, define the event
\be
\AA_N(n,i,m,\l)=
\left\{
\xi_i(m)I_{N,i}^{(2)}(m)<-\l\right\}, \quad 1\leq i\leq N.
\Eq(4.1.12)
\ee
\begin{lemma}
    \TH(4.lem2)
The following holds for all $n$, all $m\in\MM^{\circ}_{n,F}$ and all $1\leq i\leq N$.
For all $\l>0$
\be
\P\left(\AA_N(n,i,m,\l)\right)\leq \exp\left\{-\frac{1}{2}\frac{\l^2N^2}{(M-n)(N-1)}\right\}.
\Eq(4.1.13)
\ee
\end{lemma}

\begin{proof}[Proof of Lemma \thv(4.lem1)] The case $n=1$ follows from a simple calculation. Now assume that $n> 1$.
The proof of this lemma relies on the tools and strategy presented in Section \thv(S3).
Denoting by $\xi^{(n)}\equiv \left(\xi^{\nu}_i\right)_{1\leq i\leq N, 1\leq \nu\leq n}$  the  $n\times N$ matrix whose rows are given by the $n$ patterns
$
\left(\xi^{\nu}\right)_{1\leq \nu\leq n},
$
let $\L=\cup_{1\leq {j'}\leq d(n)}\L_{j'}(\xi^{(n)})$ be the partition defined by \eqv(3.1.3).
Then, for all $1\leq i\leq N$ there exists (a unique) $1\leq {i'}\leq d(n)$ such that $i\in \L_{i'}(\xi^{(n)})$ and
\be
\begin{split}
I_{N,i}^{(1)}(m)
&=
\frac{1}{N}\sum_{\nu=1}^{n}\sum_{j=1}^N \xi^{\nu}_i\xi^{\nu}_j\sign\left(\sum_{\nu=1}^n\xi^{\nu}_j m_\nu\right)-\frac{n}{N}
\\
&=
\frac{1}{N}\sum_{\nu=1}^{n}r^{(n),\nu}_{i'}\sum_{ j'=1}^{d(n)} \left|\L_{j'}(\xi^{(n)})\right|r^{(n),\nu}_{j'}
\sign\left[\left(r^{(n)}_{j'},m\right)\right]-\frac{n}{N}.
\end{split}
\Eq(4.1.7)
\ee
Under the assumptions of  Proposition \thv(3.prop1), (i), we have that on $\O^{[n]}$
\be
\begin{split}
I_{N,i}^{(1)}(m)
&= 
\sum_{\nu=1}^{n}r^{(n),\nu}_{i'}\frac{1}{d(n)}\sum_{ j'=1}^{d(n)}r^{(n),\nu}_{j'}\sign\left[\left(r^{(n)}_{j'},m\right)\right]
-\frac{n}{N}+\rho_{N}(m,n),
\end{split}
\Eq(4.1.8)
\ee
where $\left|\rho_{N}(m,n)\right|\leq \d_N n$.
Now, since $m\in\MM^{\circ}_{n,F}$ it follows from assertion (i) of Proposition \thv(2.prop1) that
\be
\frac{1}{d(n)}\sum_{ j'=1}^{d(n)}r^{(n),\nu}_{j'}\sign\left[\left(r^{(n)}_{j'},m\right)\right]= m_{\nu},
\ee
so that,
\be
\sum_{\nu=1}^{n}r^{(n),\nu}_{i'}\frac{1}{d(n)}\sum_{ j'=1}^{d(n)}r^{(n),\nu}_{j'}\sign\left[\left(r^{(n)}_{j'},m\right)\right]
=\left(r^{(n)}_{i'},m\right).
\Eq(4.1.9)
\ee
Inserting \eqv(4.1.9) into \eqv(4.1.8), 
\be
I_{N,i}^{(1)}(m)= \left(r^{(n)}_{i'},m\right)-\frac{n}{N} + \rho_{N}(m,n).
\Eq(4.1.10)
\ee
Finally, it follows from  \eqv(4.1.1) that for $i'$ as in \eqv(4.1.7)
\be
\xi_i(m)
=\sign\left[\left(r^{(n)}_{i'},m\right)\right].
\Eq(4.1.11)
\ee
By this and \eqv(4.1.10),
\be
\xi_i(m)I_{N,i}^{(1)}(m)= \left|\left(r^{(n)}_{i'},m\right)\right|+\xi_i(m)\left(-\frac{n}{N}+ \rho_{N}(m,n)\right),
\Eq(4.1.12')
\ee
where by assertion (iii) of Proposition \thv(2.prop1),
$
\inf_{1\leq i'\leq d(n)}
\bigl|\bigl(r^{(n)}_{i'},m\bigr)\bigr|\geq C(m)>0
$
for $C(m)$ defined in \eqv(2.prop1.iii3). The proof of Lemma  \thv(4.lem1) is done.
\end{proof}

\begin{proof}[Proof of Lemma \thv(4.lem2)]  
By the exponential Chebyshev inequality, for all $t>0$
\be
\hspace{-1pt}\P\left(\AA_N(n,i,m,\l)\right)
\leq 
e^{-\l t}\E\left(\exp\left\{-\frac{t}{N}\sum_{n+1\leq \mu\leq M}\sum_{1\leq j\neq i \leq N} \xi^{\mu}_i\xi^{\mu}_j\xi_j(m)\xi_i(m)\right\}\right).
\Eq(4.1.14)
\ee
Recalling \eqv(4.1.1), we see that for a fixed realisation of the random variables $(\xi^{\nu}_k)_{1\leq \nu\leq n, 1\leq k\leq N}$ and 
$\left(\xi^{\mu}_i\right)_{\mu\geq n+1}$,  the variables  $\left(\zeta^{\mu}_j\right)_{1\leq i\neq j\leq N, n+1\leq \mu\leq M}$ defined as 
$\zeta^{\mu}_j= \xi^{\mu}_i\xi^{\mu}_j\xi_j(m)\xi_i(m)$, are independent and  have the same law as $\xi^{\mu}_j$. Integrating these variables first, we readily get
\be
\begin{split}
\P\left(\AA_N(n,i,m,\l)\right)
&= e^{-\l t}\left[\cosh\left({t}/{N}\right)\right]^{(M-n)(N-1)}
\\
&\leq \exp\left\{-\l t+\frac{1}{2}(M-n)(N-1)\left(\frac{t}{N}\right)^2\right\}
\\
&\leq \exp\left\{-\frac{1}{2}\frac{\l^2N^2}{(M-n)(N-1)}\right\},
\end{split}
\Eq(4.1.15)
\ee
where we used in the second line that $\cosh(x)\leq e^{x^2/2}$ for all $x\in\R$, while the last equality follows by minimising over $t>0$.
The proof of Lemma \thv(4.lem2) is done.
\end{proof}

We are now ready to complete the  proof of Theorem \thv(1.theo1). First assume that $n>1$.
By part (i) of Proposition \thv(3.prop1) with  $\d_N\equiv\bigl(8\frac{d(n)}{N}\ln N\bigr)^{1/2}$,
\be
\P\left[\xi_i(m)T_i(\xi(m))\neq1\right]=\P\left[\left\{\xi_i(m)T_i(\xi(m))\neq 1\right\}\cap \O^{[n]}\right].
\Eq(4.1.16)
\ee
By Lemma \thv(4.lem1), for all $N>N_0$, \emph{i.e.}~sufficiently large to ensure that $\EE^{[n]}_N\subseteq   \O^{[n]}$ in Proposition \thv(3.prop1), (i), we have 
\be
\begin{split}
\P\left[\left\{\xi_i(m)T_i(\xi(m))\neq 1\right\}\cap \O^{[n]}\right]
=
\P\left[\wh\AA_N(n,i,m)\cap \O^{[n]}\right],
\end{split}
\Eq(4.1.19)
\ee
where
\be
\begin{split}
&\wh\AA_N(n,i,m)
\\
\equiv
&
\left\{\sign\left[\left|\left(r^{(n)}_{i'},m\right)\right| +\xi_i(m)\left(-\frac{n}{N}+ \rho_{N}(m,n)\right)+ \xi_i(m)I_{N,i}^{(2)}(m)\right]\neq 1\right\}.
\end{split}
\Eq(4.1.18)
\ee
Now take 
$
\l\equiv\l_N(m,n)=C(m)-(n\d_N+nN^{-1})
$
in \eqv(4.1.12), where $C(m)$ is the constant defined in \eqv(2.prop1.iii3)-\eqv(2.prop1.iii4). It then follows from the bounds on $\rho_{N}(m,n)$ and $\bigl|\bigl(r^{(n)}_{i'},m\bigr)\bigr|$ of Lemma \thv(4.lem1) that
\be
\wh\AA_N(n,i,m)\subseteq \AA_N(n,i,m,\l_N(m,n)).
\Eq(4.1.20)
\ee
Eq.~\eqv(4.1.19) then yields
\be
\begin{split}
\P\left[\left\{\xi_i(m)T_i(\xi(m))\neq 1\right\}\cap \O^{[n]}\right]
&\leq
\P\left[ \AA_N(n,i,m,\l_N(m,n))\cap \O^{[n]}\right]
\\
&\leq
\P\left[ \AA_N(n,i,m,\l_N(m,n))\right].
\end{split}
\Eq(4.1.21)
\ee
Inserting this bound in \eqv(4.1.16) and using Lemma \thv(4.lem2), we obtain that for all large enough $N$
\be
\sum_{i=1}^N\P\left[\xi_i(m)T_i(\xi(m))\neq 1\right]
\leq 
N \exp\left\{-\frac{1}{2}\left[C(m)-(n+1)\d_N\right]^2\frac{N}{M-n}\right\}
\Eq(4.1.22)
\ee
Then, choosing $M=M(N)$ as in \eqv(1.theo1.1) and recalling that $\d_N=\bigl(8\frac{d(n)}{N}\ln N\bigr)^{1/2}$, we get
\be
\sum_{i=1}^N\P\left[\xi_i(m)T_i(\xi(m))\neq 1\right]
\leq 
N^{-(1+\varepsilon(1-o(1)))}.
\Eq(4.1.23)
\ee
Note that this bound is summable in $N$. The same bound also holds in the much simpler case $n=1$ where the right-hand side of \eqv(4.lem1.1) in Lemma \thv(4.lem1) reduces to $1-1/N$, while by  \eqv(2.prop1.iii3)-\eqv(2.prop1.iii4), $C(m)= 1$.  The claim of assertion (i) of the theorem thus follows from \eqv(4.1.3), \eqv(4.1.23) and  the Borel-Cantelli lemma.

To prove assertion (ii) we use the union bound to write, instead of \eqv(4.1.3),
\be
\begin{split}
1-\P\left[\left(\bigcap_{m}\bigl\{\xi(m)=T(\xi(m))\bigr\}\right)\right]
\leq 
\sum_{m}\sum_{i=1}^N\P\left[\xi_i(m)T_i(\xi(m))\neq 1\right],
\end{split}
\Eq(4.1.24)
\ee 
where the union and sum are over $m$ in $\MM_{n,F}$. Note that, under the condition \eqv(1.theo1.3) on $M$, choosing 
$\d^2_N=4\frac{d(n)}{N}(2+n)\ln N$, both $\d_N\rightarrow 0$ as $N\rightarrow \infty$ and  $\d^2_N>2\frac{d(n)}{N}(2\ln N+n\ln M)$, so that \eqv(3.prop1.2) of Proposition \thv(3.prop1), (ii), holds. Thus, for  this choice of   $\d_N$, Lemma \thv(4.lem1) holds with the first assertion of Proposition \thv(3.prop1) replaced by the second, and  $\O^{[n]}$ replaced by $\O^{(n)}$. From this point on, proceeding as in  the proof of assertion (i) of the theorem, \eqv(4.1.22) is unchanged, save for the choices of $\d_N$ and $M$, which yield 
\be
\sup_{m\in\MM_{n,F}}\sum_{i=1}^N\P\left[\xi_i(m)T_i(\xi(m))\neq 1\right]
\leq 
N^{-(1+n+\varepsilon)+o(1)},
\ee
Using the upper bound of Proposition \thv(1.prop.mix) to bound  the sum over $m$ in \eqv(4.1.24), the claim of assertion (ii) of the theorem is readily obtained. The proof of Theorem \thv(1.theo1) is now complete.

\subsection{Proof of Theorem \thv(1.theo2)} 
    \TH(S4.2)     
 
When $F(x)=\frac{1}{p}x^p$, $p>2$, the maps $T^{\text{G}}$ and $T^{\text{HK}}$ defined in \eqv(1.2.1) and \eqv(1.2.4) respectively become
\be
T^{\text{G}}_i(\s)=\sign\left\{\frac{1}{pN^{p-1}}\sum_{\mu=1}^M \xi^{\mu}_i\left(\sum_{j\neq i}\xi^{\mu}_j\s_j\right)^{p-1}\right\},
\Eq(4.2.1)
\ee
and
\be
T^{\text{HK}}_i(\s)
=
\sign\left\{
\sum_{\mu=1}^M
\frac{1}{pN^{p-1}}\left[
\left(\xi^{\mu}_i+\sum_{j\neq i}\xi^{\mu}_j\s_j\right)^p
-\left(-\xi^{\mu}_i+\sum_{j\neq i}\xi^{\mu}_j\s_j\right)^p
\right]
\right\}.
\Eq(4.2.2)
\ee
The strategy of the proof of Theorem \thv(1.theo2) closely mirrors that of Theorem \thv(1.theo1). 
Without loss of generality we can assume that $m\in\MM^{\circ}_{n,F}$ (see \eqv(1.2.1.14)), so 
that by \eqv(1.2.1.11), $\xi(m)$ is a mixture of the first $n$ patterns $(\xi^{\nu})_{1\leq \nu\leq n}$,
\be
\xi_i(m)=\sign\left(\sum_{\nu=1}^n\xi^{\nu}_i (m_\nu)^{p-1}\right),
\Eq(4.2.0)
\ee
where $m_{\nu}>0$ for all $1\leq \nu\leq n$. The role of the terms \eqv(4.1.6) is now played by the two terms
\be
\begin{split}
I_{N,p-1,i}^{(1)}(m) 
&\equiv
\sum_{\nu=1}^n \xi^{\nu}_i\left(\frac{1}{N}\sum_{j\neq i}\xi^{\nu}_j\xi_j(m)\right)^{p-1},
\\
I_{N,p-1,i}^{(2)}(m) 
&\equiv
\sum_{\mu=n+1}^M \xi^{\mu}_i\left(\frac{1}{N}\sum_{j\neq i}\xi^{\mu}_j\xi_j(m)\right)^{p-1}.
\\
\end{split}
\Eq(4.2.4)
\ee

In the following, we will slightly abuse the notation and write $m=(m_{\nu})_{1\leq \nu\leq n}\in \R^n$. 
(Whether $m\in \R^n$ or $m\in \R^M$ will always be clear from the context.) Given $m\in \R^n$, let $(m)^{p-1}: \R^n\rightarrow\R^n$ denote the vector  $(m)^{p-1}=((m_{\nu})^{p-1})_{1\leq \nu\leq n}$.
As for  Theorem \thv(1.theo1), the proof hinges on the next two lemmata.

\begin{lemma}[Leading term]
    \TH(4.lem3)
If $n=1$, then $\xi_i(m)I_{N,p-1,i}^{(1)}(m)=(1-1/N)^{p-1}$.
If $n> 1$, then under the assumptions and notations of Proposition \thv(3.prop1), (i) with $\{\mu_1,\dots, \mu_n\}=\{1,\dots, n\}$, the following holds on $\O^{[n]}$ for all  $m\in\MM^{\circ}_{n,F}$.  For all $1\leq i\leq N$ there exists $1\leq {i'}\leq d(n)$ such that
\be
\xi_i(m)I_{N,p-1,i}^{(1)}(m)= \left|\left(r^{(n)}_{i'},(m)^{p-1}\right)\right| +\xi_i(m)\tilde\rho_{N}(m,n,p),
\Eq(4.lem3.1)
\ee
where $\bigl|\bigl(r^{(n)}_{i'},(m)^{p-1}\bigr)\bigr|$ obeys the bound \eqv(2.prop2.iii3)-\eqv(2.prop2.iii4) of Proposition \thv(2.prop2) with $\boldsymbol{f}(m)=(m)^{p-1}$ and $\left|\tilde\rho_{N}(m,n,p)\right|\leq (p-1)(\d_N n+nN^{-1})(1+o(1))$.
\end{lemma}

Given  $\l>0$, define the events
\be
\AA^{(p)}_N(n,i,m,\l)
=\left\{
\xi_i(m)I_{N,p-1,i}^{(2)}(m) <-\l\right\}, \quad 1\leq i\leq N.
\Eq(4.lem4.0)
\ee
\begin{lemma}
    \TH(4.lem4)
The following holds for all $n$, all $m\in\MM^{\circ}_{n,F}$ and all $1\leq i\leq N$.
Let $\g(N)>0$ be a diverging function of $N$  satisfying $\g(N)=o(N^{1/6})$. For all $p\geq 3$, there exist constants $c_0>0$ and $0<t_0\leq 1/2$ such that,
\be
\P\left(\AA^{(p)}_N(n,i,m,\l)\right)
\leq 
c_0NMe^{-t_0\g^2(N)}
+
\frak{p}_{p,N}(\l),
\Eq(4.lem4.1)
\ee
where $\frak{p}_{p,N}(\l)$ obeys the following bounds for all $\l>0$. Set $\s_{2(p-1)}\equiv (2p-3)!!$.
\item{(i)} For all $k\geq 1$ there exists a constant $c(k)>0$ independent of $N$ such that 
\be
\frak{p}_{p,N}(\l)\leq c(k)\left(\s_{2(p-1)}\l^{-2}M N^{-(p-1)}\right)^k.
\Eq(4.lem4.2bis)
\ee
\item{(ii)} If ${M\gg N^{\frac{p-1}{2}}\g(N)^{p-1}}$ then
\be
\frak{p}_{p,N}(\l)
\leq
(1+o(1))\exp\left\{-(1+o(1))\frac{\l^2  N^{p-1}}{2M \s_{2(p-1)}} \right\}.
\Eq(4.lem4.2)
\ee
\end{lemma}

\begin{proof}[Proof of Lemma \thv(4.lem3)]  This is an elementary repetition of the proof of Lemma \thv(4.lem1), using Proposition \thv(2.prop2) with $\boldsymbol{f}(m)=(m)^{p-1}$ instead of Proposition \thv(2.prop1). We omit the details.
\end{proof}

\begin{proof}[Proof of Lemma \thv(4.lem4)] We have to estimate the tail probability of the sum
\be
\sum_{\mu=n+1}^M\xi^{\mu}_i\xi_i(m)\left(\frac{1}{N}\sum_{j\neq i}\xi^{\mu}_j\xi_j(m)\right)^{p-1}.
\Eq(4.lem4.3)
\ee
In view of  \eqv(4.2.0), the variables $(\xi_k(m))_{1\leq k\leq N}$  and $(\xi^{\mu}_k)_{n+1\leq \mu\leq M, 1\leq k\leq N}$ are independent. Moreover, for a fixed realisation of the variables $(\xi_k(m))_{1\leq k\leq N}$, the variables 
$\xi^{\mu}_k\xi_k(m)$ are independent and  have the same distribution as  $\xi^{\mu}_k$.  It follows from these observations that  the distribution of \eqv(4.lem4.3) is a symmetric and that
\be
\P\left(\AA^{(p)}_N(n,i,m,\l)\right)
=
\P\left(
\frac{1}{N^{\frac{p-1}{2}}}
\sum_{n+1\leq \mu\leq M}Y_i^{\mu}>\l\right),
\Eq(4.lem4.4)
\ee
where for $\mu\geq n+1$
\be
Y_i^{\mu}\equiv\xi^{\mu}_i\left(X_i^{\mu}\right)^{p-1}, \quad X_i^{\mu}\equiv\frac{1}{\sqrt N}\sum_{j\neq i}\xi^{\mu}_j.
\Eq(4.lem4.5)
\ee
Because the moment generating function of the variable $Y_i^{\mu}$ diverges for all $p\geq 3$, we cannot  use the exponential Chebyshev inequality directly as we did in \eqv(4.1.14). This difficulty is usually overcome by a truncation argument. Given $\g(N)$ to be chosen later, set
\be
\overline Y_i^{\mu}\equiv 
\begin{cases}
Y_i^{\mu} & \text{if $\left|Y_i^{\mu}\right|\leq \g^{p-1}(N)$,} \\
0 &  \text{otherwise.}
\end{cases}
\Eq(4.lem4.6)
\ee
Then,
\be
\P\left(\AA^{(p)}_N(n,i,m,\l)\right)
\leq 
\sum_{n+1\leq \mu\leq M}\P\left[\left|Y_i^{\mu}\right|> \g^{p-1}(N)\right]
+
\frak{p}_{p,N}(\l),
\Eq(4.lem4.7)
\ee
where
\be
\frak{p}_{p,N}(\l)
\equiv
\P\left(\frac{1}{N^{\frac{p-1}{2}}}\sum_{n+1\leq \mu\leq M}\overline Y_i^{\mu}>\l\right).
\Eq(4.lem4.7bis)
\ee

To bound the sum appearing in the right-hand side of \eqv(4.lem4.7), we use that by the exponential Chebyshev inequality, for all $t>0$
\be
\begin{split}
\P\left[\left|Y_i^{\mu}\right|> \g^{p-1}(N)\right]
&=
\P\left[
\left|Y_i^{\mu}\right|^{\frac{2}{p-1}}> \g^2(N)
\right]
\leq 
e^{-t\g^2(N)}
\E\left[e^{t\left(X_i^{\mu}\right)^{2}}
\right],
\end{split}
\Eq(4.lem4.8)
\ee
and note that the last expectation is equal to $2^{-N}Z_{2t,N}$, where $Z_{2t,N}$ is the the partition function of the Curie-Weiss model at inverse temperature $2t$ and zero magnetic field. Now, it is well know that if $t<1/2$,
\be
\E\left[e^{t\left(X_i^{\mu}\right)^{2}}\right]<c_0N,
\Eq(4.lem4.9)
\ee
(see, \emph{e.g.}~the bound (3.45) p.~43 in \cite{Bo06}).
Combining \eqv(4.lem4.8) and \eqv(4.lem4.9), we then get that
\be
\sum_{n+1\leq \mu\leq M}\P\left[\left|Y_i^{\mu}\right|> \g^{p-1}(N)\right]
\leq 
c_0NMe^{-t_0\g^2(N)},
\Eq(4.lem4.10)
\ee
for some constants $c_0>0$ and $0<t_0\leq1/2$. 

Now consider  \eqv(4.lem4.7bis). To bound this term, the underlying idea is to choose  the truncation threshold $\g(N)$ in such a way that the density of $\overline Y_i^{\mu}$ is well approximated by that of a standard Gaussian. A first bound can be derived from a second order Chebyshev inequality, \emph{i.e.}, using that the variables  $\left(\overline Y_i^{\mu}\right)_{n+1\leq \mu\leq M}$  are independent with mean $\E(\overline Y_i^{\mu})=0$.
\be
\begin{split}
\frak{p}_{p,N}(\l)
&\leq 
\left(\l N^{\frac{p-1}{2}}\right)^{-2}\sum_{n+1\leq \mu\leq M}\E\left(\overline Y_i^{\mu}\right)^2.
\end{split}
\Eq(4.lem4.11)
\ee
Next, 
\be
\begin{split}
\E\left(\overline Y_i^{\mu}\right)^2
&=
\sum_{x\in\SS_N: |x|\leq \g(N)} x^{2(p-1)}
\P\left(
X_i^{\mu}=x
\right),
\end{split}
\Eq(4.lem4.12)
\ee
where 
$
\SS_N=\bigl\{(-N+k)/\sqrt{N}: 0\leq k\leq 2N\bigr\}
$.
By de Moivre-Laplace local limit theorem (see, \emph{e.g.}~Lemma 2 p.~46 in \cite{CT}), if  $\g(N)=o\bigl(N^{1/6}\bigr)$,
\be
\begin{split}
\E\left(\overline Y_i^{\mu}\right)^2
\leq 
\s_{N, 2(p-1)}
\equiv (1+o(1))\sum_{x\in\SS_N: |x|\leq \g(N)} x^{2(p-1)}\sfrac{1}{\sqrt{2\pi N}}e^{-\frac{1}{2}x^2},
\end{split}
\Eq(4.lem4.13)
\ee
which is a Riemann sum, and so,
\be
\lim_{N\rightarrow\infty}\s_{N, 2(p-1)}=
\s_{2(p-1)}\equiv (2p-3)!!,
\Eq(4.lem4.14)
\ee
where $\s_{2(p-1)}$ is the moment of order $2(p-1)$ of a  standard Gaussian. We thus arrive at
\be
\frak{p}_{p,N}(\l)
\leq 
\s_{N, 2(p-1)}\l^{-2}M N^{-(p-1)}.
\Eq(4.lem4.15bis)
\ee
This bound can be improved by using Chebyshev's inequality of order $2k$, $k>1$. Proceeding as above, it follows from classical combinatorics that there is a constant  $c(k)>0$, independent of $N$, such that
\be
\frak{p}_{p,N}(\l)
\leq 
c(k)\left(\s_{N, 2(p-1)}\l^{-2}M N^{-(p-1)}\right)^k.
\Eq(4.lem4.15)
\ee
Clearly, this bound will be useful only when $M$ is small enough compared to $N^{p-1}$.

To complement this result, we use that by the exponential Chebyshev inequality, for all $t>0$
\be
\begin{split}
\frak{p}_{p,N}(\l)
&\leq 
e^{-\l t}\E\left[\exp\left(tN^{-\frac{p-1}{2}}\sum_{n+1\leq \mu\leq M}\overline Y_i^{\mu}\right)\right].
\end{split}
\Eq(4.lem4.16)
\ee
By \eqv(4.lem4.5)-\eqv(4.lem4.6), first  using the independence in $n+1\leq \mu\leq M$ and then integrating the variables $\xi^{\mu}_i$, 
\be
\frak{p}_{p,N}(\l)
\leq 
e^{-\l t}\left\{\E\left[\cosh\left(tN^{-\frac{p-1}{2}}\left(X_i^{\mu}\right)^{p-1}
\1_{\left\{\left| X_i^{\mu}\right|\leq \g(N)\right\}}\right)\right]\right\}^{M-n}.
\Eq(4.lem4.17)
\ee
Next, the expectation within braces in \eqv(4.lem4.17) is equal to
\be
\P\left[\left|Y_i^{\mu}\right|> \g^{p-1}(N)\right]
+
\frak{e}_{p,N}(t),
\Eq(4.lem4.18)
\ee
where
\be
\frak{e}_{p,N}(t)
\equiv\E\left[\1_{\left\{\left| X_i^{\mu}\right|\leq \g(N)\right\}}\cosh\left(tN^{-\frac{p-1}{2}}\left(X_i^{\mu}\right)^{p-1}\right)\right].
\Eq(4.lem4.19)
\ee
We first deal with $\frak{e}_{p,N}(t)$. Let us assume again that $\g(N)=o\bigl(N^{1/6}\bigr)$ and  that $t$ obeys
\be
t\bigl(\g(N)/\sqrt N\bigr)^{p-1}=o(1).
\Eq(4.lem4.20)
\ee
Then, since  $\cosh(x)\leq 1+\frac{1}{2}(1+\OO(x^2))x^2$ for all $|x|<1$, we have for $\SS_N$ as defined  in \eqv(4.lem4.12) and reasoning as in \eqv(4.lem4.13) 
\be
\begin{split}
\frak{e}_{p,N}(t)
&=
\sum_{x\in\SS_N: |x|\leq \g(N)} 
\left[1+\frac{1}{2}(1+o(1))tN^{-\frac{p-1}{2}}x^{2(p-1)}\right]
\P\left(X_i^{\mu}=x\right)
\\
&
\leq
1+\frac{1}{2}(1+o(1))t^2N^{-(p-1)} \s_{N, 2(p-1)}
\leq 
e^{\frac{1}{2}(1+o(1))t^2N^{-(p-1)}\s_{N, 2(p-1)}}.
\end{split}
\Eq(4.lem4.21)
\ee
Inserting this bound into \eqv(4.lem4.18), and noting that the first probability in  \eqv(4.lem4.18)  has already been bounded in \eqv(4.lem4.8)-\eqv(4.lem4.9), we get, by plugging the resulting bound in \eqv(4.lem4.17)
\be
\begin{split}
\frak{p}_{p,N}(\l)
&\leq 
e^{-\l t}\left\{
c_0Ne^{-t_0\g^2(N)}+e^{\frac{1}{2}(1+o(1))t^2N^{-(p-1)} \s_{N, 2(p-1)}}
\right\}^{M-n}
\\
&
\leq 
e^{-\l t + M\frac{1}{2}(1+o(1))t^2N^{-(p-1)} \s_{N, 2(p-1)}}\left[\left\{1+c_0Ne^{-t_0\g^2(N)}\right\}^{M-n}\right].
\end{split}
\Eq(4.lem4.22)
\ee
Under the assumption that $MNe^{-t_0\g^2(N)}=o(1)$, the term within square brackets is of order one, while optimising the exponential pre-factor over $t$ leads to the choice
\be
t=\frac{\l  N^{p-1}}{(1+o(1))M \s_{N, 2(p-1)}}.
\Eq(4.lem4.23)
\ee
This yields
\be
\frak{p}_{p,N}(\l)
\leq 
(1+o(1))\exp\left\{-(1+o(1))\frac{\l^2  N^{p-1}}{2M \s_{N, 2(p-1)}} \right\},
\Eq(4.lem4.24)
\ee
provided that $t$ in \eqv(4.lem4.23) satisfies \eqv(4.lem4.20), \emph{i.e.}~provided that
$
M\gg N^{\frac{p-1}{2}}\g(N)^{p-1}
$.

Finally, by inserting \eqv(4.lem4.10), \eqv(4.lem4.15) and  \eqv(4.lem4.24) into \eqv(4.lem4.7), and remembering from \eqv(4.lem4.14) that $\s_{N, 2(p-1)}=(1+o(1))(2p-3)!!$, we obtain the claim of Lemma  \thv(4.lem4).
\end{proof}

We are now equipped to prove Theorem \thv(1.theo2).  Consider first the case  $T=T^{\text{G}}$. As for Theorem \thv(1.theo1), we seek an upper bound on the quantity that appears in the last line of \eqv(4.1.3). To do this, we proceed as in  \eqv(4.1.5)-\eqv(4.1.6) and rewrite $T_i(\xi(m))$ with the help of \eqv(4.2.4) as
\be
\begin{split}
T_i(\xi(m))
=
\sign\left\{
I_{N,p-1,i}^{(1)}(m) + I_{N,p-1,i}^{(2)}(m) 
\right\}.
\end{split}
\Eq(4.lem4.34)
\ee
The proof of assertion (i) is then a repetition of the proof of assertion (i) of Theorem \thv(1.theo1), using Lemma \thv(4.lem3) and Lemma \thv(4.lem4) instead of \thv(4.lem1) and Lemma \thv(4.lem2). The main difference is that we must first  choose the function $\g(N)$  in Lemma \thv(4.lem4). Taking $\g^2(N)=(p+3)\ln N/t_0$ guarantees that for all $M\leq N^{p-1}$ the first term in the right-hand side of \eqv(4.lem4.1) is bounded above by $c_0N^{-3}$. Clearly, for this choice, the conditions on $\g(N)$ for \eqv(4.lem4.2) to hold are verified. The claim of assertion (i) of Theorem \thv(1.theo2) then easily follows.

Consider now the case  $T=T^{\text{HK}}$. Using the binomial theorem to expand the terms 
$\bigl(\pm\xi^{\mu}_i+\sum_{j\neq i}\xi^{\mu}_j\s_j\bigr)^p$ in \eqv(4.2.2), we have
\be
\begin{split}
T_i(\s)
=
\sign\left\{
\frac{1}{pN^{p-1}}\sum_{\mu=1}^M
\sum_{\substack{1\le k\leq p:\\ k\, \text{odd}}}
\xi^{\mu}_i\left(\sum_{j\neq i}\xi^{\mu}_j\s_j\right)^{p-k}{p\choose k}
\right\}.
\end{split}
\Eq(4.lem4.35)
\ee
By this and \eqv(4.2.4),
\be
\begin{split}
T_i(\xi(m))
=
\sign\left\{
\sum_{\substack{1\le k\leq p:\\ k\, \text{odd}}}{p\choose k}
\frac{1}{pN^{k-1}}
\left[I_{N,p-k,i}^{(1)}(m) + I_{N,p-k,i}^{(2)}(m) \right]
\right\}.
\end{split}
\Eq(4.lem4.36)
\ee
We see that the term $k=1$ in \eqv(4.lem4.36) is nothing else than the argument of the sign function in  \eqv(4.lem4.34), whereas terms with $k\geq 3$ have an  $N^{-(k-1)}$ prefactor. It is therefore sufficient to show that the contribution of the latter is negligible compared to the term $k=1$, \emph{i.e.}~compared to 1.

Given  $\l'>0$, define the collection of events indexed by $1\leq i\leq N$
\be
\BB^{(p)}_N(n,i,m,\l')=\left\{
\sum_{\substack{3\le k\leq p:\\ k\, \text{odd}}}
{p\choose k}\frac{1}{pN^{k-1}} 
\left| I_{N,p-k,i}^{(1)}(m) + I_{N,p-k,i}^{(2)}(m)\right|
\geq \frac{\l'}{N}\right\}.
\Eq(4.lem4bis.0)
\ee

\begin{lemma}
    \TH(4.lem4bis)
    
The following holds for all $n$, all $m\in\MM^{\circ}_{n,F}$ and all $1\leq i\leq N$.  If $M(N)$ satisfies \eqv(1.theo2.1), then there exist constants $c_p, \tilde c_p>0$ which depend only on $p$, such that 
\be
\P\left(\BB^{(p)}_N(n,i,m,c_pC_p(m))\right)
\leq 
\frac{ \tilde c_p}{N^{2+\varepsilon}},
\Eq(4.lem4bis.1)
\ee
where $C_p(m)$ denotes the constant obtained by taking $\boldsymbol{f}(m)=(m)^{p-1}$ in \eqv(2.prop2.iii3)-\eqv(2.prop2.iii4).
Similarly, If $M(N)$ satisfies \eqv(1.theo2.3), then there exist constants $c_p, \tilde c_p>0$ which depend only on $p$, such that 
\be
\P\left(\BB^{(p)}_N(n,i,m,c_pC_p(m))\right)
\leq 
\frac{ \tilde c_p}{N^{2+n(p-1)+\varepsilon}}.
\Eq(4.lem4bis.1')
\ee
\end{lemma}

\begin{proof}[Proof of Lemma \thv(4.lem4bis)]  
On the one hand, we have the trivial deterministic bound 
\be
\left|I_{N,p-k,i}^{(1)}(m)\right|\leq n,\quad 3\leq k\leq p.
\Eq(4.lem4bis.2)
\ee
On the other hand, for all $3\leq k< p$ odd, by Lemma \thv(4.lem4) with $\g^2(N)=(p+3)\ln N/t_0$, we have that if $M(N)$ satisfies \eqv(1.theo2.1),  setting $t_k^2=C^2_p(m)\frac{(2(p-k)-1)!!}{p(2p-3)!!}$, 
\be
\begin{split}
&
\P\left(
\frac{1}{pN^{k-1}} 
\left|
I_{N,p-k,i}^{(2)}(m)\right|
\geq
 \frac{t_k}{\sqrt{pN^{k-1}}}
\right)
\\
\leq &
2
\P\left(\AA^{(p-k)}_N\left(n,i,m,\l=t_k\sqrt{pN^{k-1}}\right)\right)
\leq 
\frac{\hat c_p }{N^{2+\varepsilon}},
\end{split}
\Eq(4.lem4bis.4)
\ee
for some constant $\hat c_p>0$ that depends on $p$. Lastly, the case $p$ odd and $k=p$ must be treated separately. Here $I_{N,0,i}^{(2)}(m)=\sum_{\mu=n+1}^M \xi^{\mu}_i$, and by Lemma \thv(4.lem2), for all $t>0$ 
\be
\P\left(
\left|\frac{1}{pN^{p-1}}\sum_{\mu=n+1}^M \xi^{\mu}_i\right|
\geq \frac{t}{\sqrt{pN^{p-1}}}
\right)
\leq
2\exp\left\{-\frac{t^2}{2}\frac{N^{p-1}}{M-n}\right\}.
\Eq(4.lem4bis.3)
\ee
If $M(N)$ satisfies \eqv(1.theo2.1), choosing $t^2=C^2_p(m)/(2p-3)!!$, the right-hand side of  \eqv(4.lem4bis.3) is bounded above by
$
\frac{2 }{N^{2+\varepsilon}}
$.
From these bounds, the claim of \eqv(4.lem4bis.1) readily follows. The proof of \eqv(4.lem4bis.1') is a rerun of the proof of \eqv(4.lem4bis.1), using \eqv(1.theo2.3) instead of \eqv(1.theo2.1).
\end{proof}

It follows from Lemma \thv(4.lem4bis) that on $\left(\BB^{(p)}_N(n,i,m,c_pC_p(m))\right)^c$, with a  probability larger than 
$1-{ \tilde c_p}/{N^{2+\varepsilon}}$, 
\be
\begin{split}
T_i(\xi(m))
=
\sign\left\{
I_{N,p-1,i}^{(1)}(m) + I_{N,p-1,i}^{(2)}(m)+ \bar\rho_{N}(m,n,p)
\right\}.
\end{split}
\Eq(4.lem4bis.5)
\ee
where $\bar\rho_{N}(m,n,p)=\OO(c_pC_p(m)/N)$.
The proof of the theorem in the case $T=T^{\text{HK}}$ is thus reduced to its proof in the case $T=T^{\text{G}}$. We omit the details. The proof of assertion (i) of the theorem is done.

As in Theorem \thv(1.theo1), the proof of assertion (ii) is a minor modification of assertion (i). First observe that if
$M(N)$ is given by \eqv(1.theo2.3), then the choice
\be
\d^2_N
= 
4d(n)(1+n(p-1))\frac{\ln N}{N}
\ee
satisfies both $\d_N\downarrow 0$ as $N\uparrow \infty$ and
$
\d^2_N
>
2\frac{d(n)}{N}(2\ln N+n\ln M)
$
for all $N$ sufficiently large. Thus, \eqv(3.prop1.2) of Proposition \thv(3.prop1), (ii), holds. Based on this, the proof of assertion (ii) of Theorem \thv(1.theo2) when $T=T^{\text{G}}$ is a repetition  of the proof of assertion (ii) of Theorem \thv(1.theo1). When $T=T^{\text{HK}}$, we use in addition that by \eqv(4.lem4bis.1') of Lemma \thv(4.lem4bis),
\be
\sum_{m\in\MM_{n,F}}\P\left(\BB^{(p)}_N(n,i,m,c_pC_p(m))\right)
\leq 
\frac{ \tilde c_p}{N^{2+\varepsilon}},
\Eq(4.lem4bis.1'new)
\ee
which implies that  $T^{\text{HK}}$ can be reduced to $T^{\text{G}}$ uniformly in $m\in\MM_{n,F}$ with $\P$-probability one, for all sufficiently large $N$.
Again, we omit the details. The proof of Theorem \thv(1.theo2) is complete.  

\subsection{Proof of Theorem \thv(1.theo3)} 
    \TH(S4.3)   

Throughout this section, $F(x)=\exp\{N\b x\}$ for a given $\b>0$. To check that $T^{\text{HK}}=T^{\text{G}}$, simply observe that for each $\mu$
\be
\begin{split}
&\textstyle F\left(\frac{1}{N}\left[\xi^{\mu}_i+\sum_{ j\neq i}\xi^{\mu}_j\s_j\right]\right)
-
F\left(\frac{1}{N}\left[-\xi^{\mu}_i+\sum_{ j\neq i}\xi^{\mu}_j\s_j\right]\right)
\\
=&
e^{\b\left\{\xi^{\mu}_i+\sum_{ j\neq i}\xi^{\mu}_j\s_j\right\}}-
e^{\b\left\{-\xi^{\mu}_i+\sum_{ j\neq i}\xi^{\mu}_j\s_j\right\}}
=
2\sinh\left(\b\xi^{\mu}_i\right)e^{\b\sum_{ j\neq i}\xi^{\mu}_j\s_j}
\\
=&
N^{-1}\textstyle \left(2\b^{-1}\sinh\b\right)\xi^{\mu}_iF'\left(\frac{1}{N}\sum_{ j\neq i}\xi^{\mu}_j\s_j\right).
\end{split}
\Eq(4.3.1)
\ee
Since $2\sinh\left(\b\xi^{\mu}_i\right)>0$, it follows from the definitions \eqv(1.2.1) and \eqv(1.2.4) that $T^{\text{HK}}=T^{\text{G}}$.  Let us call this map $T$. Again, the stucture of the proof of Theorem \thv(1.theo3) is very similar to that of Theorems \thv(1.theo1) and  \thv(1.theo2).  There is no loss of generality in assuming that $m\in\MM^{\circ}_{n,F}$ (see \eqv(1.2.1.14)).

As before, will slightly abuse the notation and write $m=(m_{\nu})_{1\leq \nu\leq n}\in \R^n$ whenever no confusion is possible. 
Given $m=(m_{\nu})_{1\leq \nu\leq n}\in \R^n$, let $e^{\b N m}: \R^n\rightarrow\R^n$ denote the vector  
$e^{\b N m}=(e^{\b N m_{\nu}})_{1\leq \nu\leq n}$. The role of the terms \eqv(4.1.6) (or \eqv(4.2.4)) is then played by the two terms
\be
\begin{split}
I_{N,\b,i}^{(1)}(m) 
&\equiv
\sum_{\nu=1}^n \xi^{\nu}_i\exp\left\{\b\sum_{j\neq i}\xi^{\nu}_j\xi_j(m)\right\},
\\
I_{N,\b,i}^{(2)}(m) 
&\equiv
\sum_{\mu=n+1}^M \xi^{\mu}_i \exp\left\{\b\sum_{j\neq i}\xi^{\mu}_j\xi_j(m)\right\}.
\end{split}
\Eq(4.3.4)
\ee

The proof of Theorem \thv(1.theo3) now relies on the next three lemmata.

\begin{lemma}[Leading term]
    \TH(4.lem6)

If $n=1$, then $\xi_i(m)I_{N,\b,i}^{(1)}(m)=e^{\b(N-1)}$.
If $n> 1$, then under the assumptions and notations of Proposition \thv(3.prop1), (i) with $\{\mu_1,\dots, \mu_n\}=\{1,\dots, n\}$, the following holds on $\O^{[n]}$ for all  $m\in\MM^{\circ}_{n,F}$.  For all $1\leq i\leq N$ there exists $1\leq {i'}\leq d(n)$ such that
\be
\xi_i(m)I_{N,\b,i}^{(1)}(m)= \left|\left(r^{(n)}_{i'},e^{\b Nm}\right)\right| e^{\b N \hat\rho_{N}(m,n)},
\Eq(4.lem6.1)
\ee
where
$
\left|\hat\rho_{N}(m,n)\right|\leq \d_N+\frac{1}{N}
$
and $\bigl|\bigl(r^{(n)}_{i'},e^{\b N m}\bigr)\bigr|$ obeys the bound \eqv(2.prop2.iii3)-\eqv(2.prop2.iii4) of Proposition \thv(2.prop2) 
with and $\boldsymbol{f}(m)=(f(m_\nu))_{1\leq \nu\leq n}$ and $f(x)=e^{\b N x}\1_{x\neq 0}$.
\end{lemma}

Given  $\l>0$, define the events
\be
\AA_{N,\b}(n,i,m,\l)
=\left\{
\xi_i(m)I_{N,\b,i}^{(2)}(m) <-\l\right\}, \quad 1\leq i\leq N.
\Eq(4.lem7.0)
\ee
\begin{lemma}
    \TH(4.lem7)
For all $n$, all $m\in\MM^{\circ}_{n,F}$, all $1\leq i\leq N$, all $\l>0$ and all $c>0$
\be
\P\left(\AA_{N,\b}(n,i,m,\l)\right)
\leq 
\P\left(\sum_{\mu=n+1}^Me^{2\b\sum_{j\neq i}\xi^{\mu}_j}>\frac{\l^2}{2cN}\right)+e^{-cN}.
\Eq(4.lem7.1)
\ee
\end{lemma}

Recall that $I$ is the entropy function defined in \eqv(1.theo3.0).
\begin{lemma}
    \TH(4.lem9)
Given $0<\a\leq \ln 2$ and $n\in\N$, let 
$
M(N)=\lfloor e^{\a N}\rfloor+n
$.  The following holds for all $n'\in\N$.
\item{(i)} For all $\varrho\in [0, 1/2)$ satisfying $\a<\inf\left\{\frac{1}{2}, \frac{1}{n'+1}\right\}\inf\left\{b,1\right\}I(1-2\varrho)$ and all $b> 0$, 
\be
NM^{n'}\P\left(\sum_{\mu=n+1}^Me^{b\sum_{j\neq i}\xi^{\mu}_j}> e^{Nb(1-2\varrho)}\right)
<
(1+o(1))Ne^{-N[I(1-2\varrho)-(n'+1)\a]}.
\Eq(4.lem9.1)
\ee
\item{(ii)} For all $\varrho\in [0, 1/2)$ and all $b>0$, 
\be
\begin{split}
NM^{n'}\P\left(\sum_{\mu=n+1}^Me^{b\sum_{j\neq i}\xi^{\mu}_j}> e^{Nb(1-2\varrho)}\right)
\leq NM^{n'+1}e^{-Nb(1-2\varrho)}[\cosh(b)]^{N-n}.
\end{split}
\Eq(4.lem9.2)
\ee
\end{lemma}

\begin{proof}[Proof of Lemma \thv(4.lem6)] 
As in the proof of Lemma \thv(4.lem3), this is an elementary rerun of the proof of Lemma \thv(4.lem1), using Proposition \thv(2.prop2) 
with $\boldsymbol{f}(m)=(f(m_\nu))_{1\leq \nu\leq n}$ and $f(x)=e^{\b N x}\1_{x\neq 0}$, instead of Proposition \thv(2.prop1). We omit the details.
\end{proof}

\begin{proof}[Proof of Lemma \thv(4.lem7)] 
Following the same line of reasoning as in the paragraph below \eqv(4.lem4.3), and observing that the distribution of $\xi_i(m)I_{N,\b,i}^{(2)}(m)$ is symmetric with mean zero we have, introducing the $\s$-field 
$\FF_{i,n}\equiv\s\left\{(\xi^{\mu}_j)_{n+1\leq \mu\leq M, 1\leq j\neq i\leq N}\right\}$,
\be
\begin{split}
\P\left(\AA_{N,\b}(n,i,m,\l)\right)
&
=\P\left(\sum_{\mu=n+1}^M\xi^{\mu}_ie^{\b\sum_{j\neq i}\xi^{\mu}_j}>\l\right)
\\
&
=
\E\left[
\P\left(\sum_{\mu=n+1}^M\xi^{\mu}_ie^{\b\sum_{j\neq i}\xi^{\mu}_j}>\l\big| \FF_{i,n}\right)
\right]
\\
&
\leq 
\E\left[\exp\left(-\frac{\l^2}{2Z_{\b,N,M}}\right)\right],
\end{split}
\Eq(4.lem7.2)
\ee
where
$
Z_{\b,N,M}\equiv\sum_{\mu=n+1}^Me^{2\b\sum_{j\neq i}\xi^{\mu}_j}
$.
The last line of \eqv(4.lem7.2) follows from the well know fact (see (4.1) in \cite{LT}) that for all finite sequence $(\a_{\mu})$ of real numbers  and every $\l>0$,
\be
\P\left(\sum_{n+1\leq \mu\leq M}\a^{\mu}\xi^{\mu}_i>\l\right)\leq \exp\left(-\frac{\l^2}{2\sum_{n+1\leq \mu\leq M}\a_{\mu}^2}\right).
\Eq(4.lem7.3)
\ee
Given any $\bar\l>0$, it readily follows from the identity
$
1=\1_{\left\{Z_{\b,N,M}>\bar\l\right\}}+\1_{\left\{Z_{\b,N,M}\leq \bar\l\right\}}
$
that
\be
\begin{split}
\E\left[\exp\left(-\frac{\l^2}{2Z_{\b,N,M}}\right)\right]
&\leq 
\P\left(Z_{\b,N,M}>\bar\l\right)+\exp\left(-\frac{\l^2}{2\bar\l}\right).
\end{split}
\ee
Taking $\bar\l=(2cN)^{-1}\l^2$ for some $c>0$ and inserting the resulting bound into \eqv(4.lem7.2) gives \eqv(4.lem7.1). This completes the proof of the lemma.
\end{proof}

\begin{proof}[Proof of Lemma  \thv(4.lem9)]  
Assertion (ii) follows from a first order Chebyshev inequality. Consider now assertion (i). The case $b=1$, $n=1$ and $n'=1$  is proved in \cite{DHLUV17} (more precisely, see the estimate of the probability on the right-hand side of (5), p.~295, which starts there and is completed on p.~298). Multiplying both sides of this estimate by $M^{n'}$ gives \eqv(4.lem9.2) in the case $b=1$, $n=1$ and $n'>1$. The extension of this result to $n>1$ is elementary. Finally, its extension to the case $b\neq 1$ is a simple adaptation of the proof of \cite{DHLUV17}. More precisely, the parameter $b$ only affects the condition appearing in the last paragraph of p.~297, which requires that $\a+\b'<1-2\varrho$, where $\b'$ (called $\b$ in \cite{DHLUV17}) comes from a truncation argument.  If $b\neq 1$, this condition becomes $\a/b+\b'<1-2\varrho$. Therefore, the result of \cite{DHLUV17} is unchanged if $b\geq 1$, whereas if $b< 1$, one easily sees, by going through the proof of \cite{DHLUV17}, that taking $\a<(b/2)I(1-2\varrho)$ guarantees that $\a/b+\b'<1-2\varrho$ holds.
\end{proof}
\begin{remark} Improving the estimates of Lemma \thv(4.lem9) would require a thorough treatment of the sum $\sum_{\mu=n+1}^Me^{b\sum_{j\neq i}\xi^{\mu}_j}$. This is the partition function of a REM at the inverse temperature $b$, where the Gaussians have been replaced by binomial random variables, and with a varying number $M\equiv M(N)$ of summands. Unfortunately, replacing the Gaussians with binomials makes the calculations cumbersome. This is beyond the scope of this paper and will be done elsewhere.
\end{remark}

We now return to the proof of the theorem.

\begin{proof}[Proof of Theorem \thv(1.theo3)]
It is a simple adaptation of the proof of Theorem \thv(1.theo1), using Lemma \thv(4.lem6) and Lemma \thv(4.lem7) instead of \thv(4.lem1) and Lemma \thv(4.lem2). We only indicate the modifications.  Eq.~\eqv(4.1.5) is rewritten with the help of \eqv(4.3.4) as
\be
\begin{split}
\xi_i(m)T_i(\xi(m))
=
\sign\left\{
\xi_i(m)I_{N,\b,i}^{(1)}(m)  +\xi_i(m)I_{N,\b,i}^{(2)}(m)  
\right\}.
\end{split}
\Eq(1.theo3.20)
\ee
The event $\wh\AA_N(n,i,m)$ appearing in \eqv(4.1.19) and defined in  \eqv(4.1.18) is replaced by
\be
\wh\AA_{\b,N}(n,i,m)
\equiv
\left\{\sign\left[ \left|\left(r^{(n)}_{i'},e^{\b Nm}\right)\right| e^{\b N \hat\rho_{N}(m,n)}+ \xi_i(m)I_{N,i}^{(2)}(m)\right]\neq 1\right\},
\Eq(1.theo3.21)
\ee
and we take
$
\l\equiv\l_{\b, N}(m,n)=C_{\b,N}(m)e^{\b N \hat\rho_{N}(m,n)}
$
in \eqv(4.lem7.0), where $C_{\b,N}(m)$ denotes the constant obtained by taking 
$\boldsymbol{f}(m)=(f(m_\nu))_{1\leq \nu\leq n}$ and $f(x)=e^{\b N x}\1_{x\neq 0}$ in \eqv(2.prop2.iii3)-\eqv(2.prop2.iii4). Let us note at once that for $\b_c(m)$ as defined in \eqv(1.theo3'.1) we have the

\begin{lemma}
   \TH(4.lem8)
For all $m\in\MM^{\circ}_{n,F}$, all $\b>0$ and all large enough $N$
\be
C_{\b,N}(m)\geq e^{N\b \b_c(m)}.
\Eq(4.lem8.1)
\ee
\end{lemma}

\begin{proof}[Proof of Lemma \thv(4.lem8)]  Recall from Proposition \thv(2.prop2) that the vector $m=(m_{\nu})_{1\leq \nu\leq n}$ takes the form \eqv(2.prop2.2) where the elements of the sequence $\left(\g^{(k)}\right)_{1\leq k\leq \ell}$ are defined in \eqv(2.prop1.1) as products of terms, $\a^{(s)}$, themselves defined in \eqv(2.lem1.1). Since $\a^{(s)}<1$ for $n\geq 3$ then $\left(\g^{(k)}\right)_{1\leq k\leq \ell}$  is a strictly decreasing sequence so that for all sufficiently large $N$
\be
\min\Bigl\{
\min_{1\leq k\leq \ell-1}\left\{
2e^{\b N\g^{(k)}}-\left[s_{k+1}e^{\b N\g^{(k+1)}}+\dots+s_{\ell}e^{\b N\g^{(\ell)}}
\right]\right\}, 
e^{\b N\g^{(\ell)}}
\Bigr\}
=
e^{\b N\g^{(\ell)}}.
\nonumber
\ee
Clearly,  for $n=1$, $C_{\b,N}(m)= e^{\b N}$. The claim of the lemma now follows from \eqv(2.prop2.iii4).
\end{proof}

By Lemma \thv(4.lem6), remembering \eqv(1.theo3.21), \eqv(4.1.20) is replaced by
\be
\wh\AA_{\b,N}(n,i,m)\subseteq \AA_{N,\b}(n,i,m,\l_{\b, N}(m,n)).
\Eq(1.theo3.22')
\ee
Next, by  Lemma \thv(4.lem7), 
\be
\begin{split}
\P\left(\wh\AA_{\b,N}(n,i,m)\right)
&\leq 
\P\left(\sum_{\mu=n+1}^Me^{2\b\sum_{j\neq i}\xi^{\mu}_j}>\frac{C^2_{\b,N}(m)}{2cN}e^{2\b N \hat\rho_{N}(m,n)}\right)
+e^{-cN}
\\
&\leq 
\P\left(\sum_{\mu=n+1}^Me^{2\b\sum_{j\neq i}\xi^{\mu}_j}
> e^{-N2\b\b_c(m)(1-\zeta_N)}\right)
+e^{-cN}.
\end{split}
\Eq(1.theo3.23')
\ee
where
$
\zeta_N\equiv\b^{-1}_c(m)\left(\sqrt{\frac{2^n}{N}}+\frac{1}{N}+\frac{\ln(2cN)}{2\b N}\right)
$,
which follows from the facts that $ C_{\b,N}(m)$ satisfies \eqv(4.lem8.1),  $\left|\hat\rho_{N}(m,n)\right|\leq \d_N+\frac{1}{N}$ and  $\d_N\equiv(d(n)/N)^{1/2}$, $d(n)=2^n$. 
Thus, \eqv(4.1.22) is replaced by
\be
\begin{split}
\sum_{i=1}^N\P\left[\xi_i(m)T_i(\xi(m))\neq 1\right]
\leq &
NM \P\left(\sum_{\mu=n+1}^Me^{2\b\sum_{j\neq i}\xi^{\mu}_j}
> e^{-N2\b\b_c(m)(1-\zeta_N)}\right)
\\
+&NM e^{-cN}.
\end{split}
\Eq(1.theo3.24')
\ee
Theorem \thv(1.theo3) will now follow from assertion (i) of Lemma \thv(4.lem9) with $b=2\b$, $n'=0$, $1-2\varrho=\b_c(m)(1-\zeta_N)$ ($n$ being the same in  Lemma \thv(4.lem9) and in the theorem). Note that for all $0<\b_c(m)\leq 1$ and $\zeta_N>0$, 
$
|I(\b_c(m)(1-\zeta_N))-I(\b_c(m))|\leq \OO(|\zeta_N\ln\zeta_N|)
$.
Thus, choosing $M$ as in \eqv(1.theo3'.2bis) for arbitrarily small $\varepsilon>0$ and 
$
c=\inf\left\{\b,\frac{1}{2}\right\}I(\b_c(m))
$,
so that $c<\frac{1}{2}$,
it follows from \eqv(1.theo3.24') that for all sufficiently large $N$,
\be
\begin{split}
\sum_{i=1}^N\P\left[\xi_i(m)T_i(\xi(m))\neq 1\right]
\leq 
&
e^{-\frac{1}{2}N\varepsilon\inf\left\{\b,\frac{1}{2}\right\}}
+
Ne^{-N\varepsilon\inf\left\{\b,\frac{1}{2}\right\}},
\end{split}
\Eq(1.theo3.25')
\ee
which is summable. The proof Theorem \thv(1.theo3) is then completed as the proof of assertion  (i) of Theorem \thv(1.theo1).
\end{proof}

We can now make the remark immediately below Theorem \thv(1.theo3) more explicit by observing that the condition for the validity of Proposition \thv(3.prop1), (ii) cannot be satisfied, since $\ln M/N$ does not decay to zero as $N$ diverges. It now remains to prove Lemma \thv(1.theo3.n=1).

\begin{proof}[Proof of Lemma \thv(1.theo3.n=1)] 
Assume that $n=1$. Then, by Lemma \thv(4.lem6) with $n=1$, Lemma \thv(4.lem7) with $\l=e^{\b(N-1)}$ and $c>0$ to be chosen,
it follows from  assertion (ii) of Lemma \thv(4.lem9) with $b=2\b$, $n'=0$ and $e^{Nb(1-2\varrho)}=(2cN)^{-1}e^{2\b(N-1)}$, that
\be
\begin{split}
\sum_{i=1}^N\P\left[\xi_i(m)T_i(\xi(m))\neq 1\right]
\leq 
&
 M(2c)^{-1}e^{-(N-1)\left[2\b-\ln\cosh(2\b)\right]}+NM e^{-cN}.
\end{split}
\Eq(1.theo3.26')
\ee
For any $\d>0$, this can be made smaller than $N^{-(1+\d)}$  by choosing $c=2\ln 2$ and
\be
M(N)
\leq 2N^{-(1+\d)}e^{(N-1)\left[2\b-\ln\cosh(2\b)\right]}
=2N^{-(1+\d)}\left(\sfrac{2}{1+e^{-4\b}}\right)^{N-1}.
\Eq(1.theo3.27')
\ee
This yields \eqv(1.theo3'.2ter). 
\end{proof}


\section{Solving  $S_{n,F}$ for the classical and dense models, and consequences for the energy function}
    \TH(S5)
    
One may wonder how difficult it is to find solutions to the system of inequalities $\SS_{n,F}$ defined in \eqv(1.2.1.12bis) for a given model and how the energy of the corresponding mixed memories behaves, in particular how deep the corresponding minima of the energy are. In this section, we answer these questions for the classical and (in part) for the dense models.

Since the construction of the elements of $\MM_{n,F}$ is centred on the coefficients $\a^{(n_{k})}$ in \eqv(1.2.1.4), we start by giving their basic properties. First, it is easy to check that $\a^{(2k)}$ is decreasing whereas $k\a^{(2k)}$ is increasing. Recalling \eqv(2.rem2.1), its first few values are
\bea
\a^{(1)}&=1,\,\,\,
\nonumber
\\
\a^{(3)}&=\a^{(2)}&=\frac{1}{2},
\nonumber
\\
\a^{(5)}&=\a^{(4)}&=\frac{3}{8},
\Eq(5.1)
\\
\a^{(7)}&=\a^{(6)}&=\frac{5}{16},
\nonumber
\\
\a^{(9)}&=\a^{(8)}&=\frac{35}{128}.
\nonumber
\eea
Furthermore, all $k\geq 2$, 
\be
\a^{(2k+1)}= \a^{(2k)}=\frac{e^{-\varepsilon_{2k}}}{\sqrt{\pi k}}\,\, \text{ where }\,\,  \frac{1}{12 k+1}\leq \varepsilon_{2k}\leq \frac{1}{6 k}.
\Eq(5.lem0.1)
\ee
This estimate follows from the Stirling formula (see, \emph{e.g.}~Lemma 1 in Section 2.3 of  \cite{CT}).

\subsection{The classical Hopfield model} 
    \TH(S5.1)     
    
Throughout this section $F(x)=\frac{1}{2}x^2$. We use the notation of  Section \thv(S1.2.1). Recall in particular that each  allowable $\ell$-composition induces a sequence $\g_n=\left(\g^{(k)}\right)_{1\leq k\leq \ell}$ through \eqv(1.2.1.4)-\eqv(1.2.1.5) and that $\G_{n,F}$ is the set of sequences $\g_n$ satisfying the system $\SS_{n,F}$ defined in \eqv(1.2.1.12bis).

\begin{lemma}[Mixture coefficients]
    \TH(5.lem1)
Given $n\geq 3$  odd and $\ell\geq 2$, let $\AA_{n,\ell}$ be the subset of allowable $\ell$-compositions of $n$  of the form
\be
\begin{split}
&
n=n_1+\underbrace{2+\dots+2}+n_{\ell},
\\
& \hskip2.25truecm \ell-2 
\end{split}
\Eq(5.lem1.1)
\ee
where $n_1\geq 2$ is any even integer and $n_{\ell}\in\{1,3,5\}$. Then
\be
\G_{n,F}=\bigcup_{2\leq \ell\leq n}\left\{\g_n \mid (n_1,\dots,n_{\ell})\in\AA_n\right\}.
\Eq(5.lem1.2)
\ee
Thus, the elements of $\MM^{\circ}_{n,F}=\left\{m(\g_n) : \g_n\in\G_{n,F}\right\}$ have the following form. If $n_{\ell}=1$, 
\be
\begin{split}
m=&\biggl(\underbrace{\a^{(n_1)},\dots,\a^{(n_1)}}
,\underbrace{\frac{\a^{(n_1)}}{2},\frac{\a^{(n_1)}}{2}}
,\dots
,\underbrace{\frac{\a^{(n_1)}}{2^{\ell-2}},\frac{\a^{(n_1)}}{2^{\ell-2}}}
,\underbrace{\frac{\a^{(n_1)}}{2^{\ell-2}}},\underbrace{0,\dots,0}\biggr),
\\
& \hskip1.4truecm n_1  \hskip2.25truecm 2 \hskip2.75truecm 2  \hskip1truecm n_{\ell}=1 
\end{split}
\Eq(5.lem1.3)
\ee
and there are $M-n$ zero coordinates.
If $n_{\ell}=3$, the last three non-zero coordinates are identical and equal to
\be
\frac{\a^{(n_1)}}{2^{\ell-1}}.
\Eq(5.lem1.4)
\ee
If $n_{\ell}=5$, the last five non-zero coordinates are identical and equal to
\be
\frac{3}{8}\frac{\a^{(n_1)}}{2^{\ell-2}}.
\Eq(5.lem1.5)
\ee
\end{lemma}

Note that since $n_1\geq 2$ in Lemma \thv(5.lem1) is an arbitrary even integer, all odd integers $n\geq 3$ admit at least one $\ell$-composition satisfying the assumptions of the lemma. For example, for $n=3,5,7,9,11,13$, the following compositions satisfy the assumptions of the lemma and give rise to distinct vectors $m$ of the form \eqv(5.lem1.3). 
(An asterisk indicates solutions obtained in \cite{AGS85a}. Note that (4.7) in that paper is incorrect.)
\bea
&3&\hspace{-8pt}=2+1^*,
\nonumber
\\
&5&\hspace{-8pt}=4+1=2+3^*,
\nonumber
\\
&7&\hspace{-8pt}=6+1=4+3=2+5=2+2+3,
\nonumber
\\
&9&\hspace{-8pt}=8+1=6+3=4+5=4+2+3=2+2+2+3,
\nonumber
\\
&11&\hspace{-8pt}=10+1=8+3=6+5=4+2+5=4+2+2+3
\Eq(5.lem1.6bis)
\\
&&\hspace{-8pt}=2+2+2+5=2+2+2+2+3.
\nonumber
\\
&13&\hspace{-8pt}=12+1=10+3=8+5=8+2+3=6+2+5
\nonumber
\\
&&\hspace{-8pt}=6+2+2+3=4+2+2+5=4+2+2+2+3
\nonumber
\\
&&\hspace{-8pt}=2+2+2+2+5= 2+2+2+2+2+3.
\nonumber
\eea
Since $\a^{(1)}=1$, compositions terminating by $2+1$ and the same compositions but terminating by $3$ instead, define the same vector $m$.

It is interesting to note that regardless of the choice of $n$ and $k$, for each such $m$, the mixed memory $\xi^{(N)}(m)$
has an energy of the same order as that of the patterns themselves.  This is considered a drawback of the model in \cite{KH16}.

\begin{lemma}[Small energy gap]
    \TH(5.lem2)
Given $n\geq 3$ odd, let $\MM^{\circ}_{n,F}$ be defined as in \eqv(5.lem1.2)-\eqv(5.lem1.3). The following holds for all $m\in\MM^{\circ}_{n,F}$. If  $M\ll N$ then with $\P$-probability one, for all $1\leq \mu\leq M$
\be
-\frac{1}{2}
= \lim_{N\rightarrow\infty}E_{N,M}(\xi^{\mu})
< \lim_{N\rightarrow\infty} E_{N,M}(\xi^{(N)}(m))
\leq -\frac{e^{-2\varepsilon_{n_1}}}{\pi},
\Eq(5.lem2.1)
\ee
where $\frac{1}{12 k+1}\leq \varepsilon_{2k}\leq \frac{1}{6 k}$ for all $k\geq 1$.
\end{lemma}

\begin{proof}[Proof of Lemma \thv(5.lem1)]
Given $n\geq 3$ odd and $\ell\geq 2$, let $(n_1,\dots,n_{\ell})$ be an allowable $\ell$-composition of $n$. This means that  $n_r=2k_r$ for some  $k_r\geq 1$ and all $1\leq r\leq \ell-1$, and $n_{\ell}=2k_{\ell}+1$ for some $k_{\ell}\geq 0$. Let $\g_n$ be the associated vector of components \eqv(1.2.1.5). The first inequality of the system $\SS_{n,F}$  in \eqv(1.2.1.12bis) reads
\be
2\g^{(1)}>n_{2}\g^{(2)}+\dots+n_{\ell}\g^{(\ell)}.
\Eq(5.lem1.6)
\ee
Note that if $k_{\ell}= 0$, then $n_{\ell}=1$. Thus, $\a^{(n_{\ell})}=\a^{(1)}=1$ and
$
n_{\ell-1}\g^{(\ell-1)}+n_{\ell}\g^{(\ell)}
=
(n_{\ell-1}+1)\g^{(\ell-1)}
$,
so that the inequality \eqv(5.lem1.6) is the same as for the $(\ell-1)$-composition $(n_1,\dots,(n_{\ell-1}+1))$. We can therefore restrict the set of allowable  $\ell$-compositions to those for which  $n_{\ell}=2k_{\ell}+1$ for some $k_{\ell}\geq 1$. 

Let us first prove that if $(n_1,\dots,n_{\ell})\notin \AA_n$ then \eqv(5.lem1.6) is not satisfied. 
To do this, note that dividing both sides of \eqv(5.lem1.6) by $2\g^{(1)}$ yields
\be
1>Q,
\Eq(5.lem1.07)
\ee
where
\be
Q\equiv
\frac{n_{2}}{2}\frac{\g^{(2)}}{\g^{(1)}}+\frac{n_{3}}{2}\frac{\g^{(3)}}{\g^{(1)}}+\dots+\frac{n_{\ell-1}}{2}\frac{\g^{(\ell-1)}}{\g^{(1)}}+\frac{n_{\ell}}{2}\frac{\g^{(\ell)}}{\g^{(1)}}.
\ee
Using the definition \eqv(1.2.1.5) of $\g^{(r)}$ and the above expressions for $n_1,\dots,n_{\ell}$,
\be
Q=
k_2\a^{(2k_{2})}+k_3\a^{(2k_{2})}\a^{(2k_{3})}+\dots+k_{\ell-1}\prod_{l=2}^{\ell-1}\a^{(2k_{l})}+(k_{\ell}+\sfrac{1}{2})\prod_{l=2}^{\ell}\a^{(2k_{l})},
\ee
and by factoring, we arrive at
\be
\begin{split}
Q=
\a^{(2k_{2})}\bigl\{k_{2}+\a^{(2k_{3})}&\bigl\{k_{3}+\dots 
\\
+&\a^{(2k_{\ell-2})}\bigl\{k_{\ell-2}
+\a^{(2k_{\ell-1})}\bigl\{k_{\ell-1}+\a^{(2k_{\ell})}(k_{\ell}+\sfrac{1}{2})
\bigr\}\bigr\}\dots\bigr\}\bigr\}.
\end{split}
\Eq(5.lem1.7)
\ee
Recall that $k'\a^{(2k')}$ is increasing. Using this and the table \eqv(5.1) we  can easily check that
\be
\begin{cases}
\a^{(2k')}(k'+\sfrac{1}{2})<1 & \text{if $k'\in\{1,2\}$} \\
\a^{(2k')}(k'+\sfrac{1}{2})>1 &  \text{if $k'\geq 3$}
\end{cases},
\Eq(5.lem1.8)
\ee
that for all $\frac{3}{4}\leq \varepsilon<1$
\be
\begin{cases}
\a^{(2k')}(k'+\varepsilon)<1 & \text{if $k'=1$} \\
\a^{(2k')}(k'+\varepsilon)>1 &  \text{if $k'\geq 2$}
\end{cases},
\Eq(5.lem1.9)
\ee
and that
\be
\a^{(2k')}(k'+1)\geq 1 \text{ for all } k'\geq 1.
\Eq(5.lem1.10)
\ee
Assume first that $k_{\ell}>3$. Then, by \eqv(5.lem1.8), $\a^{(2k_{\ell})}(k_{\ell}+\sfrac{1}{2})>1$, so that 
\be
\a^{(2k_{\ell-1})}\left\{k_{\ell-1}+\a^{(2k_{\ell})}(k_{\ell}+\sfrac{1}{2})\right\}
>
\a^{(2k_{\ell-1})}(k_{\ell-1}+1).
\Eq(5.lem1.11)
\ee
Then, by \eqv(5.lem1.10), $\a^{(2k_{\ell-1})}(k_{\ell-1}+1)\geq 1$, so that
\be
\a^{(2k_{\ell-2})}\left\{k_{\ell-2}+\a^{(2k_{\ell-1})}\left\{k_{\ell-1}+\a^{(2k_{\ell})}(k_{\ell}+\sfrac{1}{2})\right\}\right\}
>
\a^{(2k_{\ell-2})}(k_{\ell-2}+1).
\Eq(5.lem1.12)
\ee
Iterating \eqv(5.lem1.11)-\eqv(5.lem1.12) gives $Q>1$, which contradicts \eqv(5.lem1.07). 
Assume next that $\ell\geq 3$, $k_{\ell}\in\{1,2\}$ and $k_{\ell-1}\geq 2$. Then $\a^{(2k_{\ell})}(k_{\ell}+\sfrac{1}{2})\in\{\frac{3}{4},\frac{15}{16}\}$. Thus, by \eqv(5.lem1.9), 
\be
\a^{(2k_{\ell-1})}\left\{k_{\ell-1}+\a^{(2k_{\ell})}(k_{\ell}+\sfrac{1}{2})\right\}
>
1,
\Eq(5.lem1.13)
\ee
and so,
\be
\a^{(2k_{\ell-2})}\left\{k_{\ell-2}+\a^{(2k_{\ell-1})}\left\{k_{\ell-1}+\a^{(2k_{\ell})}(k_{\ell}+\sfrac{1}{2})\right\}\right\}
>
\a^{(2k_{\ell-2})}(k_{\ell-2}+1).
\Eq(5.lem1.14)
\ee
But \eqv(5.lem1.14) is \eqv(5.lem1.12), and again, iterating, we get $Q>1$, contradicting \eqv(5.lem1.07). We now assume that $\ell\geq 4$, $k_r\geq 2$  for some $2\leq r\leq \ell-2$,  $k_s=1$ for all $r+1\leq s\leq \ell-1$ and $k_{\ell}\in\{1,2\}$. Then,
\be
Q=\a^{(2k_{2})}\left\{k_{2}+\a^{(2k_{3})}\left\{k_{3}+\dots +\a^{(2k_r)}(k_r+\varepsilon_r)\right\}\dots\right\}
\Eq(5.lem1.15)
\ee
where
\be
R_r
\equiv\a^{(2k_{r+1})}\left\{k_{r+1}+\dots +\a^{(2k_{\ell-2})}\left\{k_{\ell-2}
+\a^{(2k_{\ell-1})}\left\{k_{\ell-1}+\a^{(2k_{\ell})}(k_{\ell}+\sfrac{1}{2})
\right\}\right\}\dots\right\}.
\Eq(5.lem1.16)
\ee
Under our assumptions,
\be
\begin{split}
R_r
&=
\frac{1}{2}+\left(\frac{1}{2}\right)^2+\dots+\left(\frac{1}{2}\right)^{\ell-r-1}
+(k_{\ell}+\sfrac{1}{2})\a^{(2k_{\ell})}\left(\frac{1}{2}\right)^{\ell-r-1}
\\
&=1-\left(\frac{1}{2}\right)^{\ell-r-1}\left[1-(k_{\ell}+\sfrac{1}{2})\a^{(2k_{\ell})}\right],
\end{split}
\Eq(5.lem1.17)
\ee
and so, $\frac{3}{4}<\frac{7}{8}\leq R_r<1$. Therefore \eqv(5.lem1.9) holds, yielding 
$\a^{(2k_r)}(k_r+\varepsilon_r)>1$. Iterating then gives $Q>1$, which again contradicts \eqv(5.lem1.07).
This completes the proof of the claim that if $(n_1,\dots,n_{\ell})\notin \AA_n$ then \eqv(5.lem1.6) is not satisfied.
To prove the converse simply note that if $(n_1,\dots,n_{\ell})\in \AA_n$ then $Q=R_1$ and thus, by \eqv(5.lem1.17), $Q\leq 1$.
Therefore, the first of the $\ell-1$ inequalities of the system $\SS_{n,F}$ in \eqv(1.2.1.12bis)  is satisfied if and only if $(n_1,\dots,n_{\ell})\in \AA_n$. The remaining $\ell-2$ inequalities are treated in exactly the same way.
\end{proof}

\begin{proof}[Proof of Lemma \thv(5.lem2)] 
First note that 
\be
2E_{N,M}(\xi^{(N)}(m))
\leq 
-\sum_{1\leq \nu\leq n}\left(\frac{1}{N}\sum_{i=1}^N\xi^{\nu}_i\xi^{(N)}_i(m)\right)^2.
\Eq(5.lem2.2)
\ee
Since $m\in\MM^{\circ}_{n,F}$, then by Theorem \thv(1.theo2.mix) $\xi^{(N)}(m)$ is a mixed memory. Thus, it follows from Definition \thv(1.def1.2), \eqv(1.theo1.mix8) and  \eqv(5.lem1.2)-\eqv(5.lem1.3) that with $\P$-probability one
\be
\begin{split}
\lim_{N\rightarrow\infty}\sum_{1\leq \nu\leq n}\left(\frac{1}{N} \sum_{i=1}^N\xi^{\nu}_i\xi^{(N)}_i(m)\right)^2
&=n_1\left(\a^{(n_1)}\right)^2
+
\frac{2}{3}\left(\a^{(n_1)}\right)^2\left(1-\frac{1}{4^{\ell-2}}\right)
\\
+&
\left(\frac{\a^{(n_1)}}{2^{\ell-2}}\right)^2\left\{
\1_{n_{\ell}=1}+\frac{3}{4}\1_{n_{\ell}=3}+\frac{45}{64}\1_{n_{\ell}=5}
\right\},
\end{split}
\Eq(5.lem2.3)
\ee
Next, the right-hand side of \eqv(5.lem2.3) is bounded below by $n_1\left(\a^{(n_1)}\right)^2$ and so, setting $n_1=2k$ and using \eqv(5.lem0.1) we have, for all $k\geq 2$
\be
n_1\left(\a^{(n_1)}\right)^2=\frac{2}{\pi}e^{-2\varepsilon_{n_1}},
\Eq(5.lem2.4)
\ee
while for $k=1$, \eqv(5.lem2.3) is bounded below by $\frac{2}{3}$. Putting these bounds together, the right-hand side of \eqv(5.lem2.3) is bounded below by
$
\frac{2}{3}\1_{\{k=1\}} +\frac{2}{\pi}e^{-\varepsilon_{n_1}}\1_{\{k\geq 2\}}
\geq \frac{2}{3}\1_{\{k=1\}}+\inf_{\{k\geq 2\}} \frac{2}{\pi}e^{-2\varepsilon_{n_1}}
\geq \frac{2}{\pi}e^{-2\varepsilon_{n_1}}
$
for all $k\geq 1$.
Finally, it is well known that  if $M\ll N$, with $\P$-probability one $\lim_{N\rightarrow\infty}2 E_{N,M}(\xi^{\mu})=-1$ for all $1\leq \mu\leq M$, while all other $\s\in\S_N$ have strictly higher energies \cite{BG98}, \cite{TaHop98}. Combined with \eqv(5.lem2.2) and our bound on \eqv(5.lem2.3), this yields the claim of the lemma.
\end{proof}

\subsection{The dense Hopfield model} 
    \TH(S5.2)    
    
We now consider the case $F(x)=\frac{1}{p}x^p$ with $p\geq 3$. 

\begin{lemma}[Mixture coefficients]
    \TH(5.lem3)
For all $p\geq 3$, $\MM_{n,F}=\MM^{all}_n$.
\end{lemma}  

\begin{proof}[Proof of Lemma \thv(5.lem3)] We must establish that if $p\geq 3$ then for all $n\geq 3$ odd, $\G^{all}_n=\G_{n,F}$ (see \eqv(1.2.1.6), \eqv(1.2.1.12)). In the following we use the notations of the proof of Lemma \thv(5.lem1). The first inequality of the system $\SS_{n,F}$  in \eqv(1.2.1.12bis) reads
\be
2\left(\g^{(1)}\right)^{p-1}>n_{2}\left(\g^{(2)}\right)^{p-1}+\dots+n_{\ell}\left(\g^{(\ell)}\right)^{p-1}.
\Eq(5.lem3.1)
\ee
Equivalently,
\be
1>
\sum_{r=2}^{\ell-1}k_{r}\left(\prod_{s=2}^{r}\a^{(2k_{s})}\right)^{p-1}
+(k_{\ell}+\sfrac{1}{2})\left(\prod_{s=2}^{\ell}\a^{(2k_{s})}\right)^{p-1}
\equiv P+Q.
\Eq(5.lem3.2)
\ee
For $k\geq 2$, recall the estimate on $\a^{(2k)}$ from \eqv(5.lem0.1) and note that for $k=1$, 
$\a^{(2k)}\big|_{k=1}=\frac{1}{2}< \frac{e^{-\varepsilon_{2k}}}{\sqrt{\pi k}}\big|_{k=1}$. Therefore, we have the bound, valid for all $k\geq 1$
\be
\a^{(2k)}\leq \frac{e^{-\varepsilon_{2k}}}{\sqrt{\pi k}}\leq \frac{1}{\sqrt{\pi k}}.
\Eq(5.lem3.3)
\ee
With this, with have, for all $k_s\geq 1$, $\,2\leq s\leq \ell-1$
\be
\begin{split}
P&\leq 
\sum_{r=2}^{\ell-1}
\frac{1}
{
\pi^{\frac{p-1}{2}(r-1)}
\left(\prod_{s=2}^{r-1}k_s\right)^{\frac{p-1}{2}}k_r^{\frac{p-3}{2}}
}, 
\\
Q&\leq 
\frac{1}
{
\pi^{\frac{p-1}{2}(\ell-1)}
\left(\prod_{s=2}^{\ell-1}k_s\right)^{\frac{p-1}{2}}
}
\left(\sfrac{3}{2}k_\ell^{-\frac{p-3}{2}}\1_{\{k_{\ell}\geq 1\}}+\sfrac{1}{2}\1_{\{k_{\ell}=0\}}\right).
\end{split}
\Eq(5.lem3.4)
\ee
We now can see that for $p\geq 3$, the right-hand sides of $P$ and $Q$ in \eqv(5.lem3.4) are decreasing both in $p$ and in each $k_s$. Therefore,
\be
\begin{split}
P\leq 
\sum_{r=2}^{\ell-1}
\left(\frac{1}{\pi}\right)^{(r-1)},\quad
Q\leq 
\frac{3}{2}\left(\frac{1}{\pi}\right)^{(\ell-1)},
\end{split}
\Eq(5.lem3.5)
\ee
and summing the geometric progression
\be
P+Q
\leq 
\frac{1}{\pi-1}\left[1-\left(\frac{1}{\pi}\right)^{(\ell-2)}\right]
+
\frac{3}{2}\left(\frac{1}{\pi}\right)^{(\ell-1)}
\leq
\frac{1}{\pi-1}+\frac{3}{2\pi}
<0.95<1.
\ee
From this we conclude that \eqv(5.lem3.2) is satisfied for all $\g_n\in\G^{all}_n$ and all $n$ odd. In the same way, one checks that the remaining $\ell -2$ inequalities of the system $\SS_{n,F}$ are satisfied. This yields the claim of the lemma.
\end{proof}

We will not state the counterpart of Lemma \thv(5.lem2) for the dense model, but simply observe that the contribution
to $E_{N,M}(\xi^{(N)}(m))$ of the first $n$ patterns increases to zero with $p$. More precisely, for all
$
m\in\MM^{\circ}_{n,F}=\left\{m(\g_n) : \g_n\in\G^{all}_{n}\right\}
$,
by Theorem \thv(1.theo2.mix),  proceeding as in the proof of Lemma \thv(5.lem2), with $\P$-probability one 
\be
\begin{split}
\lim_{N\rightarrow\infty}\sum_{1\leq \nu\leq n}\left(\frac{1}{N} \sum_{i=1}^N\xi^{\nu}_i\xi^{(N)}_i(m)\right)^p
=\,&
\sum_{r=1}^{\ell}n_r\left(\prod_{s=1}^r\a^{(n_s)}\right)^p,
\end{split}
\Eq(5.lem3.4')
\ee
and this quantity is bounded from above and below by constants times
\be
\sum_{r=1}^{\ell}
\left(\frac{2}{\pi}\right)^{\frac{pr}{2}}\frac{1}{
\left(\prod_{s=1}^{r-1}n_s\right)^{\frac{p}{2}}n_r^{\frac{p-2}{2}}
}.
\ee
Therefore, the contribution to $E_{N,M}(\xi^{(N)}(m))$ of the first $n$ patterns as $N\rightarrow\infty$ decreases 
as $p$ decreases. This confirms the observation made in \cite{KH16} that, asymptotically as $N\rightarrow\infty$, the energy gap,
\be
\left|\sup_{1\leq \mu\leq M}E_{N,M}(\xi^{\mu})-\inf_{m} E_{N,M}(\xi^{(N)}(m))\right|,
\ee
increases as $p$ increases, provided that $M$ is sufficiently small for the contribution to the energy of the remaining patterns decays to zero.


\def\cprime{$'$}


\end{document}